\documentclass[a4paper]{amsart}
\usepackage[centertags]{amsmath}
\usepackage[breaklinks=true,colorlinks]{hyperref}
\usepackage[textwidth=15cm,hcentering]{geometry}
\usepackage{amsfonts}
\usepackage{amssymb}
\usepackage{amsthm}
\usepackage{newlfont}
\usepackage{amscd}
\usepackage{amsmath,amscd}
\usepackage{graphicx}
\usepackage[all]{xy}
\usepackage{verbatim}

\usepackage{color}

\usepackage{eucal}

\makeatletter
\def\@tocline#1#2#3#4#5#6#7{\relax
\ifnum #1>\c@tocdepth 
  \else
    \par \addpenalty\@secpenalty\addvspace{#2}%
\begingroup \hyphenpenalty\@M
    \@ifempty{#4}{%
      \@tempdima\csname r@tocindent\number#1\endcsname\relax
 }{%
   \@tempdima#4\relax
 }%
 \parindent\z@ \leftskip#3\relax \advance\leftskip\@tempdima\relax
 \rightskip\@pnumwidth plus4em \parfillskip-\@pnumwidth
 #5\leavevmode\hskip-\@tempdima #6\nobreak\relax
 \ifnum#1<0\hfill\else\dotfill\fi\hbox to\@pnumwidth{\@tocpagenum{#7}}\par
 \nobreak
 \endgroup
  \fi}
\makeatother

\makeatletter
\ifcsname phantomsection\endcsname
    \newcommand*{\qrr@gobblenexttocentry}[5]{}
\else
    \newcommand*{\qrr@gobblenexttocentry}[4]{}
\fi
\newcommand*{\addsubsection}{%
    \addtocontents{toc}{\protect\qrr@gobblenexttocentry}%
    \subsection}
\makeatother

\usepackage{lipsum}

\newcommand{\adj}{\rightleftarrows}

\newcommand{\ZZ}{\mathbb{Z}}

\newcommand{\NN}{\mathbb{N}}
\newcommand{\cA}{\mathcal{A}}
\newcommand{\cB}{\mathcal{B}}
\newcommand{\cC}{\mathcal{C}}
\newcommand{\cD}{\mathcal{D}}
\newcommand{\cE}{\mathcal{E}}
\newcommand{\cF}{\mathcal{F}}
\newcommand{\cG}{\mathcal{G}}
\newcommand{\cH}{\mathcal{H}}
\newcommand{\cI}{\mathcal{I}}
\newcommand{\cJ}{\mathcal{J}}
\newcommand{\cK}{\mathcal{K}}
\newcommand{\cL}{\mathcal{L}}
\newcommand{\cM}{\mathcal{M}}
\newcommand{\cN}{\mathcal{N}}

\newcommand{\cP}{\mathcal{P}}

\newcommand{\cR}{\mathcal{R}}
\newcommand{\cS}{\mathcal{S}}
\newcommand{\cT}{\mathcal{T}}
\newcommand{\cU}{\mathcal{U}}

\newcommand{\cW}{\mathcal{W}}
\newcommand{\cX}{\mathcal{X}}

\newcommand{\fC}{\mathfrak{C}}
\newcommand{\B}{\mathrm{B}}
\newcommand{\E}{\mathrm{E}}
\newcommand{\K}{\mathrm{K}}
\renewcommand{\L}{\mathrm{L}}

\def\alp{{\alpha}}
\def\bet{{\beta}}
\def\gam{{\gamma}}

\def\sig{{\sigma}}
\def\vphi{{\varphi}}
\def\om{{\omega}}

\def\Gam{{\Gamma}}
\def\Del{{\Delta}}

\def\Lam{{\Lambda}}

\def\vphi{{\varphi}}

\newtheorem{thm}{Theorem}[subsection]

\newtheorem{cor}[thm]{Corollary}

\newtheorem{assume}[thm]{Assumption}

\newtheorem{lem}[thm]{Lemma}
\newtheorem{prop}[thm]{Proposition}

\theoremstyle{definition}
\newtheorem{define}[thm]{Definition}
\newtheorem{conv}[thm]{Convention}

\newtheorem{example}{Example}

\theoremstyle{remark}
\newtheorem{rem}[thm]{Remark}

\def\Csep{\mathtt{SC^*}}

\DeclareMathOperator{\Shp}{Sh}

\DeclareMathOperator{\Fun}{Fun}

\DeclareMathOperator{\Set}{\cS et}
\DeclareMathOperator{\Cat}{\cC at}

\DeclareMathOperator{\Pro}{Pro}
\DeclareMathOperator{\Ind}{Ind}
\DeclareMathOperator{\Map}{Map}
\DeclareMathOperator{\Ho}{Ho}
\DeclareMathOperator{\Hom}{Hom}
\DeclareMathOperator{\prehocolim}{hocolim}
\DeclareMathOperator{\precolim}{colim}
\DeclareMathOperator{\preholim}{holim}
\DeclareMathOperator{\Lw}{Lw}
\DeclareMathOperator{\Sp}{Sp}
\DeclareMathOperator{\cocell}{cocell}
\def\Ne{\mathrm{N}}
\DeclareMathOperator{\rezk}{Rezk}
\DeclareMathOperator{\df}{def}

\DeclareMathOperator{\fib}{fib}

\DeclareMathOperator{\Ob}{Ob}
\DeclareMathOperator{\Id}{Id}

\DeclareMathOperator{\PS}{\cP\cS hv}
\DeclareMathOperator{\Sh}{\cS hv}
\DeclareMathOperator{\Mor}{Mor}
\DeclareMathOperator{\Dom}{Dom}
\DeclareMathOperator{\Cod}{Cod}

\DeclareMathOperator{\R}{R}

\DeclareMathOperator{\cof}{cof}
\DeclareMathOperator{\spe}{sp}

\DeclareMathOperator{\sm}{sm}

\DeclareMathOperator{\cosk}{cosk}
\DeclareMathOperator{\cofin}{\tau}

\DeclareMathOperator{\nil}{nil}

\DeclareMathOperator{\Ab}{Ab}
\DeclareMathOperator{\Ch}{Ch}
\DeclareMathOperator{\Quick}{Quick}
\DeclareMathOperator{\Morel}{Morel}

\newcommand{\Fib}{\mathcal{F}ib}
\newcommand{\Cof}{\mathcal{C}of}
\newcommand{\Triv}{\mathcal{T}riv}

\newcommand{\op}{\mathrm{op}}

\def\colim{\mathop{\precolim}}
\def\hocolim{\mathop{\prehocolim}}
\def\holim{\mathop{\preholim}}

\def\lrar{\longrightarrow}
\def\llar{\longleftarrow}
\def\hrar{\hookrightarrow}

\def\x{\stackrel}

\def \mcal{\mathcal}
\def \ovl{\overline}
\def \what{\widehat}

\def \uline{\underline}

\title{Pro-categories in homotopy theory}

\author{
Ilan Barnea \and Yonatan Harpaz \and Geoffroy Horel}

\address{Mathematisches Institut\\
Einsteinstrasse 62\\
D-48149 M\"unster\\
Deutschland}
\email{ilanbarnea770@gmail.com}

\address{
D\'epartement de math\'ematiques et applications \\
\'Ecole normale sup\'erieure \\
45 rue d'Ulm \\
75005, Paris \\
France
}
\email{harpazy@gmail.com}

\address{Mathematisches Institut\\
Einsteinstrasse 62\\
D-48149 M\"unster\\
Deutschland}
\email{geoffroy.horel@gmail.com}

\thanks{The first and third authors were supported by Michael Weiss's Humboldt professor grant. The second author was supported by the Fondation Sciences Math\'ematiques de Paris.}


\begin{document}

\begin{abstract}
The goal of this paper is to prove an equivalence between the model categorical approach to pro-categories, as studied by Isaksen, Schlank and the first author, and the $\infty$-categorical approach, as developed by Lurie. Three applications of our main result are described. In the first application we use (a dual version of) our main result to give sufficient conditions on an $\omega$-combinatorial model category, which insure that
its underlying $\infty$-category is $\omega$-presentable. In the second application we consider the pro-category of simplicial \'etale sheaves and use it to show that the topological realization of any Grothendieck topos coincides with the shape of the hyper-completion of the associated $\infty$-topos. In the third application we show that several model categories arising in profinite homotopy theory are indeed models for the $\infty$-category of profinite spaces. As a byproduct we obtain new Quillen equivalences between these models, and also obtain an example which settles negatively a question raised by G. Raptis.
\end{abstract}

\maketitle
\tableofcontents

\section*{Introduction}

Following the appearance of model categories in Quillen's seminal paper~\cite{Qui67}, the framework of homotopy theory was mostly based on the language of model categories and their variants (relative categories, categories of fibrant objects, Waldhausen categories, etc.). This framework has proven very successful at formalizing and manipulating constructions such as homotopy limits and colimits as well as more general derived functors, such as derived mapping spaces. There are well-known model category structures on the classical objects of study of homotopy theory like spaces, spectra and chain complexes.

When working in this setting one often requires the extension of classical category theory constructions to the world of model categories. One approach to this problem is to perform the construction on the underlying ordinary categories, and then attempt to put a model structure on the result that is inherited in some natural way from the model structures of the inputs. There are two problems with this approach. The first problem is that model categories are somewhat rigid and it is often hard, if not impossible, to put a model structure on the resulting object. The second problem is that model categories themselves have a non-trivial homotopy theory, as one usually considers Quillen equivalences as ``weak equivalences'' of model categories. It is then a priori unclear whether the result of this approach is invariant under Quillen equivalences, nor whether it yields the ``correct'' analogue of the construction from a homotopical point of view.

Let us look at a very simple example. For $\cM$ a model category and $\cC$ a small ordinary category, one can form the category of functors $\cM^{\cC}$. There is a natural choice for the weak equivalences on $\cM^{\cC}$, namely the objectwise weak equivalences. A model structure with these weak equivalences is known to exist when $\cM$ or $\cC$ satisfy suitable conditions, but is not known to exist in general. Furthermore, even when we can endow $\cM^{\cC}$ with such a model structure, it is not clear whether it encodes the desired notion from a homotopical point of view. In particular, one would like $\cM^{\cC}$ to capture a suitable notion of \textbf{homotopy coherent diagrams} in $\cM$. Writing down exactly what this means when $\cM$ is a model category is itself not an easy task.

These issues can be resolved by the introduction of \textbf{$\infty$-categories}. Given two $\infty$-categories $\cC,\cD$, a notion of a \textbf{functor category} $\Fun(\cC,\cD)$ arises naturally, and takes care of all the delicate issues surrounding homotopy coherence in a clean and conceptual way. On the other hand, any model category $\cM$, and in fact any category with a notion of weak equivalences, can be localized to form an $\infty$-category $\cM_\infty$. The $\infty$-category $\cM_\infty$ can be characterized by the following universal property: for every $\infty$-category $\cD$, the natural map
\[ \Fun(\cM_\infty,\cD) \lrar \Fun(\cM,\cD) \]
is fully faithful, and its essential image is spanned by those functors $\cM \lrar \cD$ which send weak equivalences in $\cM$ to equivalences. The $\infty$-category $\cM_\infty$ is called the \textbf{$\infty$-localization} of $\cM$, and one says that $\cM$ is a \textbf{model} for $\cM_\infty$.

One may then formalize what it means for a model structure on $\cM^{\cC}$ to have the ``correct type'': one wants the $\infty$-category modelled by $\cM^{\cC}$ to coincides with the $\infty$-category $\Fun(\cC,\cM_{\infty})$. When $\cM$ is a combinatorial model category it is known that $\cM^{\cC}$ both exists and has the correct type in the sense above (see~\cite[Proposition 4.2.4.4]{Lur09} for the simplicial case). For general model categories it is not known that $\cM^{\cC}$ has the correct type, even in cases when it is known to exist.

Relying on the theory of $\infty$-categories for its theoretical advantages, it is often still desirable to use model categories, as they can be simpler to work with and more amenable to concrete computations. It thus becomes necessary to be able to \textbf{compare} model categorical constructions to their $\infty$-categorical counterparts.

The aim of this paper is to achieve this comparison for the construction of \textbf{pro-categories}. In classical category theory, one can form the pro-category $\Pro(\cC)$ of a category $\cC$, which is the \textbf{free completion} of $\cC$ under \textbf{cofiltered limits}. This can be formalized in term of a suitable universal property: given a category $\cD$ which admits cofiltered limits, the category of functors $\Pro(\cC) \lrar \cD$ which preserve cofiltered limits is naturally equivalent, via restriction, with the category of all functors $\cC \lrar \cD$. It is often natural to consider the case where $\cC$ already possesses \textbf{finite limits}. In this case $\Pro(\cC)$ admits \textbf{all small limits}, and enjoys the following universal property: (*) if $\cD$ is any category which admits small limits, then the category of functors $\Pro(\cC) \lrar \cD$ which preserve limits can be identified with the category of functors $\cC \lrar \cD$ which preserve finite limits.

If $\cC$ is an $\infty$-category, one can define the pro-category of $\cC$ using a similar universal construction. This was done in~\cite{Lur09} for $\cC$ a small $\infty$-category and in~\cite{Lur11b} for $\cC$ an accessible $\infty$-category with finite limits. On the other hand, when $\cC$ is a model category, one may attempt to construct a model structure on $\Pro(\cC)$ which is naturally inherited from that of $\cC$. This was indeed established in~\cite{EH76} when $\cC$ satisfies certain conditions (``Condition N'') and later in~\cite{Isa04} when $\cC$ is a proper model category. The first author and Schlank  \cite{BS1} observed that a much simpler structure on $\cC$ is enough to construct, under suitable hypothesis, a model structure on $\Pro(\cC)$. Recall that

\begin{define}\label{d:wfc-0}
A \textbf{weak fibration category} is a category $\cC$ equipped with two subcategories
\[\Fib, \cW \subseteq \cC\]
containing all the isomorphisms, such that the following conditions are satisfied:
\begin{enumerate}
\item $\cC$ has all finite limits.
\item $\cW$ has the 2-out-of-3 property.
\item For every pullback square
\[\xymatrix{
X \ar[r]\ar^{g}[d] & Y \ar^{f}[d] \\
Z \ar[r] & W \\
}\]
with $f \in \Fib$ (resp. $f \in \Fib \cap \cW$) we have $g \in \Fib$ (resp. $g \in \Fib \cap \cW$).
\item Every morphism $f: X \lrar Y$ can be factored as $X \x{f'}{\lrar} Z \x{f''}{\lrar} Y$ where $f' \in \cW$ and $f'' \in \Fib$.
\end{enumerate}
\end{define}

The structure of weak fibration category was first studied by Anderson \cite{An78}, who referred to it as a right homotopy structure, and is closely related to the notion of a category of fibrant objects due to Brown \cite{Bro73}. Other variants of this notion were studied by Cisinski \cite{Cis10b}, Radlescu-Banu \cite{RB06}, Szumiło \cite{Szu14} and others. We note that the full subcategory of any weak fibration category spanned by the fibrant objects is a category of fibrant objects in the sense of \cite{Bro73}, and the inclusion functor induces an equivalence of $\infty$-categories after $\infty$-localization. This last statement, which is somewhat subtle when one does not assume the factorizations of Definition~\ref{d:wfc-0} (4) to be functorial, appears as Proposition \ref{p:equivalence C_f C} and is due to Cisinski.

The main result of~\cite{BS1} is the construction of a model structure on the pro-category of a weak fibration category $\cC$, under suitable hypothesis. The setting of weak fibration categories is not only more flexible than that of model categories, but it is also conceptually more natural: as we will show in \S\ref{s:homotopical prelim}, the underlying $\infty$-category of a weak fibration category has finite limits, while the underlying $\infty$-category of a model category has all limits. It is hence the setting in which the $\infty$-categorical analogue of universal property (*) comes into play, and arguably the most natural context in which one wishes to construct pro-categories. In \S\ref{s:pro-model} we give a general definition of what it means for a model structure on $\Pro(\cC)$ to be \textbf{induced} from a weak fibration structure on $\cC$. Our approach unifies the constructions of~\cite{EH76}, ~\cite{Isa04} and~\cite{BS1}, and also answers a question posed by Edwards and Hastings in~\cite{EH76} (see Remark \ref{r:EH}).

Having constructed a model structure on $\Pro(\cC)$, a most natural and urgent question is the following: is $\Pro(\cC)$ a model for the $\infty$-category $\Pro(\cC_\infty)$?
Our main goal in this paper is to give a positive answer to this question:

\begin{thm}[see Theorem~\ref{t:main}]\label{t:main0}
Assume that the induced model structure on $\Pro(\cC)$ exists. Then the natural map
\[ \cF: \Pro(\cC)_\infty \lrar \Pro(\cC_\infty) \]
is an equivalence of $\infty$-categories.
\end{thm}

We give three applications of our general comparison theorem. The first application concerns combinatorial model categories. If $\cM$ is a combinatorial model category, then by~\cite[Proposition A.3.7.6.]{Lur09} and the main result of~\cite{Dug01}, its underlying $\infty$-category $\cM_\infty$ is presentable. For many purposes, however, it is useful to know that $\cM_\infty$ is not only presentable, but also \textbf{$\omega$-presentable}, i.e., equivalent to the ind-category of its subcategory of $\omega$-compact objects. This fact is often not so easy to prove, even if we know that $\cM$ is $\omega$-combinatorial (in the sense that its underlying category is $\omega$-presentable and $\cM$ admits a sets of generating cofibrations and trivial cofibrations whose domains and codomains are $\omega$-compact). Using (a dual version of) our main result we are able to give a simple sufficient condition on an $\omega$-combinatorial model category, which insures that its underlying $\infty$-category is $\omega$-presentable. Namely, we show
\begin{prop}[see Proposition~\ref{p:omega}]
Let $\cM$ be an $\omega$-combinatorial model category and let $\cM_0 \subseteq \cM$ be the full subcategory spanned by $\omega$-compact objects. Suppose that every morphism in $\cM_0$ can be factored inside $\cM_0$ into a cofibration followed by a weak equivalence. Then $(\cM_0)_\infty$ is essentially small, admits finite colimits and
$$\Ind((\cM_0)_\infty)\simeq\cM_\infty.$$
In particular, $\cM_\infty$ is $\omega$-presentable, and every $\omega$-compact object in $\cM_\infty$ is a retract of an object in $\cM_0$.
\end{prop}

Our second application involves the theory of \textbf{shapes of topoi}. In~\cite{AM69}, Artin and Mazur defined the \textbf{\'etale homotopy type} of an algebraic variety. This is a pro-object in the homotopy category of spaces, which depends only on the \'etale site of $X$. Their construction is based on the construction of the \textbf{shape} of a topological space $X$, which is a similar type of pro-object constructed from the site of open subsets of $X$. More generally, Artin and Mazur's construction applies to any \textbf{locally connected} site.

In~\cite{BS1} the first author and Schlank used their model structure to define what they call the \textbf{topological realization} of a Grothendieck topos. Their construction works for any Grothendieck topos and \textbf{refines} the previous constructions form a pro-object in the homotopy category of spaces to a pro-object in the category of simplicial sets. On the $\infty$-categorical side, Lurie constructed in~\cite{Lur09} an $\infty$-categorical analogue of shape theory and defined the shape assigned to any $\infty$-topos as a pro-object in the $\infty$-category $\cS_\infty$ of spaces. A similar type of construction also appears in~\cite{ToVe}. One then faces the same type of pressing question: Is the topological realization constructed in~\cite{BS1} using model categories equivalent to the one defined in~\cite{Lur09} using the language of $\infty$-categories? In \S\ref{s:topos} we give a positive answer to this question:

\begin{thm}[see Theorem~\ref{t:main shape}]
For any Grothendieck site $\cC$ there is a weak equivalence
\[|\cC|\simeq \Shp(\what{\Sh}_\infty(\cC)) \]
of pro-spaces, where $|\cC|$ is the topological realization constructed in~\cite{BS1} and $\Shp(\what{\Sh}_\infty(\cC)) \in \Pro(\cS_\infty)$ is the shape of the hyper-completed $\infty$-topos $\what{\Sh}_\infty(\cC)$ constructed in~\cite{Lur09}.
\end{thm}

Combining the above theorem with~\cite[Theorem 1.15]{BS1} we obtain:
\begin{cor}
Let $X$ be a locally Noetherian scheme, and let $X_{\acute{e}t}$ be its \'etale site. Then the image of  $\Shp(\what{\Sh}_\infty(X_{\acute{e}t}))$ in $\Pro(\Ho(\cS_\infty))$ coincides with the \'etale homotopy type of $X$.
\end{cor}

Our third application is to the study of \textbf{profinite homotopy theory}. Let $\cS$ be the category of simplicial sets, equipped with the Kan-Quillen model structure. The existence of the induced model structure on $\Pro(\cS)$ (in the sense above) follows from the work of~\cite{EH76} (as well as~\cite{Isa04} and~\cite{BS1} in fact). In~\cite{Isa05}, Isaksen showed that for any set $K$ of fibrant object of $\cS$, one can form the maximal left Bousfield localization $L_K \Pro(\cS)$ of $\Pro(\cS)$ for which all the objects in $K$ are local. The weak equivalences in $L_K\Pro(\cS)$ are the maps $X \lrar Y$ in $\Pro(\cS)$ such that the map
\[\Map^h_{\Pro(\cS)}(Y,A) \lrar \Map^h_{\Pro(\cS)}(X,A)\]
is a weak equivalence for every $A$ in $K$. When choosing a suitable candidate $K = K^{\pi}$, the model category $L_{K^{\pi}}\Pro(\cS)$ can be used as a theoretical setup for \textbf{profinite homotopy theory}.

On the other hand, one may define what profinite homotopy theory should be from an $\infty$-categorical point of view. Recall that a space $X$ is called \textbf{$\pi$-finite} if it has finitely many connected components, and finitely many non-trivial homotopy groups which are all finite. The collection of $\pi$-finite spaces can be organized into an $\infty$-category $\cS^{\pi}_\infty$, and the associated pro-category $\Pro(\cS^{\pi}_{\infty})$ can equally be considered as the natural realm of profinite homotopy theory. One is then yet again faced with the salient question: is $L_{K^{\pi}}\Pro(\cS)$ a model for the $\infty$-category $\Pro(\cS^{\pi}_{\infty})$? In \S\ref{ss:example1} we give a positive answer to this question:

\begin{thm}[see Corollary~\ref{c:isaksen-finite}]\label{t:isaksen-0}
The underlying $\infty$-category of $L_{K^{\pi}}\Pro(\cS)$ is naturally equivalent to the $\infty$-category $\Pro(\cS^{\pi}_{\infty})$ of profinite spaces.
\end{thm}

A similar approach was undertaken for the study of \textbf{$p$-profinite homotopy theory}, when $p$ is a prime number. Choosing a suitable candidate $K=K^p$, Isaksen's approach yields a model structure $L_{K^p}\Pro(\cS)$ which can be used as a setup for $p$-profinite homotopy theory. On the other hand, one may define $p$-profinite homotopy theory from an $\infty$-categorical point of view. Recall that a space $X$ is called \textbf{$p$-finite} if it has finitely many connected components and finitely many non-trivial homotopy groups which are all finite $p$-groups. The collection of $p$-finite spaces can be organized into an $\infty$-category $\cS^{p}_\infty$, and the associated pro-category $\Pro(\cS^{p}_{\infty})$ can be considered as a natural realm of $p$-profinite homotopy theory (see~\cite{Lur11b} for a comprehensive treatment). Our results allow again to obtain the desired comparison:
\begin{thm}[see Corollary~\ref{c:isaksen-p finite}]
The underlying $\infty$-category of $L_{K^{p}}\Pro(\cS)$ is naturally equivalent to the $\infty$-category $\Pro(\cS^{p}_{\infty})$ of $p$-profinite spaces.
\end{thm}

Isaksen's approach is not the only model categorical approach to profinite and $p$-profinite homotopy theory. In~\cite{Qui11} Quick constructs a model structure on the category $\what{\cS}$ of \textbf{simplicial profinite sets} and uses it as a setting to perform profinite homotopy theory. His construction is based on a previous construction of Morel (\cite{Mor96}), which endowed the category of simplicial profinite sets with a model structure aimed at studying $p$-profinite homotopy theory. In \S\ref{ss:quick} we show that Quick and Morel's constructions are Quillen equivalent to the corresponding Bousfield localizations studied by Isaksen.

\begin{thm}[see Theorem~\ref{t:isaksen-is-quick} and Theorem~\ref{t:isaksen-is-morel}]\label{t:isaksen-quick-0}
There are  \textbf{Quillen equivalences}
\[ \Psi_{K^\pi}: L_{K^\pi}\Pro(\cS) \adj \what{\cS}_{\Quick} : \Phi_{K^\pi} \]
\[ \Psi_{K^p}: L_{K^p}\Pro(\cS) \adj \what{\cS}_{\Morel} : \Phi_{K^p}. \]
\end{thm}

These Quillen equivalences appear to be new. A weaker form of the second equivalence was proved by Isaksen in~\cite[Theorem 8.7]{Isa05}, by constructing a length two zig-zag of adjunctions between $L_{K^p}\Pro(\cS)$ and $\what{\cS}_{\Morel}$ where the middle term of this zig-zag is not a model category but only a relative category.

A key point in the construction of these Quillen equivalences is a notion which we call \textbf{$\tau$-finite} simplicial sets. A simplicial set is called $\tau$-finite if it is levelwise finite and $n$-coskeletal for some $n \geq 0$. We denote by $\cS_{\tau} \subseteq \cS$ the full subcategory spanned by $\tau$-finite simplicial sets. The category $\cS_{\tau}$ is clearly essentially small, and we show that its pro-category is equivalent to the category of simplicial profinite sets. These results also enable us to show (see Remark \ref{r:Raptis}) that the opposite of Quick’s model structure is an example of an $\omega$-combinatorial model category, whose class of weak equivalences is not $\omega$-accessible (as a full subcategory of the class of all morphisms). To the knowledge of the authors such an example has not yet appeared in the literature. It settles negatively a question raised by G. Raptis in personal communication with the first author.

Finally, let us briefly mention two additional applications which will appear in forthcoming papers. In a joint work with Michael Joachim and Snigdhayan Mahanta (see~\cite{BJM}) the first author constructs a model structure on the category $\Pro(\Csep)$, where $\Csep$ is the category of separable $C^*$-algebras, and uses it to define a bivariant $\K$-theory category for the objects in $\Pro(\Csep)$. Theorem~\ref{t:main0} is then applied to show that this bivariant $\K$-theory category indeed extends the known bivariant $\K$-theory category constructed by Kasparov. In~\cite{Hor} the third author relies on Theorem~\ref{t:isaksen-0} and Theorem~\ref{t:isaksen-quick-0} in order to prove that the group of homotopy automorphisms of the profinite completion of the little $2$-disk operad is isomorphic to the profinite Grothendieck-Teichm\"uller group.

\addsubsection*{Overview of the paper}

In \S\ref{ss:large} we formulate the set theoretical framework and terminology used throughout the paper. Such framework is required in order to work fluently with both large and small $\infty$-categories. The reader who is familiar with these issues can very well skip this section and refer back to it as needed.

\S\ref{s:homotopical prelim} is dedicated to recalling and sometimes proving various useful constructions and results in higher category theory. In particular, we will recall the notions of $\infty$-categories, relative categories, categories of fibrant objects, weak fibration categories and model categories. Along the way we will fill what seems to be a gap in the literature and prove that the $\infty$-category associated to any category of fibrant objects, or a weak fibration category, has finite limits, and that the $\infty$-category associated to any model category has all limits and colimits.

In \S\ref{s:pro} we recall a few facts about pro-categories, both in the classical categorical case and in the $\infty$-categorical case. In particular, we construct and establish the universal property of the pro-category of a general locally small $\infty$-category, a construction that seems to be missing from the literature.

In \S\ref{s:pro-model} we explain what we mean by a model structure on $\Pro(\cC)$ to be induced from a weak fibration structure on $\cC$. We establish a few useful properties of such a model structure and give various sufficient conditions for its existence (based on the work of~\cite{EH76},~\cite{Isa04} and~\cite{BS1}).

In \S\ref{s:main} we will conduct our investigation of the underlying $\infty$-category of $\Pro(\cC)$ where $\cC$ is a weak fibration category such that the induced model category on $\Pro(\cC)$ exists. Our main result, which is proved in \S\ref{ss:proof-of-main}, is that the underlying $\infty$-category of $\Pro(\cC)$ is naturally equivalent to the pro-category of the underlying $\infty$-category of $\cC$, as defined by Lurie. In \S\ref{ss:omega presentable} we give an application of our main theorem to the theory of combinatorial model categories.

In \S\ref{s:topos} we give an application of our main theorem to the theory of shapes $\infty$-topoi. The main result is Theorem~\ref{t:main shape}, which shows that the shape of the hyper-completion of the $\infty$-topos of sheaves on a site can be computed using the topological realization of~\cite{BS1}.

Finally, in \S\ref{s:profinite} we give another application of our main result to the theory of profinite and $p$-profinite homotopy theory. We compare various models from the literature due to Isaksen, Morel and Quick and we show that they model the pro-category of the $\infty$-category of either $\pi$-finite or $p$-finite spaces.

\addsubsection*{Acknowledgments}

We wish to thank Denis-Charles Cisinski for sharing with us the proof of Proposition ~\ref{p:equivalence C_f C}.

\section{Set theoretical foundations}\label{ss:large}

In this paper we will be working with both small and large categories and $\infty$-categories. Such a setting can involve some delicate set theoretical issues. In this section we fix our set theoretical working environment and terminology. We note that these issues are often ignored in texts on categories and $\infty$-categories, and that the reader who wishes to trust his or her intuition in these matters may very well skip this section and refer back to it as needed.

We refer the reader to \cite{Shu} for a detailed account of various possible set theoretical foundations for category theory. Our approach is based mainly on \S 8 of loc. cit. We will be working in ZFC and further assume
\begin{assume}\label{a:large-cardinal}
For every cardinal $\alp$ there exists a strongly inaccessible cardinal $\kappa$ such that $\kappa > \alp$.
\end{assume}

\begin{define}
We define for each ordinal $\alpha$ a set $V_\alpha$ by transfinite induction as follows:
\begin{enumerate}
\item $V_0:=\varnothing$
\item $V_{\alpha+1}:=\cP(V_{\alpha})$
\item If $\beta$ is a limit ordinal we define $V_\beta:=\bigcup_{\alpha<\beta}V_\alpha$.
\end{enumerate}
We refer to elements of $V_\alp$ as \textbf{$\alp$-sets}.
\end{define}

If $\alp$ is a strongly inaccessible cardinal then it can be shown that $V_\alp$ is a \emph{Grothendieck universe}, and thus a model for ZFC.

\begin{define}
Let $\alp$ be a strongly inaccessible cardinal. An \textbf{$\alp$-category} $\cC$ is a pair of $\alp$-sets $\Ob(\cC)$ and $\Mor(\cC)$, together with three functions
\[\Dom:\Mor(\cC)\lrar\Ob(\cC),\]
\[\Cod:\Mor(\cC)\lrar\Ob(\cC),\]
\[\Id:\Ob(\cC)\lrar\Mor(\cC),\]
satisfying the well-known axioms. A \textbf{functor} between $\alp$-categories $\cC$ and $\cD$ consists of a pair of functions $\Ob(\cC)\lrar \Ob(\cD)$ and $\Mor(\cC)\lrar\Mor(\cD)$ satisfying the well-known identities. Given two objects $X,Y \in \Ob(\cC)$ we denote by $\Hom_{\cC}(X,Y) \subseteq \Mor(\cC)$ the inverse image of $(X,Y) \in \Ob(\cC)$ via the map $(\Dom,\Cod): \Mor(\cC) \lrar \Ob(\cC)\times \Ob(\cC)$.
\end{define}

\begin{rem}
If $\cC$ is an $\alp$-category for some strongly inaccessible cardinal $\alp$ then $\cC$ can naturally be considered as a $\bet$-category for any strongly inaccessible cardinal $\bet > \alp$.
\end{rem}

\begin{define}
Let $\bet > \alp$ be strongly inaccessible cardinals. We denote by $\Set_\alp$ the $\bet$-category of $\alp$-sets. We denote by $\Cat_\alp$ the $\bet$-category of $\alp$-categories.
\end{define}

\begin{define}
Let $\alp$ be a strongly inaccessible cardinal.
A \textbf{simplicial $\alp$-set} is a functor $\Del^{\op} \lrar \Set_\alp$ (where $\Del$ is the usual category of finite ordinals). A \textbf{simplicial set} is a simplicial $\alp$-set for some strongly inaccessible cardinal $\alp$. For $\bet > \alp$ a strongly inaccessible cardinal we denote by $\cS_\alp$ the $\bet$-category of simplicial $\alp$-sets.
\end{define}

In this paper we will frequently use the notion of \textbf{$\infty$-category}. Our higher categorical setup is based on the theory quasi-categories due to~\cite{Joy08} and~\cite{Lur09}.

\begin{define}[Joyal, Lurie]\label{d:infty-cat}
Let $\alp$ be a strongly inaccessible cardinal.
An \textbf{$\alp$-$\infty$-category} is a simplicial $\alp$-set satisfying the right lifting property with respect to the maps $\Lam^n_i \hrar \Del^n$ for $0 < i < n$ (where $\Lam^n_i$ is the simplicial set obtained by removing from $\partial\Del^n$ the $i$'th face).
\end{define}

For every strongly inaccessible cardinal $\alp$ the nerve functor $\Ne:\Cat_\alp \lrar \cS_\alp$ is fully faithful and lands in the full subcategory spanned by $\alp$-$\infty$-categories. If $\cC\in \Cat_\alp$ we will often abuse notation and write $\cC$ for the $\alp$-$\infty$-category $\Ne\cC$.

\begin{define}
We denote by $\kappa$ the smallest strongly inaccessible cardinal, by $\lambda$ the smallest strongly inaccessible cardinal bigger than $\kappa$ and by $\delta$ the smallest strongly inaccessible cardinal bigger than $\lambda$.
\end{define}

We will refer to $\kappa$-sets as  \textbf{small sets}, to simplicial $\kappa$-sets as \textbf{small simplicial sets}, to $\kappa$-categories as \textbf{small categories} and to $\kappa$-$\infty$-categories as \textbf{small $\infty$-categories}. For any strongly inaccessible cardinal $\alp$ we say that an $\alp$-$\infty$-category $\cC$ is \textbf{essentially small} if it is equivalent to a small $\infty$-category.

The following special cases merit a short-hand terminology:
\begin{define}
The notations $\Set$, $\Cat$ and $\cS$ without any cardinal stand for the $\lambda$-categories $\Set_\kappa$, $\Cat_\kappa$ and $\cS_\kappa$ respectively. The notations $\ovl{\Set}, \ovl{\Cat}$ and $\ovl{\cS}$ stand for the $\delta$-categories $\Set_\lambda$, $\Cat_\lambda$ and $\cS_\lambda$ respectively.
\end{define}

\begin{define}\label{d:locally-small-2}
Let $\alp$ be a strongly inaccessible cardinal. We say that an $\alp$-category is \textbf{locally small} if $\Hom_{\cC}(X,Y)$ is a small set for every $X,Y \in \cC$. Similarly, we will say that an $\alp$-$\infty$-category is \textbf{locally small} if the mapping space $\Map_{\cC}(X,Y)$ is weakly equivalent to a small simplicial set for every $X,Y \in \cC$.
\end{define}

In ordinary category theory one normally assumes that all categories are locally small. In the setting of higher category theory it is much less natural to include this assumption in the definition itself. In order to be as consistent as possible with the literature we employ the following convention:
\begin{conv}\label{c:infty_cat}
The term \textbf{category} without an explicit cardinal always refers to a locally small $\lambda$-category. By contrast, the term \textbf{$\infty$-category} without an explicit cardinal always refers to a $\lambda$-$\infty$-category (which is not assumed to be locally small).
\end{conv}

\begin{define}\label{d:small functor}
Let $\bet > \alp$ be strongly inaccessible cardinals and let $f:\cC \lrar \cD$ be a map of $\bet$-$\infty$-categories. We say $f$ is \textbf{$\alp$-small} if there exists a full sub-$\alp$-category $\cC_0 \subseteq \cC$ such that $f$ is a left Kan extension of $f|_{\cC_0}$ along the inclusion $\cC_0 \subseteq \cC$. When $\alp=\kappa$ we will also say that $f$ is \textbf{small}.
\end{define}

\begin{rem}\label{r:equiv small}
The following criterion for $\alp$-smallness is useful to note. Let $f: \cC \lrar \cD$ be a map of $\bet$-$\infty$-categories. Suppose there exists a diagram of the form
$$ \xymatrix{
\cC_0 \ar^{g}[r]\ar_{h}[d] & \cD \\
\cC \ar_{f}[ur] & \\
}$$
with $\cC_0$ an $\alp$-$\infty$-category, and a natural transformation $u: g \Rightarrow f \circ h$ exhibiting $f$ as a left Kan extension of $g$ along $h$. Since $h$ factors through a full inclusion $\cC_0' \subseteq \cC$, with $\cC_0'$ an $\alp$-$\infty$-category, it follows that $f$ is a left Kan extension of some functor $h':\cC_0'\lrar \cD$ along the inclusion $\cC_0' \subseteq \cC$. But then we have that $h'\simeq f|_{\cC'_0}$ (\cite[after Proposition 4.3.3.7]{Lur09}), and so $f$ is $\alp$-small.
\end{rem}

\section{Preliminaries from higher category theory}\label{s:homotopical prelim}

In this section we recall some necessary background from higher category theory and prove a few preliminary results which will be used in the following sections.

\subsection{Cofinal and coinitial maps}\label{ss:cofinal}
In this subsection we recall the notion of \textbf{cofinal} and \textbf{coinitial} maps of $\infty$-categories.

Let $\vphi: \cC \lrar \cD$ be a map of $\infty$-categories (see Definition~\ref{d:infty-cat} and Convention \ref{c:infty_cat}). Given an object $d \in \cD$ we denote by $\cC_{/d} = \cC \times_{\cD} \cD_{/d}$ where $\cD_{/d}$ is the $\infty$-category of \textbf{objects over $d$} (see~\cite[Proposition 1.2.9.2]{Lur09}). If $\cC$ and $\cD$ are (the nerves of) ordinary categories, then $\cC_{/d}$ is an ordinary category whose objects are given by pairs $(c,f)$ where $c$ is an object in $\cC$ and $f: \vphi(c) \lrar d$ is a map in $\cD$. Similarly, we denote by $\cC_{d/} = \cC \times_{\cD} \cD_{d/}$ where $\cD_{d/}$ is now the $\infty$-category of object \textbf{under} $d$.

\begin{define}\label{d:cofinal}
Let $\vphi: \cC \lrar \cD$ be a map of $\infty$-categories. We say that $\vphi$ is \textbf{cofinal} if $\cC_{d/}$ is weakly contractible (as a simplicial set) for every $d \in \cD$. Dually, we say that $\vphi$ is \textbf{coinitial} if $\cC_{/d}$ is weakly contractible for every $d \in \cD$.
\end{define}

\begin{rem}
Let $\cC$ be the $\infty$-category with one object $\ast \in \cC$ and no non-identity morphisms. Then $f: \cC \lrar \cD$ is cofinal if and only if the object $f(\ast)$ is a final object in $\cD$. Similarly, $f: \cC \lrar \cD$ is coinitial if and only if $f(\ast)$ is an initial object in $\cD$.
\end{rem}

A fundamental property of cofinal and coinitial maps is the following:
\begin{thm}[{\cite[Proposition 4.1.1.8]{Lur09}}]\label{t:cofinal}\
\begin{enumerate}
\item
Let $\vphi: \cC \lrar \cD$ be a cofinal map and let $\cF: \cD^{\triangleright} \lrar \cE$ be a diagram. Then $\cF$ is a colimit diagram if and only if $\cF \circ \vphi^{\triangleright}$ is a colimit diagram.
\item
Let $f: \cC \lrar \cD$ be a coinitial map and let $\cF: \cD^{\triangleleft} \lrar \cE$ be a diagram. Then $\cF$ is a limit diagram if and only if $\cF \circ \vphi^{\triangleleft}$ is a limit diagram.
\end{enumerate}
\end{thm}

Thomason proved in~\cite{Th79} that the homotopy colimit of a diagram of nerves of categories may be identified with the nerve of the corresponding \textbf{Grothendieck construction}. This yields the following important case of Theorem~\ref{t:cofinal}:
\begin{thm}\label{t:cofinal-2}
Let $\cC, \cD$ be ordinary categories and let $f: \cC \lrar \cD$ be a cofinal map (in the sense of Definition~\ref{d:cofinal}). Let $\cF: \cD \lrar \ovl{\Cat}$ be any functor, let $\cG(\cD,\cF)$ be the Grothendieck construction of $\cF$, and let $\cG(\cC,\cF \circ f)$ be the Grothendieck construction of $\cF \circ f$. Then the induced map
\[ \Ne\cG(\cC,\cF \circ f) \x{\simeq}{\lrar} \Ne\cG(\cD,\cF) \]
is a weak equivalence of simplicial sets.
\end{thm}

\begin{cor}[Quillen's theorem A]
Let $f: \cC \lrar \cD$ be a cofinal functor between ordinary categories. Then the induced map on nerves
\[ \Ne(\cC) \x{\simeq}{\lrar} \Ne(\cD) \]
is a weak equivalence of simplicial sets.
\end{cor}

\begin{rem}
Definition~\ref{d:cofinal} pertains to the notions of cofinality and coinitiality which are suitable for higher category theory. In the original definition of these notions, the categories $\cC_{d/}$ and $\cC_{/d}$ were only required to be \textbf{connected}. This is enough to obtain Theorem~\ref{t:cofinal} when $\cE$ is an \textbf{ordinary category}. We note, however, that for functors whose domains are filtered (resp. cofiltered) categories, the classical and the higher categorical definitions of cofinality (resp. coinitiality) coincide (see Lemma~\ref{l:coincide} below).
\end{rem}

\subsection{Relative categories and $\infty$-localizations}\label{ss:infinity}
In this subsection we recall the notion of \textbf{relative categories} and the formation of \textbf{$\infty$-localizations}, a construction which associates an underlying $\infty$-category to any relative category. Let us begin with the basic definitions.

\begin{define}\label{d:rel}
A \textbf{relative category} is a category $\cC$ equipped with a subcategory
$\cW \subseteq \cC$ that contains all the objects. We refer to the maps in $\cW$ as \textbf{weak equivalences}. A \textbf{relative map} $(\cC,\cW) \lrar (\cD,\cU)$ is a map $f: \cC \lrar \cD$ sending $\cW$ to $\cU$.
\end{define}

Given a relative category $(\cC,\cW)$ one may associate to it an $\infty$-category $\cC_\infty = \cC[\cW^{-1}]$, equipped with a map $\cC \lrar \cC_\infty$, which is characterized by the following universal property: for every $\infty$-category $\cD$, the natural map
\[ \Fun(\cC_\infty,\cD) \lrar \Fun(\cC,\cD) \]
is fully faithful, and its essential image is spanned by those functors $\cC \lrar \cD$ which send $\cW$ to equivalences. The $\infty$-category $\cC_\infty$ is called the \textbf{$\infty$-localization} of $\cC$ with respect to $\cW$. In this paper we also refer to $\cC_\infty$ as the \textbf{underlying $\infty$-category of $\cC$}, or the $\infty$-category \textbf{modelled by $\cC$}. We note that this notation and terminology is slightly abusive, as it makes no direct reference to $\cW$. We refer the reader to~\cite{Hin} for a more detailed exposition. The $\infty$-category $\cC_\infty$ may be constructed in the following three equivalent ways

\begin{enumerate}
\item
One may construct the \textbf{Hammock localization} $L^H(\cC,\cW)$ of $\cC$ with respect to $\cW$ (see~\cite{DK80}), and obtain a simplicial category. The $\infty$-category $\cC_\infty$ can then be obtained by taking the coherent nerve of any fibrant model of $L^H(\cC,\cW)$ (with respect to the Bergner model structure).
\item
One may consider the \textbf{marked simplicial set} $\Ne_+(\cC,\cW) = (\Ne(\cC),\cW)$. The $\infty$-category $\cC_\infty$ can then be obtained by taking the underlying simplicial set of any fibrant model of $\Ne_+(\cC,\cW)$ (with respect to the Cartesian model structure, see~\cite[\S 3]{Lur09}).
\item
One may apply to $(\cC,\cW)$ the \textbf{Rezk nerve construction} to obtain a simplicial space $\Ne_{\rezk}(\cC,\cW)$. The space of $n$-simplices of this simplicial space is the nerve of the category whose objects are functors $[n] \lrar \cC$ and whose morphisms are natural transformations which are levelwise weak equivalences. The $\infty$-category $\cC_\infty$ can then be obtained by applying the functor $W_{\bullet,\bullet} \mapsto W_{\bullet,0}$ to any fibrant model of $\Ne_{\rezk}(\cC,\cW)$ (in the complete Segal space model structure).
\end{enumerate}

We refer to~\cite{Hin} for the equivalence of the first two constructions and to the appendix of \cite{Hor16} for the equivalence of the third. Given a relative map $f:(\cC,\cW)\lrar(\cD,\cU)$ we denote by $f_{\infty}:\cC_{\infty}\lrar\cD_{\infty}$ the induced map. The map $f_\infty$ is essentially determined by the universal property of $\infty$-localizations, but it can also be constructed explicitly, depending on the method one uses to construct $\cC_\infty$.

\begin{rem}
We note that the fibrant replacements in either the Bergner, the Cartesian or the complete Segal space model structures can be constructed in such a way that the resulting map on objects is the identity. We will always assume that we use such a fibrant replacement when constructing $\cC_\infty$. This implies that the resulting map $\cC\lrar\cC_\infty$ is also the identity on objects.
\end{rem}

\begin{rem}
If $\cC$ is a category then we may view $\cC$ as a relative category with the weak equivalences being the isomorphisms. In this case we have $\cC_\infty\simeq \cC$.
\end{rem}

\begin{define}\label{d:spaces}
We will denote by $\cS_\infty$ and $\ovl{\cS}_\infty$ the $\infty$-localizations of $\cS$ and $\ovl{\cS}$ respectively with respect to weak equivalences of simplicial sets (see \S\ref{ss:large} for the relevant definitions). We will refer to objects of $\cS_\infty$ as \textbf{small spaces} and to objects of $\ovl{\cS}_\infty$ as \textbf{large spaces}. We will say that a space $X \in \ovl{\cS}_\infty$ is \textbf{essentially small} if it is equivalent to an object in the image of $\cS_\infty \subseteq \ovl{\cS}_\infty$.
\end{define}

As first observed by Dwyer and Kan, the construction of $\infty$-localizations allows one to define \textbf{mapping spaces} in general relative categories:

\begin{define}
Let $(\cC,\cW)$ be a relative category and let $X,Y \in \cC$ be two objects. We denote by
\[ \Map^h_{\cC}(X,Y) :=\Map_{\cC_\infty}(X,Y) \]
the \textbf{derived mapping space} from $X$ to $Y$.
\end{define}

\begin{rem}
If $\cC$ is not small, then $\cC_\infty$ will not be locally small in general (see Definition~\ref{d:locally-small-2}). However, when $\cC$ comes from a model category, it is known that $\cC_\infty$ is always locally small.
\end{rem}

\subsection{Categories of fibrant objects}

In this subsection we recall and prove a few facts about categories of fibrant objects. Let $\cC$ be a category and let $\cM,\cN$ be two classes of morphisms in $\cC$. We denote by ${\cM}\circ {\cN}$ the class of arrows of $\cC$ of the form $g\circ f$ with $g\in\cM$ and $f\in \cN$. Let us begin by recalling the definition of a category of fibrant objects.

\begin{define}[\cite{Bro73}]\label{d:fib-obj}
A \textbf{category of fibrant objects} is a category $\cC$ equipped with two subcategories
\[\Fib, \cW \subseteq \cC\]
containing all the isomorphisms, such that the following conditions are satisfied:
\begin{enumerate}
\item
$\cC$ has a terminal object $\ast \in \cC$ and for every $X \in \cC$ the unique map $X \lrar \ast$ belongs to $\Fib$.
\item $\cW$ satisfies the 2-out-of-3 property.
\item If $f: Y \lrar W$ belongs to $\Fib$ and $h: Z \lrar W$ is any map then the pullback
$$\xymatrix{
X \ar[r]\ar_{g}[d] & Y \ar^{f}[d] \\
Z \ar^{h}[r] & W \\
}$$
exists and $g$ belongs to $\Fib$. If furthermore $f$ belongs to $\cW$ then $g$ belongs to $\cW$ as well.
\item We have $\Mor(\cC) = \Fib \circ \cW$.
\end{enumerate}
\end{define}
We refer to the maps in $\Fib$ as \textbf{fibrations} and to the maps in $\cW$ as \textbf{weak equivalences}. We refer to maps in $\Fib \cap \cW$ as \textbf{trivial fibrations}.

\begin{rem}\label{r:brown}
We note that properties $(1)$ and $(3)$ of Definition~\ref{d:fib-obj} imply that any category of fibrant objects $\cC$ admits finite products. Some authors (notably~\cite{Bro73}) replace property (4) with the a priori weaker statement that for any $X \in \cC$ the diagonal map $X \lrar X \times X$ admits a factorization as in (4) (such a factorization is sometimes called a \textbf{path object} for $X$). By the factorization lemma of~\cite{Bro73} this results in an equivalent definition. In fact, the factorization lemma of~\cite{Bro73} implies something slightly stronger: any map $f: X \lrar Y$ in $\cC$ can be factored as $X \x{i}{\lrar} Z \x{p}{\lrar} Y$ such that $p$ is a fibration and $i$ is a right inverse of a trivial fibration.
\end{rem}

\begin{define}\label{d:left-exact}
A functor $\cF:\cC \lrar \cD$ between categories of fibrant objects is called a \textbf{left exact functor} if
\begin{enumerate}
\item
$\cF$ preserves the terminal object, fibrations and trivial fibrations.
\item
Any pullback square of the form
$$\xymatrix{
X \ar[r]\ar_{g}[d] & Y \ar^{f}[d] \\
Z \ar^{h}[r] & W \\
}$$
such that $f \in \Fib$ is mapped by $\cF$ to a pullback square in $\cD$.
\end{enumerate}
\end{define}

\begin{rem}\label{r:fac-lemma}
By Remark~\ref{r:brown} any weak equivalence $f: X \lrar Y$ in a category of fibrant objects $\cC$ can be factored as $X \x{i}{\lrar} Z \x{p}{\lrar} Y$ such that $p$ is a trivial fibration and $i$ is a right inverse of a trivial fibration. It follows that any left exact functor $f: \cC \lrar \cD$ preserves weak equivalences.
\end{rem}

Given a category of fibrant objects $(\cC,\cW,\Fib)$ we may consider the $\infty$-localization $\cC_\infty = \cC[\cW^{-1}]$ associated to the underlying relative category of $\cC$. In~\cite{Cis10} Cisinski constructs a concrete and convenient model for computing derived mapping spaces in categories of fibrant objects. Let us recall the definition.

\begin{define}\label{d:hom}
Let $\cC$ be a category equipped with two subcategories $\cW,\Fib$ containing all isomorphisms and a terminal object $\ast \in \cC$. Let $X,Y \in \cC$ two objects. We denote by $\uline{\Hom}_{\cC}(X,Y)$ the category of diagrams of the form
\[ \xymatrix{
Z \ar^{f}[r]\ar_{p}[d] & Y \ar[d] \\
X \ar[r] & \ast \\
}\]
where $\ast \in \cC$ is the terminal object and $p: Z \lrar X$ belongs to $\cW \cap \Fib$.
\end{define}

\begin{rem}\label{r:map}
For any object $X \in \cC$ the category $\uline{\Hom}_{\cC}(X,\ast)$ can be identified with the full subcategory of $\cC_{/X}$ spanned by $\Fib\cap\cW$. In particular, $\uline{\Hom}_{\cC}(X,\ast)$ has a terminal object and is hence weakly contractible. For any object $Y \in \cC$ we may identify the category $\uline{\Hom}_{\cC}(X,Y)$ with the Grothendieck construction of the functor $\uline{\Hom}_{\cC}(X,\ast)^{\op} \lrar \Set$ which sends the object $Z \lrar X$ to the set $\Hom_{\cC}(Z,Y)$.
\end{rem}

There is a natural map from the nerve $\Ne\uline{\Hom}_{\cC}(X,Y)$ to the simplicial set $\Map_{L^H(\cC,\cW)}(X,Y)$ where $L^H(\cC,\cW)$ denotes the hammock localization of $\cC$ with respect to $\cW$. We hence obtain a natural map
\begin{equation}\label{e:map}
\Ne\uline{\Hom}_{\cC}(X,Y) \lrar \Map^h_{\cC}(X,Y).
\end{equation}

\begin{rem}\label{r:natural}
If $\cC$ is a category of fibrant objects then $\uline{\Hom}_{\cC}(X,Y)$ depends covariantly on $Y$ and contravariantly on $X$ (via the formation of pullbacks). Furthermore, the map~\ref{e:map} is compatible with these dependencies.
\end{rem}

\begin{prop}[{\cite[Proposition 3.23]{Cis10}}]\label{p:cis-map}
Let $\cC$ be a category of fibrant objects. Then for every $X,Y \in \cC$ the map~\ref{e:map} is a weak equivalence.
\end{prop}

Cisinski's comprehensive work on categories of fibrant objects shows that such a category admits a well-behaved notion of \textbf{homotopy limits} for diagrams indexed by finite posets (and more generally any category whose nerve has only finitely many non-degenerate simplices). Recent work of Szumi{\l}o (see~\cite{Szu14}) shows that a certain variant of the notion of a category of fibrant objects (which includes, in particular, a two-out-of-six axiom for weak equivalences) is in fact \textbf{equivalent}, in a suitable sense, to that of an $\infty$-category admitting finite limits (i.e., limits indexed by simplicial sets with finitely many non-degenerate simplices). Unfortunately, the functor used in~\cite{Szu14} to turn a category of fibrant objects into an $\infty$-category is not the localization functor discussed in \S\ref{ss:infinity} (although recent work of Kapulkin and Szumiło appears to bridge this gap; see \cite{KS15}). All in all, as these lines are written, there has not yet appeared in the literature a proof of the fact that if $\cC$ is a category of fibrant objects, then $\cC_\infty$ has all finite limits. Our goal in the rest of this section is to fill this gap by supplying a proof which is based on Cisinski's work.

Let $\cD$ be a category of fibrant objects and let $\cT$ denote the category
\[\xymatrix{ & \ast\ar[d] \\ \ast\ar[r]  & \ast}\]
Let $\cD^{\cT}_{\spe} \subseteq \cD^{\cT}$ denote the subcategory spanned by those diagrams
\[ \xymatrix{
& X \ar_{f}[d] \\
Y \ar[r]^{g} & Z \\
}\]
such that both $f$ and $g$ are fibrations.

\begin{lem}\label{l:reedy-limit-cis}
Let $\cD$ be a category of fibrant objects. If $\cF:\cT^{\triangleleft} \lrar \cD$ is a limit diagram such that $\cF|_{\cT}$ is belongs to $\cD^{\cT}_{\spe}$ then $\cF_\infty: \cT^{\triangleleft}_\infty \lrar \cD_\infty$ is a limit diagram.
\end{lem}
\begin{proof}
This follows directly from~\cite[Proposition 3.6]{Cis10} and Proposition~\ref{p:cis-map}.
\end{proof}

\begin{lem}\label{l:rigid-2}
Let $\cD$ be a category of fibrant objects and let $u:\Ne(\cT) \lrar \cD_\infty$ be a diagram. Then there exists a diagram $\cF_{\spe}: \cT \lrar \cD$ which belongs to $\cD^{\cT}_{\spe}$ such that the composite
\[ \Ne(\cT) \x{\Ne(\cF_{\spe})}{\lrar} \Ne(\cD) \lrar \cD_\infty \]
is homotopic to $u$.
\end{lem}
\begin{proof}
Let $L^H(\cD,\cW) \x{\simeq}{\lrar} \cD_{\Del}$ be a fibrant replacement with respect to the Bergner model structure such that the map $\Ob(L^H(\cD,\cW)) \lrar \Ob(\cD_{\Del})$ is the identity, so that $\cD_\infty \simeq \Ne(\cD_{\Del})$. By adjunction, the diagram $u:\Ne(\cT) \lrar \cD_\infty$ corresponds to a functor of simplicial categories
\[ \cF:\fC(\Ne(\cT)) \lrar \cD_\Del \]
Since $\cT$ contains no composable pair of non-identity morphisms, the simplicial set $\Ne(\cT)$ does not have any non-degenerate simplices above dimension $1$. It then follows that the counit map $\fC(\Ne(\cT)) \lrar \cT$ is an \textbf{isomorphism}, and so we may represent $\cF$ by a diagram
\[
\xymatrix{
& X \ar_{F}[d] \\
Y \ar^{G}[r] & Z \\
}\]
in $\cD_{\Del}$, which we still denote by the same name $\cF: \cT \lrar \cD_{\Del}$. According to Proposition~\ref{p:cis-map} the maps
\[ \uline{\Hom}_{\cD}(X,Z) \lrar \Map_{\cD_{\Del}}(X,Z) \]
and
\[ \uline{\Hom}_{\cD}(Y,Z) \lrar \Map_{\cD_{\Del}}(Y,Z) \]
are weak equivalences. It follows that there exists a zig-zag $X \x{p}{\llar} X' \x{f}{\lrar} Z$ (with $p$ a trivial fibration) whose corresponding vertex $F' \in \Map_{\cD_{\Del}}(X,Z)$ is homotopic to $F$ and a zig-zag $Y \x{q}{\llar} Y' \x{g}{\lrar} Z$ (with $q$ a trivial fibration) whose corresponding vertex $G' \in \Map_{\cD_{\Del}}(Y,Z)$ is homotopic to $G$. We may then conclude that $\cF$ is homotopic to the diagram $\cF':\cT \lrar \cD_{\Del}$ determined by $F'$ and $G'$. On the other hand, since $p$ and $q$ are weak equivalences it follows that $\cF'$ is equivalent to the composition $\cT \x{\cF''}{\lrar} \cD \lrar \cD_{\Del}$ where $\cF'': \cT\lrar \cD$ is given by
\[
\xymatrix{
& X' \ar_{f}[d] \\
Y' \ar^{g}[r] & Z \\
}\]
Finally, by using property (4) of Definition~\ref{d:fib-obj} we may replace $\cF''$ with a levelwise equivalent diagram $\cF_{\spe}$ which belongs to $\cD^{\cT}_{\spe}$. Now the composed map
\[ \Ne(\cT) \x{\Ne(\cF_{\spe})}{\lrar} \Ne(\cD) \lrar \cD_\infty \]
is homotopic to $u$ as desired.
\end{proof}

\begin{prop}\label{p:brown_PB}
Let $\cD$ be a category of fibrant objects. Then $\cD_{\infty}$ admits finite limits.
\end{prop}
\begin{proof}
According to~\cite[Proposition 4.4.2.6]{Lur09} it is enough to show that $\cD_\infty$ has pullbacks and a terminal object. The fact that the terminal object of $\cD$ is also terminal in $\cD_\infty$ follows from Remark~\ref{r:map}. Finally, the existence of pullbacks in $\cD_{\infty}$ follows from Lemma~\ref{l:reedy-limit-cis} and Lemma~\ref{l:rigid-2}.
\end{proof}

By Remark~\ref{r:fac-lemma} any left exact functor $\cF: \cC \lrar \cD$ preserves weak equivalences and hence induces a functor $\cF_\infty: \cC_\infty \lrar \cD_\infty$ on the corresponding $\infty$-categories.

\begin{prop}\label{p:left exact preserves limits}
Let $\cF:\cC\lrar\cD$ be a left exact functor between categories of fibrant objects. Then $\cF_{\infty}:\cC_{\infty}\lrar \cD_{\infty}$ preserves finite limits.
\end{prop}
\begin{proof}
It suffices to prove that $\cF_{\infty}$ preserves pullbacks and terminal objects. Since the terminal object of $\cC$ is also the terminal in $\cC_\infty$ and since $\cF$ preserves terminal objects it follows that $\cF_\infty$ preserves terminal objects. Now let $\cT$ be as above. By Definition~\ref{d:left-exact} we see that $f$ maps limits $\cT^{\triangleleft}$-diagrams which contain only fibrations to limit diagrams. It then follows from Lemmas~\ref{l:reedy-limit-cis} and~\ref{l:rigid-2} that $\cF_\infty$ preserves limit $\cT^{\triangleleft}$-diagrams, i.e., pullback diagrams.
\end{proof}

\subsection{Weak fibration categories}\label{ss:wfc}

Most relative categories appearing in this paper are \textbf{weak fibration categories}.

\begin{define}\label{d:PB}
Let $\cC$ be category and let $\cM \subseteq \cC$ be a subcategory. We say
that $\cM$ is \textbf{closed under base change} if whenever we have a pullback square:
\[
\xymatrix{A\ar[d]_g\ar[r] & B\ar[d]^f\\
C\ar[r] & D}
\]
such that $f$ is in $\cM$, then $g$ is in $\cM$.
\end{define}

\begin{define}[{cf. \cite{An78, BS1}}]\label{d:weak_fib}
A \textbf{weak fibration category} is a category $\cC$ equipped with two subcategories
\[\Fib, \cW \subseteq \cC\]
containing all the isomorphisms, such that the following conditions are satisfied:
\begin{enumerate}
\item $\cC$ has all finite limits.
\item $\cW$ has the 2-out-of-3 property.
\item The subcategories $\Fib$ and $\Fib \cap \cW$ are closed under base change.
\item (Factorization axiom) We have $\Mor({\cC}) = {\cF ib}\circ {\cW}$.
\end{enumerate}
\end{define}

We refer to the maps in $\Fib$ as \textbf{fibrations} and to the maps in $\cW$ as \textbf{weak equivalences}. We refer to maps in $\Fib \cap \cW$ as \textbf{trivial fibrations}.

\begin{define}
A functor $\cC \lrar \cD$ between weak fibration categories is called a \textbf{weak right Quillen functor} if it preserves finite limits, fibrations and trivial fibrations.
\end{define}

We now recall some terminology from~\cite{BS1}.

\begin{define}\label{def CDS}
Let $\cT$ be a poset. We say that $\cT$ is \textbf{cofinite} if for every element $t \in \cT$ the set $\cT_t:=\{s\in \cT| s < t\}$ is \textbf{finite}.
\end{define}

\begin{define}\label{def natural}
Let $\cC$ be a category admitting finite limits, $\cM$ a class of morphisms in $\cC$, $\cI$ a small category, and $F:X\lrar Y$ a morphism in $\cC^{\cI}$. Then $F$ is:
\begin{enumerate}
\item A \textbf{levelwise $\cM$-map} if for every $i\in \cI$ the morphism $F_i:X_i \lrar Y_i$ is in $\cM$. We denote by $\Lw(\cM)$ the class of levelwise $\cM$-maps.
\item A \textbf{special $\cM$-map} if the following holds:
    \begin{enumerate}
    \item The indexing category $\cI$ is an cofinite poset (see Definition ~\ref{def CDS}).
    \item For every $i \in \cI$ the natural map
     \[X_i \lrar Y_i \times_{\lim \limits_{j<i} Y_j} \lim \limits_{j<i} X_j \]
     belongs to $\cM$.
    \end{enumerate}
    We denote by $\Sp(\cM)$ the class of special $\cM$-maps.
\end{enumerate}
We will say that a diagram $X \in \cC^{\cI}$ is a \textbf{special $\cM$-diagram} if the terminal map $X \lrar \ast$ is a special $\cM$-map.
\end{define}

The following proposition from~\cite{BS1} will be used several times, and is recalled here for the convenience of the reader.

\begin{prop}[{\cite[Proposition 2.19]{BS1}}]\label{p:forF_sp_is_lw_strong}
Let $\mcal{C}$ be a category with finite limits, and $\mcal{M} \subseteq \mcal{C}$ a subcategory that is closed under base change, and contains all the isomorphisms. Let $F:X\lrar Y$ be a natural transformation between diagrams in $\mcal{C}$ which is a special $\mcal{M}$-map. Then $F$ is a levelwise $\mcal{M}$-map.
\end{prop}

The following constructions of weak fibration structures on functors categories will be useful.

\begin{lem}\label{l:projective-injective}
Let $(\cC,\cW,\Fib)$ be a weak fibration category and $\cT$ be a cofinite poset.
\begin{enumerate}
\item
There exists a weak fibration structure on $\cC^{\cT}$ in which the weak equivalences are the levelwise weak equivalences and the fibrations are the levelwise fibrations (see Definition~\ref{def natural}). We refer to this structure as the \textbf{projective} weak fibration structure on $\cC^{\cT}$.
\item
There exists a weak fibration structure on $\cC^{\cT}$ in which the weak equivalences are the levelwise weak equivalences and the fibrations are the special $\Fib$-maps (see Definition~\ref{def natural}). We refer to this structure as the \textbf{injective} weak fibration structure on $\cC^{\cT}$.
\end{enumerate}
\end{lem}

\begin{proof}
We first note that the category $\cC^{\cT}$ has finite limits, and these may be computed levelwise. Furthermore, it is clear that levelwise weak equivalences satisfy the 2-out-of-3 property. Now given a morphism in $\cC^\cT$ we can factor it into a levelwise weak equivalence followed by a special $\Fib$-map by employing the construction described in~\cite[Definition 4.3]{BS15}. By Proposition~\ref{p:forF_sp_is_lw_strong} the latter is also a levelwise fibration. This establishes the factorization axiom for both the projective and injective weak fibration structures. Now levelwise fibrations and levelwise trivial fibrations are clearly closed under composition and base change. It follows from Proposition~\ref{p:forF_sp_is_lw_strong} that a special $\Fib$-map is trivial in the injective structure if and only if it is a levelwise trivial fibration. To finish the proof it hence suffices to show that the special $\Fib$-maps are closed under composition and base change.

We begin with base change. Let $f: \{X_t\} \lrar \{Y_t\}$ be a special $\Fib$-map and let $g: \{Z_t\} \lrar \{Y_t\}$ be any map in $\cC^{\cT}$. Let $t \in \cT$ be an element. Consider the diagram
\[ \xymatrix{
Z_t \times_{Y_t} X_t \ar[r]\ar[d] & X_t \ar[d] \\
Z_t \times_{\displaystyle\mathop{\lim}_{s < t} Z_s}\displaystyle\mathop{\lim}_{s < t} \left(Z_s \times_{Y_s} X_s\right) \ar[r]\ar[d] & Y_t \times_{\displaystyle\mathop{\lim}_{s < t} Y_s} \displaystyle\mathop{\lim}_{s < t} X_s  \ar[d] \ar[r] & \displaystyle\mathop{\lim}_{s < t} X_s \ar[d] \\
Z_t \ar[r] & Y_t \ar[r] & \displaystyle\mathop{\lim}_{s < t} Y_s \\
}\]
Since limits commute with limits it follows that the large bottom horizontal rectangle is Cartesian. Since the right bottom inner square is Cartesian we get by the pasting lemma for Cartesian squares that the bottom left inner square is Cartesian. Since the vertical left rectangle is Cartesian we get from the pasting lemma that the top left square is Cartesian. The desired result now follows from the fact that $\Fib$ is closed under base change.

We now turn to composition. Let $f: \{X_t\} \lrar \{Y_t\}$ and $g: \{Y_t\} \lrar \{Z_t\}$ be special $\Fib$-maps in $\cC^{\cT}$, an let $t\in \cT$. We need to show that
\[X_t\lrar Z_t\times_{\lim \limits_{s<t} Z_s}\lim_{s<t}X_s\]
belongs to $\Fib$. But this map is the composition of two maps
\[X_t\lrar Y_t\times_{\lim \limits_{s<t}Y_s}\lim_{s<t}X_s\lrar Z_t\times_{\lim \limits_{s<t}Z_s}\lim_{s<t}X_s.\]
The first map belongs to $\Fib$ because $f: \{X_t\} \lrar \{Y_t\}$ is a special $\Fib$-map, and the second map belongs to $\Fib$ because we have a pullback square
\[\xymatrix{Y_t\times_{\displaystyle\mathop{\lim}_{s<t}Y_s} \displaystyle\mathop{\lim}_{s<t}X_s\ar[r]\ar[d] & Y_t\ar[d]\\
Z_t\times_{\displaystyle\mathop{\lim}_{s<t}Z_s}\displaystyle\mathop{\lim}_{s<t}X_s \ar[r] & Z_t\times_{\displaystyle\mathop{\lim}_{s<t}Z_s}\displaystyle\mathop{\lim}_{s<t}Y_s,}\]
and the right vertical map belongs to $\Fib$ because $g: \{Y_t\} \lrar \{Z_t\}$ is a special $\Fib$-map.
Thus the result follows from that fact that $\Fib$ is closed under composition and base change.
\end{proof}

Any weak fibration category $(\cC,\Fib,\cW)$ has an underlying structure of a relative category given by $(\cC,\cW)$, and hence an associated $\infty$-category $\cC_\infty$ (see \S\ref{ss:infinity}). However, unlike the situation with categories of fibrant objects, weak right Quillen functors $f: \cC \lrar \cD$ do not, in general, preserve weak equivalences. To overcome this technicality, one may consider the full subcategory $\iota: \cC^{\fib} \hrar \cC$ spanned by the fibrant objects. Endowed with the fibrations and weak equivalences inherited from $\cC$, the category $\cC^{\fib}$ has the structure of a category of fibrant objects (see Definition~\ref{d:fib-obj}). By Remark~\ref{r:fac-lemma} weak right Quillen functors preserve weak equivalences between fibrant objects, and so the restriction gives a relative functor $f^{\fib}:\cC^{\fib}\lrar\cD$. Thus, any weak right Quillen functor induces a diagram of $\infty$-categories of the form
\[ \xymatrix{
& \left(\cC^{\fib}\right)_\infty \ar^{f^{\fib}_\infty}[dr] \ar_{\iota_\infty}[dl] & \\
\cC_\infty && \cD_\infty. \\
}\]
We will prove in Proposition~\ref{p:equivalence C_f C} below that the map $\iota_\infty: \cC^{\fib}_\infty \lrar \cC_\infty$ is an equivalence of $\infty$-categories. This implies that one can complete the diagram above into a triangle
\[ \xymatrix{
& \left(\cC^{\fib}\right)_\infty \ar^{f^{\fib}_\infty}[dr] \ar_{\iota_\infty}[dl] & \\
\cC_\infty \ar[rr]^{f_\infty} && \cD_\infty \\
}\]
of $\infty$-categories, together with a commutation homotopy $u:f_\infty \circ \iota \simeq f^{\fib}_\infty$. Furthermore, the pair $(f_\infty, u)$ is unique up to a contractible space of choices. We call such a $(f_\infty, u)$ a \textbf{right derived functor} for $f$. In the following sections we simply write $f_\infty: \cC_\infty \lrar \cD_\infty$, without referring explicitly to $u$, and suppressing the choice that was made.

The rest of this subsection is devoted to proving that $\iota_\infty:\cC^{\fib}_{\infty} \lrar \cC_\infty$ is an equivalence of $\infty$-categories. The proof we shall give is due to Cisinski and was described to the authors in personal communication.

Given a weak fibration category $(\cC,\cW,\Fib)$ we denote by $\cW^{\fib} \subseteq \cW$ the full subcategory of $\cW$ spanned by the fibrant objects.

\begin{lem}\label{l:cofinal-fib}
Let $(\cC,\cW,\Fib)$ be a weak fibration category. Then the functor
\[ \cW^{\fib} \lrar \cW \]
is cofinal.
\end{lem}

\begin{proof}
Let $X \in \cW$ be an object. We need to show that the category of fibrant replacements $\cW^{\fib}_{X/}$ is contractible. By~\cite[Lemme d'asph\'ericit\'e p. 509]{Cis10}, it suffices to prove that for any finite poset $\cT$, any simplicial map
\[\Ne(\cT) \lrar \Ne(\cW^{\fib}_{X/})\]
is connected to a constant map by a zig-zag of simplicial homotopies. Since the nerve functor $\Ne:\ovl{\Cat} \lrar \ovl{\cS}$ is fully faithful, it suffices to prove that any functor
\[f:\cT \lrar \cW^{\fib}_{X/}\]
is connected to a constant functor by a zig-zag of natural transformations. Such a functor is the same data as a functor $f:\cT \lrar \cW$ which is objectwise fibrant together with a natural transformation $X \lrar f$ in $\cW^{\cT}$ (where $X$ denotes the constant functor with value $X$).

Consider the \textbf{injective} weak fibration structure on $\cC^{\cT}$ (see Lemma~\ref{l:projective-injective}). Using the factorization property we may factor the morphism $f \lrar *$ as
\[f\xrightarrow{\Lw(\cW)} f'\xrightarrow{\Sp(\cF ib)}*.\]
By Proposition~\ref{p:forF_sp_is_lw_strong} $f': \cT \lrar \cC$ is also levelwise fibrant and by~\cite[Proposition 2.17]{BS1} the limit $\lim_{\cT} f' \in \cC$ is fibrant. We can now factor the map $X \lrar \lim_{\cT} f'$ as a weak equivalence $X \lrar Y$ followed by a fibration $Y \lrar \lim_{\cT}f'$. Then $Y$ is fibrant, and by the 2-out-of-3 property in $\cC^{\cT}$ the map $Y \lrar f'$ is a levelwise weak equivalence. Thus, $Y$ determines a constant functor $\cT \to \cW^{\fib}_{X/}$ which is connected to $f$ by a zig-zag of natural transformations
$f\Rightarrow f' \Leftarrow Y.$
\end{proof}

\begin{prop}[Cisinski]\label{p:equivalence C_f C}
Let $\cC$ be a weak fibration category. Then the inclusion $\cC^{\fib}\to\cC$ induces an equivalence
\[ \left(\cC^{\fib}\right)_\infty \lrar \cC_\infty .\]
\end{prop}

\begin{proof}
According to Section \ref{ss:infinity} it suffices to show that the induced map on Rezk's nerve
\[\Ne_{\rezk}(\cC^{\fib})\lrar \Ne_{\rezk}(\cC)\]
is an equivalence in the model structure of complete Segal spaces. For each $[n]$, the category $\cC^{[n]}$ can be endowed with the \textbf{projective} weak fibration structure (see Lemma~\ref{l:projective-injective}). Applying Lemma~\ref{l:cofinal-fib} to $\cC^{[n]}$ and using Quillen's theorem A, we get that the map
\[\Ne_{\rezk}(\cC^{\fib})\lrar \Ne_{\rezk}(\cC)\]
is a levelwise equivalence and hence an equivalence in the complete Segal space model structure.
\end{proof}

We finish this subsection by stating a few important corollaries of Proposition~\ref{p:equivalence C_f C}.
\begin{cor}\label{c:map}
Let $\cC$ be a weak fibration category and let $X,Y \in \cC$ be two \textbf{fibrant} objects. Then the natural map
\[ \Ne\uline{\Hom}_{\cC}(X,Y) \x{\simeq}{\lrar} \Map^h_{\cC}(X,Y) \]
is a weak equivalence.
\end{cor}
\begin{proof}
Combine Proposition~\ref{p:equivalence C_f C} and Proposition~\ref{p:cis-map}.
\end{proof}

\begin{rem}
Corollary~\ref{c:map} can be considered as a generalization of~\cite[Proposition 6.2]{BS1}.
\end{rem}

\begin{cor}\label{p:WFC_PB}
Let $\cD$ be a weak fibration category. Then $\cD_{\infty}$ has finite limits.
\end{cor}
\begin{proof}
Combine Proposition~\ref{p:equivalence C_f C} and Proposition~\ref{p:brown_PB}.
\end{proof}

\begin{cor}\label{p:weak right preserves limits}
Let $f:\cC\lrar\cD$ be a weak right Quillen functor between weak fibration categories. Then the right derived functor $f_{\infty}:\cC_{\infty}\lrar \cD_{\infty}$ preserves finite limits.
\end{cor}
\begin{proof}
Combine Proposition~\ref{p:equivalence C_f C} and Proposition~\ref{p:left exact preserves limits}.
\end{proof}

\subsection{Model categories}\label{ss:model}
In this subsection we recall some basic definitions and constructions from the theory of \textbf{model categories}. We then attempt to fill a gap in the literature by showing that if $\cM$ is a model category then $\cM_\infty$ admits small limits and colimits (Theorem~\ref{t:model-complete}), and that these may be computed using the standard model categorical toolkit (Proposition~\ref{p:injective-limit}). We begin by recalling the basic definitions.

\begin{define}\label{d:model_structure}
A \textbf{model category} is a quadruple $(\cC,\cW,\Fib,\Cof)$, consisting of a category $\cC$, and three subcategories $\cW,\cF ib,\Cof$ of $\cC$, called \textbf{weak equivalences}, \textbf{fibrations}, and \textbf{cofibrations}, satisfying the following properties:
\begin{enumerate}
\item The category $\cC$ has all small limits and colimits.
\item The subcategory $\cW$ satisfies the two-out-of-three property.
\item The subcategories $\cW,\Fib,\Cof$ are closed under retracts.
\item Trivial cofibrations have the left lifting property with respect to fibrations, and the cofibrations have the
left lifting property with respect to trivial fibrations.
\item Any morphism in $\cC$ can be factored (not necessarily functorially) into a cofibration followed by a trivial
fibration, and into a trivial cofibration followed by a fibration.
\end{enumerate}
\end{define}

\begin{rem}
Note that our definition of a model category is weaker than the one given in~\cite{Hir03}, because we do not require the factorizations to be functorial. Indeed, in one of our main examples, the projective model structure on $\Pro(\Sh_{\Del}(\cC))$ considered in Section~\ref{s:topos}, we do not have functorial factorizations.
\end{rem}

The notion of a morphism of model categories is given by a \textbf{Quillen adjunction}, that is, an adjunction $\cL:\cM \rightleftarrows \cN: \cR$ such that $\cL$ preserves cofibrations and trivial cofibrations and $\cR$ preserves fibrations and trivial fibrations. Note that any model category is in particular a weak fibration category with respect to $\Fib$ and $\cW$, and every right Quillen functor can also be considered as a weak right Quillen functor between the corresponding weak fibration categories.

We note that the method described in \S\ref{ss:wfc} of constructing derived functors can be employed for \textbf{model categories} as well. Given a Quillen adjunction $\cL:\cM \rightleftarrows \cN: \cR$, we may form a right derived functor $\cR_\infty: \cN_\infty \lrar \cM_\infty$ using the full subcategory $\cN^{\fib}$ spanned by the fibrant objects and a left derived functor $\cL_\infty:\cM_\infty \lrar \cN_\infty$ using the full subcategory $\cM^{\cof} \subseteq \cM$ spanned by cofibrant objects.

Let $\cM$ be a model category and $\cT$ a cofinite poset. One is often interested in endowing the functor category $\cM^{\cT}$ with a model structure. We first observe the following:
\begin{lem}\label{l:poset-reedy}
Let $\cT$ be a cofinite poset (see Definition~\ref{def CDS}). Then $\cT$ is a Reedy category with only descending morphisms.
\end{lem}
\begin{proof}
For each $t \in \cT$ define the \textbf{degree} $\deg(t)$ of $t$ to be the maximal integer $k$ such that there exists an ascending chain in $\cT$ of the form
\[ t_0 < t_1 < ... < t_k = t \]
Since $\cT$ is cofinite the degree of each element is a well-defined integer $\geq 0$. The resulting map $\deg:\cT \lrar \NN \cup \{0\}$ is strongly monotone (i.e. $t < s \Rightarrow \deg(t) < \deg(s)$) and hence exhibits $\cT$ as a Reedy category in which all the morphisms are descending.
\end{proof}

\begin{rem}\label{r:special-is-reedy}
Let $\cM$ be a model category and $\cT$ a cofinite poset. Then a map $f: X \lrar Y$ in $\cM^{\cT}$ is a special $\Fib$-map (see Definition~\ref{def natural}) if and only if it is a Reedy fibration with respect to the Reedy structure of Lemma~\ref{l:poset-reedy}. Similarly, an object $X \in \cM^{\cT}$ is a special $\Fib$-diagram if and only if it is Reedy fibrant.
\end{rem}

Let $\cI$ be a small category. Recall that a model structure on $\cM^{\cI}$ is called \textbf{injective} if its weak equivalences and cofibrations are defined levelwise. If an injective model structure on $\cM^{\cI}$ exists then it is unique.
\begin{cor}\label{c:poset-reedy}
Let $\cM$ be a model category and $\cT$ a cofinite poset. Then the \textbf{injective model structure} on $\cM^{\cT}$ exists and coincides with the Reedy model structure associated to the Reedy structure of Lemma~\ref{l:poset-reedy}. Furthermore, the underlying weak fibration structure coincides with the injective weak fibration structure of Lemma~\ref{l:projective-injective}.
\end{cor}
\begin{proof}
This follows directly from the fact the Reedy structure on $\cT$ has only descending morphisms, and that the Reedy model structure always exists. The last property follows from Remark~\ref{r:special-is-reedy}.
\end{proof}

It seems to be well-known to experts that if $\cM$ is a model category then $\cM_\infty$ admits all small limits and colimits, and that these limits and colimits can be computed via the ordinary model theoretical techniques. For simplicial combinatorial model categories such results were established as part of the general theory due to Lurie which relates simplicial combinatorial model categories and presentable $\infty$-categories (see~\cite[Proposition A.3.7.6]{Lur09}). The theory can be extended to general combinatorial model categories using the work of Dugger~\cite{Dug01} (see~\cite[Propositions 1.3.4.22, 1.3.4.23 and  1.3.4.24]{Lur11a}). However, it seems that a complete proof that the underlying $\infty$-category of any model category admits limits and colimits has yet to appear in the literature. For applications to our main theorem, and for the general benefit of the theory, we will bridge this gap. We first show that the model categorical construction of limits always gives the correct limit in the underlying $\infty$-category.

\begin{prop}\label{p:injective-limit}
Let $\cM$ be a model category and let $\cT$ be a small category such that the injective model structure on $\cM^{\cT}$ exists. Let $\ovl{\cF}: \cT^{\triangleleft} \lrar \cM$ be a limit diagram such that $\cF = \ovl{\cF}_{|\cT}$ is injectively fibrant. Then the image of $\ovl{\cF}$ in $\cM_\infty$ is a limit diagram.
\end{prop}

\begin{proof}
Let $L^H(\cM,\cW) \x{\simeq}{\lrar} \cM_{\Del}$ be a fibrant replacement with respect to the Bergner model structure such that the map $\Ob(L^H(\cM,\cW)) \lrar \Ob(\cM_{\Del})$ is the identity, so that $\cM_\infty \simeq \Ne(\cM_{\Del})$. In light of~\cite[Theorem 4.2.4.1]{Lur09}, it suffices to show that $\ovl{\cF}$ is a homotopy limit diagram in $\cM_{\Del}$ in the following sense (see~\cite[Remark A.3.3.13]{Lur09}): for every object $Z$ the induced diagram
\[ \Map_{\cM_{\Del}}(Z,\ovl{\cF}(-)): \cT^{\triangleleft} \lrar \cS \]
is a homotopy limit diagram of simplicial sets. Since every object is weakly equivalent to a fibrant object it will suffice to prove the above claim for $Z$ fibrant. Since $\cF$ is injectively fibrant it is also levelwise fibrant, and since the limit functor $\lim: \cM^{\cT} \lrar \cM$ is a right Quillen functor with respect to the injective model structure we get that $\ovl{\cF}$ is levelwise fibrant as well. By Corollary~\ref{c:map} and Remark~\ref{r:natural} we have a levelwise weak equivalence
\[ \Ne\uline{\Hom}_{\cM}(Z,\ovl{\cF}(-)) \x{\simeq}{\lrar} \Map_{\cM_{\Del}}(Z,\ovl{\cF}(-)) .\]
Hence, it suffices to prove that the diagram $\Ne\uline{\Hom}_{\cM}(Z,\ovl{\cF}(-))$ is a homotopy limit diagram of simplicial sets.

Let $Z_\bullet$ be a special cosimplicial resolution for $Z$ in the sense of~\cite[Remark 6.8]{DK80}.
Let $Y$ be any fibrant object in $\cM$ and let $\cH(Z_\bullet,Y)$ be the Grothendieck construction of the functor $\Del^{\op} \lrar \Set$ sending $[n]$ to $\Hom_{\cM}(Z_n,Y)$.
Recall from Remark~\ref{r:map} that the category $\uline{\Hom}_{\cM}(Z,\ast)$ is just the full subcategory of $\cC_{/Z}$ spanned by trivial fibrations. By~\cite[Proposition 6.12]{DK80} we have a coinitial functor $\Del\lrar \uline{\Hom}_{\cM}(Z,\ast)$ which sends $[n]$ to the composed map $Z_n \lrar Z_0\lrar Z$. (Note that the term \textbf{left cofinal} loc. cit. is what we call \textbf{coinitial} here.). By Remark~\ref{r:map} we may identify the category $\uline{\Hom}_{\cM}(Z,Y)$ with the Grothendieck construction of the functor $\uline{\Hom}_{\cM}(Z,\ast)^{\op} \lrar \Set$ which sends a trivial fibration $W \lrar Z$ to the set $\Hom_{\cM}(W,Y)$. By Theorem~\ref{t:cofinal-2} we obtain a natural weak equivalence
\[ \Ne\cH(Z_\bullet,Y) \x{\simeq}{\lrar} \Ne\uline{\Hom}_{\cM}(Z,Y) .\]
Note that the objects of $\cH(Z_\bullet,Y)$ are pairs $([n],f)$ where $[n] \in \Del$ is an object and $f: Z_n \lrar Y$ is a map in $\cM$. Thus, we may identify $\cH(Z_\bullet,Y)$ with the category of simplices of the simplicial set $\Hom_{\cM}(Z_\bullet,Y)$. We thus have a natural weak equivalence
\[ \Ne\cH(Z_\bullet,Y) \x{\simeq}{\lrar} \Hom_{\cM}(Z_\bullet,Y) .\]
Hence, it suffices to show that the diagram
\[\Hom_{\cM}(Z_\bullet,\ovl{\cF}(-)): \cT^{\triangleleft} \lrar \cS \]
is a homotopy limit diagram of simplicial sets. Now for $A \in \Set$ and $M \in \cM$ we define
\[A\otimes M:=\coprod_{a\in A}M\in\cM.\]
This makes $\cM$ tensored over $\Set$. For any simplicial set $K:\Delta^{\op}\lrar \Set$ we can now define
$\cL_{Z_\bullet}(K) := K \otimes_{\Delta} Z_\bullet\in\cM$
to be the appropriate coend. We now have an adjunction
\[ \cL_{Z_\bullet}: \cS \adj \cM :\cR_{Z_\bullet} ,\]
where $\cR_{Z_\bullet}(X) = \Hom_\cM(Z_\bullet,X)$. In light of~\cite[Corollary 16.5.4]{Hir03} this adjunction is a Quillen pair.  We then obtain an induced Quillen pair
\[ \cL^{\cT}_{Z_\bullet}: \cS^{\cT} \adj \cM^{\cT} :\cR^{\cT}_{Z_\bullet} \]
between the corresponding injective model structures. Since $\cF \in \cM^{\cT}$ is an injectively fibrant diagram it follows that $\Hom_{\cM}(Z_\bullet,\cF(-))$ is injectively fibrant. Since $\Hom_{\cM}(Z_\bullet,\ovl{\cF}(-))$ is a limit diagram we may conclude that it is also a homotopy limit diagram.
\end{proof}

\begin{rem}
Applying Proposition~\ref{p:injective-limit} to the opposite model structure on $\cM^{\op}$ we obtain the analogous claim for colimits of projectively cofibrant diagrams in $\cM$.
\end{rem}

\begin{rem}
It is not known if the injective model structure on $\cM^{\cI}$ exists in general. However, using, for example,~\cite[Variant 4.2.3.15]{Lur09}, one may show that for any small category $\cI$ there exists a cofinite poset $\cT$ and a coinitial map $\cT \lrar \cI$ (see Definition~\ref{d:cofinal}). One may hence always compute homotopy limits of functors $\cF: \cI \lrar \cM$ by first restricting to $\cT$, and then computing the homotopy limit using the injective model structure on $\cM^{\cT}$. According to Proposition~\ref{p:injective-limit} and Theorem~\ref{t:cofinal} this procedure yields the correct limit in $\cM_\infty$.
\end{rem}

\begin{thm}\label{t:model-complete}
Let $\cM$ be a model category. Then the $\infty$-category $\cM_\infty$ has all small limits and colimits.
\end{thm}
\begin{proof}

We prove the claim for limits. The case of colimits can be obtained by applying the proof to the opposite model structure on $\cM^{\op}$. According to~\cite[Proposition 4.4.2.6]{Lur09} it is enough to show that $\cM_\infty$ has pullbacks and small products. The existence of pullbacks follows from Corollary~\ref{p:WFC_PB}, since $\cM$ is in particular a weak fibration category. To see that $\cM_\infty$ admits small products let $S$ be a small set. We first observe that the product model structure on $\cM^S$ coincides with the injective model structure. Let $L^H(\cM,\cW) \x{\simeq}{\lrar} \cM_{\Del}$ be a fibrant replacement with respect to the Bergner model structure such that the map $\Ob(L^H(\cM,\cW)) \lrar \Ob(\cM_{\Del})$ is the identity, so that $\cM_\infty \simeq \Ne(\cM_{\Del})$. The data of a map $S \lrar \cM_{\infty}$ is just the data of a map of sets
\[ S \lrar (\cM_{\infty})_0= \Ob(\cM_{\Del}) = \Ob(\cM) .\]
Let $u:S\lrar \mathrm{Ob}(\cM)$ be such a map. We can then factor each map $u(s)\lrar\ast$ as a weak equivalence followed by a fibration. This gives us a map $v:S\lrar\cM^{\fib}$ with a weak equivalence $u\lrar v$ and $v$ injectively fibrant. Now, according to Proposition~\ref{p:injective-limit}, the limit of $v$ is a model for the homotopy product of $v$ and hence also of $u$.
\end{proof}

\begin{rem}\label{r:adjunction}
Let $\cL:\cM \rightleftarrows \cN: \cR$ be a Quillen adjunction. According to~\cite[Theorem 2.1]{MG15} the derived functors $\cL_\infty$ and $\cR_\infty$ are adjoints in the $\infty$-categorical sense. It follows that $\cL_\infty$ preserves colimits and $\cR_\infty$ preserves limits (see~\cite[Proposition 5.2.3.5]{Lur09}).
\end{rem}

\section{Pro-categories}\label{s:pro}
\subsection{Pro-categories in ordinary category theory}

In this subsection we recall some general background on pro-categories and prove a few lemmas which will be used in \S\ref{s:main}. Standard references include~\cite{SGA4-I} and~\cite{AM69}.

\begin{define}\label{d:cofiltered}
We say that a category $\cI$ \textbf{cofiltered} if the following conditions are satisfied:
\begin{enumerate}
\item $\cI$ is non-empty.
\item For every pair of objects $i,j \in \cI$, there exists an object $k \in \cI$, together with
morphisms $k \lrar i$ and $k \lrar j$.
\item For every pair of morphisms $f,g:i \lrar j$ in $\cI$, there exists a morphism $h:k\lrar i$ in $\cI$, such that $f\circ h=g\circ h$.
\end{enumerate}
We say that a category $\cI$ is \textbf{filtered} if $\cI^{\op}$ is cofiltered.
\end{define}

\begin{conv}\label{cv:poset}
If $\cT$ is a small partially ordered set, then we view $\cT$ as a small category which has a single morphism $u\lrar v$ whenever $u\geq v$. It is then clear that a poset $\cT$ is cofiltered if and only if $\cT$ is non-empty, and for every $a,b\in \cT$, there exists a $c\in \cT$, such that $c\geq a$ and $c\geq b$.
\end{conv}

We now establish a few basic properties of cofiltered categories.

\begin{lem}\label{l:finite}
Let $\cI$ be cofiltered category and let $\cE$ be a category with finitely many objects and finitely many morphisms. Then any functor $\cF: \cE \lrar \cI$ extends to a functor $\ovl{\cF}:\cE^{\triangleleft} \lrar \cI$.
\end{lem}

\begin{proof}
If $\cE$ is empty, then the desired claim is exactly property (1) of Definition~\ref{d:cofiltered}. Now assume that $\cE$ is non-empty. Since $\cE$ has finitely many objects we may find, by repeated applications of property (2) of Definition~\ref{d:cofiltered}, an object $i \in \cI$ admitting maps $f_e: i \lrar \cF(e)$ for every $e \in \cE$. Now for every morphism $g: e \lrar e'$ in $\cE$ we obtain two maps $i \lrar \cF(e')$, namely $f_{e'}$ on the one hand and $\cF(g) \circ f_e$ on the other. Since $\cE$ has finitely many morphisms we may find, by repeated applications of property (3) of Definition~\ref{d:cofiltered}, a map $h: j \lrar i$ in $\cI$ such that $f_{e'} \circ h = \cF(g) \circ f_e \circ h$ for every morphism $g: e \lrar e'$ in $\cE$. The morphisms $f_{e} \circ h: j \lrar \cF(e)$ now form an extension of $\cF$ to a functor $\ovl{\cF}: \cE^{\triangleleft} \lrar \cI$.
\end{proof}

\begin{lem}\label{l:coincide}
Let $\cI$ be a cofiltered category and let $\cF: \cI \lrar \cJ$ be a functor such that for each $j \in \cJ$ the category $\cI_{/j}$ is \textbf{connected}. Then $\cF$ is coinitial (see Definition~\ref{d:cofinal}).
\end{lem}

\begin{proof}
Since $\cI_{/j}$ is connected, it suffices, By~\cite[Lemme d'asph\'ericit\'e p. 509]{Cis10}, to prove that for any \textbf{connected} finite poset $\cT$, and any map
\[ \cG:\cT \lrar \cI_{/j}, \]
there exists a natural transformation $X \Rightarrow \cG$ where $X:\cT \lrar \cI_{/j}$ is a constant functor. Now let $\cG: \cT \lrar \cI_{/j}$ be such a map. The data of $\cG$ can be equivalently described as a pair $(\cG_{\cI},\cG_{\cJ})$ where $\cG_{\cI}: \cT \lrar \cI$ is a functor and $\cG_{\cJ}: \cT^{\triangleright} \lrar \cJ$ is a functor sending the cone point to $j$ and such that ${\cG_{\cJ}}|_{\cT} = \cF \circ \cG_{\cI}$. For each $t \in \cT$ let us denote by $\alp_t: \cG_{\cJ}(t) \lrar j$ the map determined by $\cG_{\cJ}$.

By Lemma~\ref{l:finite} there exists an extension $\ovl{\cG}_{\cI}:\cT^{\triangleleft} \lrar \cI$. Let $i_0 \in \cI$ be the image of the cone point of $\cT^{\triangleleft}$ under $\ovl{\cG}_{\cI}$ and for each $t$ let $\bet_t: i_0 \lrar \cG_{\cI}(t)$ be the map determined by $\ovl{\cG}_{\cI}$. Let $\gam_t: \cF(i_0) \lrar j$ be the map obtained by composing the map $\cF(\bet_t): \cF(i_0) \lrar \cF(\cG_{\cI}(t)) = \cG_{\cJ}(t)$ and the map $\alp_t: \cG_{\cJ}(t) \lrar j$. We claim that the maps $\gam_t$ are all identical. Indeed, if $t > s$ then the commutativity of he diagram
\[ \xymatrix{
& \cG_{\cJ}(t) \ar^{\alp_t}[dr]\ar[dd] & \\
\cF(i_0) \ar^{\cF(\bet_t)}[ur]\ar_{\cF(\bet_s)}[dr] & & j \\
& \cG_{\cJ}(s) \ar_{\alp_s}[ur] & \\
}\]
shows that $\gam_t = \gam_s$. Since $\cT$ is connected it follows that $\gam_t = \gam_s$ for every $t,s \in \cT$. Let us call this map $\gam: \cF(i_0) \lrar j$. Then the pair $(i_0,\gam)$ corresponds to an object $X \in \cI_{/j}$, which can be interpreted as a constant functor $X:\cT \lrar \cI_{/j}$. The maps $\bet_t$ now determine a natural transformation $X \Rightarrow \cG$ as desired.
\end{proof}

\begin{lem}\label{l:coinit-cofiltered}
Let $\cF: \cI \lrar \cJ$ be a coinitial functor. If $\cI$ is cofiltered then $\cJ$ is cofiltered.
\end{lem}
\begin{proof}
See~\cite[Lemma 3.12]{BS15}.
\end{proof}

\begin{define}\label{d:pro-discrete}
Let $\cC$ be a category. We define $\Pro(\cC)$ to be the category whose objects are diagrams $X:\cI \lrar \cC$ such that $\cI$ is small and cofiltered (see Definition~\ref{d:cofiltered}) and whose morphism sets are given by
\[\Hom_{\Pro(\cC)}(X,Y):=\lim\limits_s \colim\limits_t \Hom_{\cC}(X_t,Y_s).\]
Composition of morphisms is defined in the obvious way. We refer to $\Pro(\cC)$ as the \textbf{pro-category} of $\cC$ and to objects of $\Pro(\cC)$ as \textbf{pro-objects}.
\end{define}

A pro-object $X: \cI \lrar \cC$ will often be written as $X = \{X_i\}_{i \in \cI}$ where $X_i = X(i)$. There is a canonical full inclusion $\iota: \cC \lrar \Pro(\cC)$ which associates to $X \in \cC$ the constant diagram with value $X$, indexed by the trivial category. We say that a pro-object is \textbf{simple} if it is in the image of $\iota$. Given a pro-object $X = \{X_i\}_{i \in \cI}$ and a functor $p: \cJ \lrar \cI$ we will denote by $p^*X \x{\df}{=} X \circ p$ the restriction (or reindexing) of $X$ along $p$.

If $X,Y: \cI \lrar \cC$ are two pro-objects indexed by $\cI$ then any natural transformation: $X \lrar Y$ gives rise to a morphism $X \lrar Y$ in $\Pro(\cC)$. More generally, for pro-object $X = \{X_i\}_{i \in \cI}, Y = \{Y_j\}_{j \in \cJ}$ if $p:\cJ \lrar \cI$ is a functor and $\phi:p^*X \lrar Y$ is a map in $\cC^{\cJ}$, then the pair $(p,\phi)$ determines a morphism $\nu_{p,\phi}:X \lrar Y$ in $\Pro(\cC)$ (whose image in $\colim\limits_t \Hom_{\cC}(X_t,Y_s)$ is given by $\phi_s:X_{p(s)} \lrar Y_s$).

The following special case of the above construction is well-known.
\begin{lem}\label{l:cofinal}
Let $p:\cJ \lrar \cI$ be a coinitial functor between small cofiltered categories, and let $X = \{X_i\}_{i \in \cI}$ be a pro-object indexed by $\cI$. Then the morphism of pro-objects $\nu_{p,\Id}:X\lrar p^*X$ determined by $p$ is an \textbf{isomorphism}. For the purpose of brevity we will denote $\nu_p \x{\df}{=} \nu_{p,\Id}$.
\end{lem}

\begin{proof}
For any pro-object $X = \{X_i\}_{i \in \cI}$, the maps $X \lrar X_i$ exhibit $X$ as the limit, in $\Pro(\cC)$, of the diagram $i \mapsto X_i$. Since restriction along coinitial maps preserves limits (see Theorem~\ref{t:cofinal}) it follows that the induced map $\nu_{p}:X\lrar p^*X$ is an isomorphism.
\end{proof}

\begin{define}\label{d:reindex}
We refer to the isomorphisms $\nu_{p}:X\lrar p^*X$ described in Lemma~\ref{l:cofinal} as \textbf{reindexing isomorphisms}.
\end{define}

\begin{define}\label{def CDS-2}
Let $\cT$ be a small poset. We say that $\cT$ is \textbf{inverse} if it is both cofinite and cofiltered.
\end{define}

The following lemma (and variants thereof) is quite standard.
\begin{lem}\label{l:cofinal_CDS}
Let $\cI$ be a small cofiltered category. Then there exists an inverse poset $\cT$ (see Definition~\ref{def CDS-2}) and a coinitial functor $p:\cT \lrar \cI$.
\end{lem}
\begin{proof}
A proof of this can be found, for example, in~\cite[Proposition 5.3.1.16]{Lur09}.
\end{proof}

\begin{cor}\label{c:index-poset}
Any pro-object is isomorphic to a pro-object which is indexed by an inverse poset.
\end{cor}

Although not every map of pro-objects is induced by a natural transformation, it is always isomorphic to one. More specifically, we recall the following lemma:

\begin{lem}\label{l:every map natural}
Let $f: \{Z_i\}_{i\in\cI} \lrar \{X_j\}_{j\in\cJ}$ be a map in $\Pro(\cC)$. Then there exists a cofiltered category $\cT$, coinitial functors $p:\cT\to\cI$ and $q:\cT\to\cJ$, a natural transformation $p^*Z\to q^*X$ and a commutative square in $\Pro(\cC)$ of the form
\[\xymatrix{Z\ar[r]^f \ar[d]_{\nu_{p}} & X \ar[d]^{\nu_{q}}\\
             p^*Z\ar[r] & q^*X.}\]
\end{lem}

\begin{proof}
This is shown in ~\cite[Appendix 3.2]{AM69}.
\end{proof}

Corollary~\ref{c:index-poset} demonstrates that isomorphic pro-objects might have non isomorphic indexing categories. Thus the assignment of the indexing category to every pro-object is non-functorial. It is often useful to assign functorially a ``canonical indexing category'' to every pro-object. This will be done in Definition~\ref{d:canonical index}.

Let $\cC$ be a category, $X = \{X_i\}_{i \in \cI} \in \Pro(\cC)$ a pro-object and $f: X \lrar Y$ a map in $\Pro(\cC)$ where $Y \in \cC \subseteq \Pro(\cC)$ a simple object. Let $\cH:\cI^{\op} \lrar \Set$ be the functor which associates to $i \in \cI$ the set of maps $g: X_i \lrar Y$ such that the composite $X \lrar X_i \x{g}{\lrar} Y$ is equal to $f$. Let $\cG(\cI^{\op},\cH)$ the Grothendieck construction of $\cH$.

\begin{lem}\label{l:lem-1}
The category $\cG(\cI^{\op},\cH)$ is weakly contractible
\end{lem}

\begin{proof}
By the main result of~\cite{Th79} the nerve of $\cG(\cI^{\op},\cH)$ is a model for the homotopy colimit of the functor $\cH: \cI^{\op} \lrar \Set$. Since $\cI^{\op}$ is filtered this homotopy colimit it weakly equivalent to the actual colimit of the diagram. It will hence suffice to show that $\colim_{i \in \cI^{\op}} \cH = \ast$. Now the functor $\cH$ fits into a Cartesian square of functors of the form
\[\xymatrix{
\cH(i) \ar[r]\ar[d] & \Hom_{\cC}(X_i, Y) \ar[d] \\
\ast \ar[r] & \Hom_{\Pro(\cC)}(X,Y) \\
}
\]
where the image of the bottom horizontal map is the point $f \in \Hom_{\Pro(\cC)}(X,Y)$. Since $\cI^{\op}$ is filtered the square
\[ \xymatrix{
\displaystyle\mathop{\colim}_{i \in \cI^{\op}}\cH(i) \ar[r]\ar[d] & \displaystyle\mathop{\colim}_{i \in \cI^{\op}}\Hom_{\cC}(X_i, Y) \ar[d] \\
\ast \ar[r] & \Hom_{\Pro(\cC)}(X,Y) \\
}
\]
is a Cartesian square of sets (see~\cite[Theorem 9.5.2]{Sch72}). By definition of $\Pro(\cC)$ the right vertical map is an isomorphism. It follows that the left vertical map is an isomorphism as well as desired.
\end{proof}

\begin{cor}\label{c:pro}
Let $\cC$ be a category and $X = \{X_i\}_{i \in \cI} \in \Pro(\cC)$ a pro-object. Then the natural functor
\[ \cI \lrar \cC_{X/} \]
is coinitial and $\cC_{X/}$ is cofiltered. In particular, if $\cC$ is small then the pro-object $\cC_{X/} \lrar \cC$ given by $(X \lrar Y) \mapsto Y$ is naturally isomorphic to $X$.
\end{cor}
\begin{proof}
Combine Lemmas~\ref{l:lem-1},~\ref{l:cofinal} and~\ref{l:coinit-cofiltered}.
\end{proof}

\begin{define}\label{d:canonical index}
Let $X = \{X_i\}_{i \in \cI} \in \Pro(\cC)$ be a pro-object.
We refer to $\cC_{X/}$ as the \textbf{canonical indexing category} of $X$ and to $\cJ$ as the \textbf{actual indexing category} of $X$.
\end{define}
\subsection{Pro-categories in higher category theory}

In~\cite{Lur09} Lurie defined pro-categories for \textbf{small} $\infty$-categories and in~\cite{Lur11b} the definition was adjusted to accommodate \textbf{accessible} $\infty$-categories which admit finite limits (such $\infty$-categories are typically not small). The purpose of this subsection is to extend these definitions to the setting of general \textbf{locally small} $\infty$-categories (see Definition \ref{d:locally-small-2}).
It is from this point on in the paper that set theoretical issues of ``largeness" and ``smallness" begin to play a more important role, and the interested reader might want to go back to Section \ref{ss:large} to recall our setting and terminology.

We denote by $\Fun_{\sm}(\cC,\cD) \subseteq \Fun(\cC,\cD)$ the full subcategory spanned by small functors (see Definition~\ref{d:small functor}).

\begin{lem}\label{l:locally small}
If $\cD$ is locally small, then the $\infty$-category $\Fun_{\sm}(\cC,\cD)$ is locally small.
\end{lem}
\begin{proof}
Let $f,g: \cC \lrar \cD$ be two small functors. Then there exists a small full-subcategory $\cC_0 \subseteq \cC$ such that both $f,g$ are left Kan extended from $\cC_0$. Then
\[ \Map_{\Fun(\cC,\cD)}(f,g) \simeq \Map_{\Fun(\cC_0,\cD)}(f|_{\cC_0},g|_{\cC_0}) \]
and the latter space is small.
\end{proof}

\begin{lem}\label{l:colimit}
The full subcategory $\Fun_{\sm}(\cC,\cD) \subseteq \Fun(\cC,\cD)$ is closed under small colimits.
\end{lem}
\begin{proof}
Given a family of small functors $f_i:\cC \lrar \cD$ indexed by a small $\infty$-category $\cI$, we may find a small full subcategory $\cC_0 \subseteq \cC$ such that $f_i$ is a left Kan extension of $f_i|_{\cC_0}$ for every $i \in \cI$. Since left Kan extension commute with colimits it follows that $\colim_i f_i$ is a left Kan extension of $\colim_i f_i|_{\cC_0}$.
\end{proof}

\begin{lem}\label{l:small rep}
If $f:\cC \lrar \ovl{\cS}_\infty$ (see Definition~\ref{d:spaces}) is a small colimit of corepresentable functors then $f$ is small. The converse holds if $f$ takes values in essentially small spaces.
\end{lem}
\begin{proof}
Suppose $f$ is corepresentable by $c \in \cC$. Then $f$ is a left Kan extension of the functor $\Del^0\lrar \ovl{\cS}_\infty$ which sends the object of $\Del^0$ to the terminal space along the map $\Del^0 \lrar \cC$ which sends the object of $\Del^0$ to $c$. Thus, by Remark \ref{r:equiv small}, $f$ is small. By Lemma~\ref{l:colimit} every small colimit of corepresentable functors is small.

Now suppose that $f$ is small and takes values in essentially small spaces. Let $\widetilde{\cC} \lrar \cC$ be a left fibration classifying $f$. Since $f$ is small there exists a small full subcategory $\cC_0 \subseteq \cC$ such that $f$ is a left Kan extension of $g = f|_{\cC_0}$. Let $\widetilde{\cC}_0 = \widetilde{\cC} \times_{\cC} \cC_0$. Then the left fibration $\widetilde{\cC}_0 \lrar \cC_0$ classifies $g$ and by the straightening unstraightening equivalence of~\cite[Theorem 2.2.1.2]{Lur09} it follows that $g$ can be identified with the colimit of the composition
\[ \widetilde{\cC}_0^{\op} \lrar \cC_0^{\op} \lrar \Fun\left(\cC_0,\ovl{\cS}_\infty\right) \] where the second map is the Yoneda embedding of $\cC_0^{\op}$. Since $f$ is a left Kan extension of $g$ we may identify $f$ with the colimit in $\Fun(\cC,\ovl{\cS}_\infty)$ of the composed map
\[ \widetilde{\cC}_0^{\op} \lrar \cC^{\op} \lrar \Fun(\cC,\ovl{\cS}_\infty). \]
Since $\widetilde{\cC}_0\lrar\cC_0$ is a left fibration classifying a functor $g: \cC_0 \lrar\ovl{\cS}_\infty$ which has a small domain and takes values in essentially small spaces it follows from the straightening unstraightening equivalence that $\widetilde{\cC}_0$ is essentially small. Thus we can replace it with an equivalent small $\infty$-category and so the proof is complete.
\end{proof}

Let us recall the higher categorical analogue of Definition~\ref{d:cofiltered}.

\begin{define}[{\cite[Definition 5.3.1.7]{Lur09}}]\label{d:oo-cofiltered}
Let $\cC$ be an $\infty$-category. We say that $\cC$ is \textbf{cofiltered} if for every map $f: K \lrar \cC$ where $K$ is a simplicial set with finitely many non-degenerate simplices, there exists an extension of the form $\ovl{f}: K^{\triangleleft} \lrar \cC$.
\end{define}

\begin{rem}
For ordinary categories Definition~\ref{d:oo-cofiltered} and Definition~\ref{d:cofiltered} coincide. This follows from Lemma~\ref{l:finite}.
\end{rem}

We begin by establishing the following useful lemma:
\begin{lem}\label{l:full-oo}
Let $\cC$ be a cofiltered $\infty$-category and let $\cD \subseteq \cC$ be a full subcategory such that for every $c \in \cC$ the category $\cD_{/c}$ is non-empty. Then $\cD$ is cofiltered and the inclusion $\cD \subseteq \cC$ is coinitial.
\end{lem}
\begin{proof}
Let $K$ be a simplicial set with finitely many non-degenerate simplices and let $p: K \lrar \cD$ be a map. Consider the right fibration $\cD_{/p} \lrar \cD$. We need to show that $\cD_{/p}$ is not empty. Let $q: K \lrar \cC$ be the composition of $p$ with the full inclusion $\cD \subseteq \cC$. Since $\cC$ is cofiltered the $\infty$-category $\cC_{/q}$ is non-empty. Since the inclusion $\cD \subseteq \cC$ is full the square
\[ \xymatrix{
\cD_{/p} \ar[r]\ar[d] & \cC_{/q} \ar[d] \\
\cD \ar[r] & \cC \\
}\]
is Cartesian. It will hence suffice to show that there exists a $d \in \cD$ such that the fiber $\cC_{/q} \times_{\cC} \{d\}$ is non-empty. Now let $x \in \cC_{/q}$ be an element whose image in $\cC$ is $c \in \cC$. By our assumptions there exists a map of the form $d \lrar c$ with $d \in \cD$. Since $\cC_{/q}$ is a right fibration there exists an arrow $y \lrar x$ in $\cC_{/q}$ such that the image of $y$ in $\cC$ is $d$. Hence $\cC_{/q} \times_{\cC} \{d\} \neq \emptyset$ and we may conclude that $\cD$ is cofiltered.

Let us now show that the inclusion $\cD \subseteq \cC$ is coinitial. Let $c \in \cC$ be an object. Then the inclusion $\cD_{/c} \hrar \cC_{/c}$ is fully-faithful. Furthermore, for every map $f:c' \lrar c$, considered as an object $f \in \cC_{/c}$, the $\infty$-category $(\cD_{/c})_{/f}$ is equivalent to the $\infty$-category $\cD_{/c'}$ and is hence non-empty. By~\cite[Lemma 5.3.1.19]{Lur09} the $\infty$-category $\cC_{/c}$ is cofiltered. Applying again the argument above to the inclusion $\cD_{/c} \hrar \cC_{/c}$ we conclude that $\cD_{/c}$ is cofiltered, and is hence weakly contractible by~\cite[Lemma 5.3.1.18]{Lur09}.
\end{proof}

We now turn to the main definition of this subsection.

\begin{define}\label{d:pro}
Let $\cC$ be a locally small $\infty$-category. We say that a functor $f:\cC \lrar \ovl{\cS}_\infty$ is a \textbf{pro-object} if $f$ is small, takes values in essentially small spaces, and is classified by a left fibration $\widetilde{\cC} \lrar \cC$ such that $\widetilde{\cC}$ is cofiltered. We denote by $\Pro(\cC) \subseteq \Fun\left(\cC,\ovl{\cS}_\infty\right)^{\op}$ the full subcategory spanned by pro-objects.
\end{define}

\begin{rem}
If $\cC$ is a small $\infty$-category then Definition~\ref{d:pro} reduces to~\cite[Definition 5.3.5.1]{Lur09}.
\end{rem}

\begin{rem}
By definition the essential image of $\Pro(\cC)$ in $\Fun(\cC,\ovl{\cS}_\infty)$ is contained in the essential image of $\Fun_{\sm}\left(\cC,\cS_\infty\right) \subseteq \Fun_{\sm}\left(\cC,\ovl{\cS}_\infty\right)$. It follows by Lemma \ref{l:locally small} that $\Pro(\cC)$ is locally small.
\end{rem}

\begin{lem}
Any corepresentable functor $f:\cC \lrar \ovl{\cS}_\infty$ is a pro-object.
\end{lem}
\begin{proof}
By Lemma \ref{l:small rep} we know that $f$ is small, and since $\cC$ is locally small $f$ takes values in essentially small spaces.
Let $\widetilde{\cC} \lrar \cC$ be the left fibration classifying $f$. If $f$ is corepresentable by $c \in \cC$ then $\widetilde{\cC} \simeq \cC_{c/}$ has an initial object and is thus cofiltered by \cite[Proposition 5.3.1.15]{Lur09}.
\end{proof}

\begin{define}
By the previous lemma we see that the Yoneda embedding $\cC \hrar \Fun\left(\cC,\ovl{\cS}_\infty\right)^{\op}$ factors through $\Pro(\cC)$, and we denote it by $\iota_{\cC}: \cC \hrar \Pro(\cC)$. We say that a pro-object is \textbf{simple} if it belongs to the essential image of $\iota_{\cC}$.
\end{define}

\begin{lem}\label{l:generate}
Every pro-object is a small cofiltered limit of simple objects.
\end{lem}

\begin{proof}
Let $f: \cC \lrar \ovl{\cS}_\infty$ be a pro-object.
We define $\widetilde{\cC}$, $\cC_0$, $g$ and $\widetilde{\cC}_0$ as in the second part of the proof of Lemma \ref{l:small rep}. As we have shown there, $\widetilde{\cC}_0$ is (essentially) small and we may identify $f$ with the colimit in $\Fun(\cC,\ovl{\cS}_\infty)$ of the composed map
\[ \widetilde{\cC}_0^{\op} \lrar \cC^{\op} \lrar \Fun(\cC,\ovl{\cS}_\infty), \]
where the second map is the Yoneda embedding.
Thus $f$ can be identified with the \textbf{limit} in $\Pro(\cC)$ of the composed map
\[ \widetilde{\cC}_0 \lrar \cC \lrar \Pro(\cC). \]
It will hence suffice to show that $\widetilde{\cC}_0$ is cofiltered. Since $f$ is a pro-object the $\infty$-category $\widetilde{\cC}$ is cofiltered by Definition. Since $f$ is a left Kan extension of $g$ it follows that for every $c \in \widetilde{\cC}$ the category $\left(\widetilde{\cC}_0\right)_{/c}$ is \textbf{non-empty}. The desired result now follows form Lemma~\ref{l:full-oo}.
\end{proof}

\begin{lem}\label{l:closed-filtered}
The full subcategory $\Pro(\cC) \subseteq \Fun\left(\cC,\ovl{\cS}_\infty\right)^{\op}$ is closed under small cofiltered limits.
\end{lem}
\begin{proof}
The same proof as~\cite[Proposition 5.3.5.3]{Lur09} can be applied here, using Lemma~\ref{l:colimit}.
\end{proof}

\begin{cor}\label{c:universal}
The full subcategory $\Pro(\cC) \subseteq \Fun\left(\cC,\ovl{\cS}_\infty\right)^{\op}$ is the smallest one containing the essential image of $\iota_{\cC}$ and closed under small cofiltered limits.
\end{cor}
\begin{proof}
This follows from Lemma~\ref{l:generate} and Lemma~\ref{l:closed-filtered}.
\end{proof}

\begin{rem}\label{r:lurie}
If $\cC$ is an accessible $\infty$-category which admits finite limits, then $\Pro(\cC)$ as defined above coincides with the pro-category defined in~\cite[Definition 3.1.1]{Lur11b}, namely, $\Pro(\cC) \subseteq \Fun(\cC,\ovl{\cS}_\infty)^{\op}$ is the full subcategory spanned by \textbf{accessible functors which preserve finite limits}. This follows from the fact that they both satisfy the characterization of Corollary~\ref{c:universal} (see the proof of~\cite[Proposition 3.1.6]{Lur11b}).
\end{rem}

\begin{define}\label{d:cocompact}
Let $\cC$ be an $\infty$-category. We say that $X \in \cC$ is $\omega$-\textbf{cocompact} if the functor $\cC \lrar \ovl{\cS}_\infty$ corepresented by $X$ preserves cofiltered limits.
\end{define}

\begin{lem}\label{l:cocompact}
Let $X \in \Pro(\cC)$ be a simple object. Then $X$ is $\omega$-\textbf{cocompact}.
\end{lem}

\begin{proof}
By Lemma~\ref{l:closed-filtered} it will suffice to show that $X$ is $\omega$-cocompact when considered as an object of $\Fun(\cC,\ovl{\cS}_\infty)^{\op}$. But this now follows from~\cite[Proposition 5.1.6.8]{Lur09} in light of our large cardinal axiom.
\end{proof}

We now wish to show that if $\cC$ is an \textbf{ordinary category} then Definition~\ref{d:pro} coincides with Definition~\ref{d:pro-discrete} up to a natural equivalence. For this purpose we let $\Pro(\cC)$ denote the category defined in \ref{d:pro-discrete}. For each pro-object $X = \{X\}_{i \in \cI} \in \Pro(\cC)$ we may consider the associated functor $R_X: \cC \lrar \Set$ given by
\[ R_X(Y) = \Hom_{\Pro(\cC)}(X,Y) = \colim_{i \in \cI}\Hom_\cC(X_i,Y) .\]
The equivalence of Definition~\ref{d:pro-discrete} and~\ref{d:pro} for $\cC$ now follows from the following proposition:
\begin{prop}\label{p:characterization}
The association $X \mapsto R_X$ determines a fully-faithful embedding $\iota:\Pro(\cC) \hrar \Fun(\cC,\Set)^{\op}$. A functor $\cF: \cC \lrar \Set$ belongs to the essential image of $\iota$ if and only if $\cF$ is small and its Grothendieck construction $\cG(\cC,\cF)$ is cofiltered.
\end{prop}

\begin{proof}
The fact that $X \mapsto R_X$ is fully-faithful follows from the fact that the Yoneda embedding $\cC\to \Fun(\cC,\Set)^{\op}$ is fully faithful and lands in the subcategory of $\Fun(\cC,\Set)^{\op}$ spanned by $\omega$-cocompact objects.

Now let $X = \{X_i\}_{i \in \cI}$ be a pro-object. Then $R_X: \cC \lrar \Set$ is a small colimit of corepresentable functors and is hence small by Lemma \ref{l:small rep}. The Grothendieck construction of $R_X$ can naturally be identified with $\cC_{X/}$ and is hence cofiltered by Corollary~\ref{c:pro}.

On the other hand, let $\cF: \cC \lrar \Set$ be a small functor such that $\cG(\cC,\cF)$ is cofiltered. Let $\cC_0 \subseteq \cC$ be a small full subcategory such that $\cF$ is a left Kan extension of $\cF|_{\cC_0}$ and let $\cI = \cG(\cC_0,\cF|_{\cC_0})$ be the associated Grothendieck construction. Since the inclusion $\cC_0 \subseteq \cC$ is fully-faithful, the induced map $ \cI \lrar \cG(\cC,\cF) $ is fully faithful. Now let $(c,x) \in \cG(\cC,\cF)$ be an object, so that $x$ is an element of $f(c)$. Since $\cF$ is a left Kan extension of $\cF|_{\cC_0}$ there exists a map $\alp:c_0 \lrar c$ with $c_0 \in \cC_0$ and an element $y \in \cF(c_0)$ such that $\cF(\alp)(y) = x$. This implies that $\alp$ lifts to a map $(c_0,y) \lrar (c,x)$ in $\cG(\cC,\cF)$. It follows that for every $(c,x) \in \cG(\cC,\cF)$ the category $\cI_{/(c,x)}$ is non-empty. By Lemma~\ref{l:full-oo} we get that $\cI$ is cofiltered and the inclusion $\cI \hrar \cG(\cC,\cF)$ is coinitial. Let $X = \{X_i\}_{i \in \cI}$ be the pro-object corresponding to the composed map $\cI \lrar \cC_0 \hrar \cC$. We now claim that $R_X$ is naturally isomorphic to $\cF$. Let $Y \in \cC$ be an object and choose a full subcategory $\cC_0' \subseteq \cC_0$ which contains both $\cC_0$ and $Y$. Let $\cI' = \cG(\cC_0',\cF|_{\cC_0'})$ be the associated Grothendieck construction. Then $\cF|_{\cC_0'}$ is a left Kan extension of $\cF|_{\cC_0}$ and $\cF$ is a left Kan extension of $\cF|_{\cC_0'}$. By the arguments above $\cI'$ is cofiltered and the functor $ \cI \lrar \cI' $ is coinitial. Let $X = \{X'_{i'}\}_{i' \in \cI'}$ be the pro-object corresponding to the composed map $\cI' \lrar \cC_0' \hrar \cC$. We then have natural isomorphisms
\[ \colim_{i \in \cI^{\op}} \Hom_{\cC}(X_i,Y)
\cong \colim_{i' \in (\cI')^{\op}} \Hom_{\cC_0'}(X'_{i'},Y) \cong \int_{c_0' \in \cC_0'} \cF(c_0') \times \Hom_{\cC'_0}(c_0',Y) \cong \cF(Y) \]
\end{proof}

We finish this subsection by verifying that $\Pro(\cC)$ satisfies the expected universal property (compare~\cite[Proposition 3.1.6]{Lur11b}):

\begin{thm}\label{t:universal}
Let $\cC$ be a locally small $\infty$-category and let $\cD$ be a locally small $\infty$-category which admits small cofiltered limits. Let $\Fun_{\mathrm{cofil}}(\cC,\cD) \subseteq \Fun(\cC,\cD)$ denote the full subcategory spanned by those functors which preserve small cofiltered limits. Then composition with the Yoneda embedding restricts to an equivalence of $\infty$-categories
\begin{equation}\label{e:univer}
\Fun_{\mathrm{cofil}}(\Pro(\cC),\cD) \x{\simeq}{\lrar} \Fun(\cC,\cD)
\end{equation}
\end{thm}

\begin{proof}
This is a particular case of~\cite[Proposition 5.3.6.2]{Lur09} where $\cK$ is the family of small cofiltered simplicial sets and $\cR$ is empty. Note that \cite[Proposition 5.3.6.2]{Lur09} is stated for a small $\infty$-category $\cC$ (in the terminology of loc. cit.) and makes use of the $\infty$-category $\cS_\infty$ of small spaces. In light of our large cardinal axiom~\ref{a:large-cardinal} we may replace $\cS_\infty$ with $\ovl{\cS}_\infty$ and apply~\cite[Proposition 5.3.6.2]{Lur09} to the $\infty$-category $\cC^{\op}$. The fact that the $\infty$-category $\cP^{\cK}(\cC)$ constructed in the proof of~\cite[Proposition 5.3.6.2]{Lur09} coincides with $\Pro(\cC)$ follows from Corollary~\ref{c:universal}.
\end{proof}

The universal property~\ref{t:universal} allows, in particular, to define the \textbf{prolongation} of functors in the setting of $\infty$-categories.
\begin{define}\label{d:prolong}
Let $f: \cC \lrar \cD$ be a map of locally small $\infty$-categories. A \textbf{prolongation} of $f$ is a cofiltered limit preserving functor $\Pro(f): \Pro(\cC) \lrar \Pro(\cD)$, together with an equivalence $u: \Pro(f)|_{\cC} \simeq \iota_{\cD} \circ f$ (where $\iota_{\cD}: \cD \hrar \Pro(\cD)$ is the full embedding of simple objects). By Theorem~\ref{t:universal} we see that a prolongation $(\Pro(f),u)$ is unique up to a contractible choice.
\end{define}

\section{The induced model structure on $\Pro(\cC)$}\label{s:pro-model}

\subsection{Definition}

In this subsection we define what we mean for a model structure on $\Pro(\cC)$ to be \textbf{induced} by a weak fibration structure on $\cC$. Sufficient hypothesis on $\cC$ for this procedure to be possible appear in~\cite{EH76,Isa04,BS1,BS2}. We begin by establishing some useful terminology.

\begin{define}\label{def mor}
Let $\cC$ be a category with finite limits, and $\cM$ a class of morphisms in $\cC$. Denote by:
\begin{enumerate}
\item $\R(\cM)$ the class of morphisms in $\cC$ that are retracts of morphisms in $\cM$.
\item ${}^{\perp}\cM$ the class of morphisms in $\cC$ with the left lifting property against any morphism in $\cM$.
\item $\cM^{\perp}$ the class of morphisms in $\cC$ with the right lifting property against any morphism in $\cM$.
\item $\Lw^{\cong}(\cM)$ the class of morphisms in $\Pro(\cC)$ that are isomorphic to a levelwise $\cM$-map (see Definition~\ref{def natural}).
\item $\Sp^{\cong}(\cM)$ the class of morphisms in $\Pro(\cC)$ that are isomorphic to a special $\cM$-map (see Definition~\ref{def natural}).
\end{enumerate}
\end{define}

\begin{lem}[{\cite[Proposition 2.2]{Isa04}}]\label{l:ret_lw}
Let $\cM$ be any class of morphisms in $\mcal{C}$. Then $\R(\Lw^{\cong}(\cM)) = \Lw^{\cong}(\cM)$.
\end{lem}

\begin{define}\label{d:induce}
Let $(\cC,\cW,\cF ib)$ be a weak fibration category. We say that a model structure $(\Pro(\cC),\mathbf{W}, \mathbf{Cof}, \mathbf{Fib})$ on $\Pro(\cC)$ is \textbf{induced} from $\cC$ if the following conditions are satisfied:
\begin{enumerate}
\item The cofibrations are $\mathbf{Cof} =  {}^{\perp} (\Fib\cap \cW)$.
\item The trivial cofibrations are $\mathbf{Cof}\cap\mathbf{W} = {}^{\perp} \Fib$.
\item If $f:Z \lrar X$ is a morphism in $\cC^\cT$, with $\cT$ a cofiltered category, then there exists a cofiltered category $\cJ$, a coinitial functor $\mu:\cJ\lrar \cT$ and a factorization
\[ \mu^*Z\xrightarrow{g} Y\xrightarrow{h}\mu^*X\]
in $\cC^{\cJ}$ of the map $\mu^*f: \mu^*Z \lrar \mu^*X$ such that $g$ is a cofibration in $\Pro(\cC)$ and $h$ is both a trivial fibration in $\Pro(\cC)$  and a levelwise trivial fibration.
\end{enumerate}
\end{define}

Since a model structure is determined by its cofibrations and trivial cofibrations, we see that the induced model structure is unique if it exists.

We now recall some terminology from \cite{BS1}.

\begin{define}
Let $(\cC,\cW,\Fib)$ be a weak fibration category and $\cC_s \subseteq \cC$ a full-subcategory which is closed under finite limits. We say that $\cC_s$ is a \textbf{full weak fibration subcategory of $\cC$} if $(\cC_s,\cW \cap \cC_s,\Fib \cap \cC_s)$ satisfies the axioms of a weak fibration category.
\end{define}

\begin{define}\label{d:dense}
Let $(\cC,\cW,\Fib)$ be a weak fibration category and $\cC_s \subseteq \cC$ a full weak fibration subcategory of $\cC$. We say that $\cC_s$ is \textbf{dense} if the following condition is satisfied: if $X \x{f}{\lrar} H \x{g}{\lrar} Y$ is a pair of composable morphisms in $\cC$ such that $X,Y \in \cC_s$ and $g$ is a fibration (resp. trivial fibration) then there exists a diagram of the form
\[ \xymatrix{
& H'\ar^{g'}[rd]\ar[dd] & \\
X \ar^{f'}[ur]\ar^{f}[dr] & & Y \\
& H \ar^{g}[ur] & \\
}\]
such that $g'$ is a fibration (resp. trivial fibration)  and $H'\in \cC_s$.
\end{define}

\begin{define}\label{d:locally-small}
Let $\cC$ be a weak fibration category. We say that $\cC$ is \textbf{homotopically small} if for every map of the form $f: \cI \lrar \cC$ where $\cI$ is a small cofiltered category, there exists a dense \textbf{essentially small} weak fibration subcategory $\cC_s \subseteq \cC$ such that the image of $f$ is contained in $\cC_s$.
\end{define}

\begin{rem}
  Note that any essentially small weak fibration category is clearly homotopically small.
\end{rem}

\begin{prop}\label{p:induce}
Let $(\cC,\cW,\cF ib)$ be a weak fibration category. Suppose that there exists a model structure on $\Pro(\cC)$  such that the following conditions are satisfied:
\begin{enumerate}
\item The cofibrations are ${}^{\perp} (\Fib\cap \cW)$.
\item The trivial cofibrations are ${}^{\perp} \Fib$.
\end{enumerate}
If $\cC$ is a model category or is homotopically small then the model structure on $\Pro(\cC)$ is induced from $\cC$ in the sense of Definition \ref{d:induce}.
\end{prop}

\begin{proof}
We only need to verify that Condition (3) in Definition \ref{d:induce} is satisfied.

Suppose that $\cC$ is a model category. Let $f:Z \lrar X$ be a morphism in $\cC^\cT$, with $\cT$ a cofiltered category. Choose an inverse poset $\cA$ with a coinitial functor $\mu:\cA\lrar \cT$ and consider the induced map $\mu^*f: \mu^*Z \lrar \mu^*X$. Since $\cF ib\cap\cW$ and $\cC of$ are classes of morphisms satisfying $(\cF ib\cap\cW) \circ \cC of = \Mor(\cC)$ we may employ the construction described in~\cite[Definition 4.3]{BS15} to factor $\mu^*f$ in $\cC^\cA$ as
\[ \mu^*Z\xrightarrow{\Lw(\cC of)} Y\xrightarrow{\Sp(\cF ib\cap\cW)}\mu^*X.\]

By \cite[Proposition 4.1]{BS15} we have that
$$\Lw(\cC of)= {}^{\perp}\R(\Sp(\cF ib\cap\cW)= {}^{\perp}(\cF ib\cap\cW).$$
$$\R(\Sp(\cF ib\cap\cW)=\Lw(\cC of)^{\perp}=({}^{\perp}(\cF ib\cap\cW))^{\perp}.$$
Thus, the first map is a cofibration and the second map is both a trivial fibration and a levelwise trivial fibration (see Proposition~\ref{p:forF_sp_is_lw_strong}).

Now suppose that $\cC$ is homotopically small. Let $f:Z \lrar X$ be a morphism in $\cC^\cT$, with $\cT$ a cofiltered category. Following the proof of ~\cite[Proposition 3.15]{BS1}, we can find an inverse poset $\cA$ equipped with a coinitial functor $\mu:\cA \lrar \cT$, together with a factorization of $\mu^*f$ in $\cC^\cA$ as
\[\mu^*Z\xrightarrow{{}^{\perp}(\cF ib\cap\cW)} Z'\xrightarrow{\Sp(\cF ib\cap\cW)}\mu^*X.\]

By \cite[Proposition 5.10]{BS2} we have that
$$\Sp(\cF ib\cap\cW)\subseteq \cocell(\cF ib\cap\cW)\subseteq ({}^{\perp}(\cF ib\cap\cW))^{\perp}.$$

Thus the first map is a cofibration and the second map is both a trivial fibration and a levelwise trivial fibration (see Proposition~\ref{p:forF_sp_is_lw_strong}).
\end{proof}

\subsection{Existence results}

We shall now describe sufficient conditions on $\cC$ which insure the existence of an induced model structure on $\Pro(\cC)$.

We denote by $[1]$ the category consisting of two objects and one non-identity morphism between them. Thus, if $\cC$ is any category, the functor category $\cC^{[1]}$ is just the category of morphisms in $\cC$.

\begin{define}\label{d:admiss}
A relative category $(\cC,\cW)$ is called \textbf{pro-admissible} if $\Lw^{\cong}(\cW)\subseteq \Pro(\cC)^{[1]}$ satisfies the 2-out-of-3 property.
\end{define}

\begin{lem}[Isaksen]\label{l:proper}
Let $\cM$ be a proper model category. Then $(\cM,\cW)$ is pro-admissible.
\end{lem}
\begin{proof}
Combine Lemma 3.5 and Lemma 3.6 of~\cite{Isa04}.
\end{proof}

\begin{rem}
Lemma~\ref{l:proper} can be generalized to a wider class of relative categories via the notion of \textbf{proper factorizations}, see~\cite[Proposition 3.7]{BS2}.
\end{rem}

Our first sufficient condition is based on the work of Isaksen:

\begin{thm}[\cite{Isa04}]\label{t:model_Isa}
Let $(\cC,\cW,\cF ib,\cC of)$ be a pro-admissible \textbf{model category} (e.g., a proper model category). Then the induced model structure on $\Pro(\cC)$ exists. Furthermore, we have:
\begin{enumerate}
\item The weak equivalences in $\Pro(\cC)$ are given by $\mathbf{W} =\Lw^{\cong}(\cW)$.
\item The fibrations in $\Pro(\cC)$ are given by $\mathbf{Fib} := \R(\Sp^{\cong}(\Fib))$.
\item The cofibrations in $\Pro(\cC)$ are given by $\mathbf{Cof} =\Lw^{\cong}(\cC of)$.
\item The trivial cofibrations in $\Pro(\cC)$ are given by $\mathbf{Cof}\cap\mathbf{W} =\Lw^{\cong}(\cC of\cap\cW)$.
\item The trivial fibrations in $\Pro(\cC)$ are given by $\mathbf{Fib}\cap\mathbf{W} =\R(\Sp^{\cong}(\cF ib\cap\cW))$.
\end{enumerate}
\end{thm}

\begin{proof}
The existence of a model structure satisfying Conditions (1) and (2) of Definition~\ref{d:induce}, as well as Properties (1)-(5) above, is proven in~\cite[\S 4]{Isa04}. The rest follows from Proposition \ref{p:induce}. Note that the results of~\cite{Isa04} are stated for a proper model category $\cC$. However, the properness of $\cC$ is only used to show that $\cC$ is pro-admissible (see Lemma~\ref{l:proper}), while the arguments of~\cite[\S 4]{Isa04} apply verbatim to any pro-admissible model category.
\end{proof}

The following case is based on the work of the work of Schlank and the first author:

\begin{thm}[~\cite{BS1}]\label{t:model-2}
Let $(\cC,\cW,\Fib)$ be a homotopically small pro-admissible weak fibration category and assume that $\cC$ is either essentially small or admits small colimits. Then the induced model structure on $\Pro(\cC)$ exists. Furthermore, we have:
\begin{enumerate}
\item The weak equivalences in $\Pro(\cC)$ are given by $\mathbf{W} =\Lw^{\cong}(\cW)$.
\item The fibrations in $\Pro(\cC)$ are given by $\mathbf{Fib} := \R(\Sp^{\cong}(\Fib))$.
\item The trivial fibrations in $\Pro(\cC)$ are given by $\mathbf{Fib}\cap\mathbf{W} =\R(\Sp^{\cong}(\cF ib\cap\cW))$.
\end{enumerate}
\end{thm}

\begin{proof}
The existence of a model structure satisfying Conditions (1) and (2) of Definition~\ref{d:induce}, as well as Properties (1)-(3) above, is proven in~\cite[Theorem 4.18]{BS1}. The rest follows from Proposition \ref{p:induce}
\end{proof}

\subsection{The weak equivalences in the induced model structure}

In this subsection, we let $\cC$ be a weak fibration category and \textbf{assume} that the induced model structure on $\Pro(\cC)$ exists (see Definition~\ref{d:induce}). Our goal is to relate the weak equivalences of $\Pro(\cC)$ to the class $\Lw^{\cong}(\cW)$ (see Definition~\ref{def mor}).

\begin{prop}\label{p:Lw are WE}
Every map in $\Lw^{\cong}(\cW)$ is a weak equivalence in $\Pro(\cC)$.
\end{prop}

\begin{proof}
Since any isomorphism is a weak equivalence it is enough to show that every map in $\Lw(\cW)$ is a weak equivalence in $\Pro(\cC)$. Let $\cI$ be a cofiltered category and let $f:Z \lrar X$ be a morphism in $\cC^{\cI}$ which is levelwise in $\cW$. By condition (3) of Definition~\ref{d:induce}, there exists a cofiltered category $\cJ$ with a coinitial functor $\mu:\cJ\lrar \cI$ and a factorization
\[ \mu^*Z\xrightarrow{g} Y\xrightarrow{h}\mu^*X \]
in $\cC^{\cJ}$ of the map $\mu^*f: \mu^*Z \lrar \mu^*X$ such that $g$ is a cofibration in $\Pro(\cC)$ and $h$ is both a trivial fibration in $\Pro(\cC)$  and a levelwise trivial fibration. Since $f$ is a levelwise weak equivalence, we get that $g$ is a levelwise weak equivalence. Since the weak equivalences in $\Pro(\cC)$ are closed under composition, it is enough to show that $g$ is a trivial cofibration in $\Pro(\cC)$, or, equivalently, that $g\in {}^{\perp} \Fib$. But $g\in {}^{\perp} (\Fib\cap \cW)\cap \Lw(\cW)$, so this follows from \cite[Proposition 4.17]{BS1}.
\end{proof}

\begin{cor}\label{c:factorizations}
Every map $f:Z \lrar X$ in $\Pro(\cC)$ can be factored as
\[ Z \x{g}{\lrar} Z' \x{h}{\lrar} X' \x{\cong}{\lrar} X \]
such that $g$ is a weak equivalence, $h$ is a levelwise fibration and the isomorphism $X' \x{\cong}{\lrar} X$ is a reindexing isomorphism (see Definition~\ref{d:reindex}).
\end{cor}

\begin{proof}
Let $f: Z \lrar X$ be a map in $\Pro(\cC)$. By Lemma~\ref{l:every map natural}, we may assume that $f$ is given by a morphism in $\cC^{\cT}$, with $\cT$ a cofiltered category.

Now choose an inverse poset $\cA$ with a coinitial functor $\mu:\cA\lrar \cT$. Since $\cF ib$ and $\cW$ are classes of morphisms in $\cC$ such that $\cF ib\circ \cW  = \Mor(\cC)$, we can, by the construction described in~\cite[Definition 4.3]{BS15}, factor $\mu^*f$ as
\[\mu^*Z\xrightarrow{\Lw(\cW)} Z'\xrightarrow{\Sp(\Fib)}\mu^*X.\]
The first map is a weak equivalence by Proposition~\ref{p:Lw are WE} and the second map is in $\Lw(\Fib)$ by Proposition~\ref{p:forF_sp_is_lw_strong}, so the conclusion of the lemma follows.
\end{proof}

Proposition~\ref{p:Lw are WE} admits two partial converses.
\begin{lem}\label{l:converse-cof}
Every trivial cofibration in $\Pro(\cC)$ belongs to $\Lw^{\cong}(\cW)$.
\end{lem}

\begin{proof}
Since $\cC$ is a weak fibration category we know that $\cC$ has finite limits and that $\Mor(\cC)=\Fib\circ\cW$. By~\cite[Proposition 4.1]{BS15} we know that $\Mor(\Pro(\cC))=\Sp^{\cong}(\Fib)\circ\Lw^{\cong}(\cW)$. Now, by~\cite[Proposition 4.1 and Lemma 4.5]{BS15} and Lemma~\ref{l:ret_lw} we have that
\[\mathbf{Cof}\cap\mathbf{W} = {}^{\perp} \Fib={}^{\perp}\Sp^{\cong}(\Fib)\subseteq \R(\Lw^{\cong}(\cW))=\Lw^{\cong}(\cW).\]
\end{proof}

\begin{lem}\label{l:converse-fib}
Every trivial fibration in $\Pro(\cC)$ belongs to $\Lw^{\cong}(\cW)$.
\end{lem}

\begin{proof}
Let $f:Z \lrar X$ be a morphism in $\Pro(\cC)$. By Lemma~\ref{l:every map natural} $f$ is isomorphic to a natural transformation $f':Z' \lrar X'$ over a common indexing category $\cT$. By condition (3) of Definition~\ref{d:induce}, there exists a cofiltered category $\cJ$ with a coinitial functor $\mu:\cJ\lrar \cT$ and a factorization in $\cC^\cJ$ of the map $\mu^*f': \mu^*Z' \lrar \mu^*X'$ of the form
\[ \mu^*Z'\xrightarrow{g} Y\xrightarrow{h}\mu^*X'\]
such that $g$ is a cofibration in $\Pro(\cC)$ and $h$ is a levelwise trivial fibration.
We thus obtain a factorization of $f$ of the form
\[Z \x{\mathbf{Cof}}{\lrar} Y \x{\Lw^{\cong}(\cW)}{\lrar} X.\]
It follows that $\Mor(\Pro(\cC))=\Lw^{\cong}(\cW)\circ\mathbf{Cof}$. Now, by~\cite[Lemma 4.5]{BS15} and Lemma~\ref{l:ret_lw} we have that
\[\mathbf{Fib}\cap\mathbf{W} = {\mathbf{Cof}}^{\perp} \subseteq \R(\Lw^{\cong}(\cW))=\Lw^{\cong}(\cW).\]
\end{proof}

Combining Lemmas~\ref{l:converse-cof} and~\ref{l:converse-fib} we obtain
\begin{cor}
Every weak equivalence in $\Pro(\cC)$ is a composition of two maps in $\Lw^{\cong}(\cW)$.
\end{cor}

\begin{cor}\label{c:compose Lw}
If the class $\Lw^{\cong}(\cW)$ is closed under composition then the weak equivalences in $\Pro(\cC)$ are precisely $\mathbf{W} =\Lw^{\cong}(\cW)$.
\end{cor}

\begin{rem}\label{r:EH}
In~\cite{EH76}, Edwards and Hastings give conditions on a model category which they call Condition N (see~\cite[Section 2.3]{EH76}). They show in~\cite[Theorem 3.3.3]{EH76} that a model category $\cC$, satisfying Condition N, gives rise to a model structure on $\Pro(\cC)$. By Proposition \ref{p:induce} we get that this model structure is induced on $\Pro(\cC)$ in the sense of Definition~\ref{d:induce}.

In~\cite{EH76} Edwards and Hastings ask whether the weak equivalences in their model structure are precisely $\Lw^{\cong}(\cW)$. Using the results above we may give a positive answer to their question. Indeed, in a model category satisfying Condition N we have that either every object is fibrant or every object is cofibrant. It follows that such a model category is either left proper or right proper. By~\cite[Proposition 3.7 and Example 3.3]{BS2} we have that $\Lw^{\cong}(\cW)$ is closed under composition and hence by Corollary~\ref{c:compose Lw} the weak equivalences in $\Pro(\cC)$ coincide with $\Lw^{\cong}(\cW)$. In particular, a model category satisfying Condition N is pro-admissible and the existence of the induced model structure is a special case of Theorem~\ref{t:model_Isa}.
\end{rem}

\begin{rem}
In all cases known to the authors the weak equivalences in the induced model structure coincide with $\Lw^{\cong}(\cW)$. It is an interesting question whether or not there exist weak fibration categories for which the induced model structure exists but $\Lw^{\cong}(\cW) \subsetneq \mathbf{W}$. In fact, we do not know of any example of a weak fibration category for which $\Lw^{\cong}(\cW)$ does not satisfy two-out-of-three.
\end{rem}

\section{The underlying $\infty$-category of $\Pro(\cC)$}\label{s:main}

Throughout this section we let $\cC$ be a weak fibration category and assume that the induced model structure on $\Pro(\cC)$ exists (see Definition~\ref{d:induce}). In the previous section we have shown that this happens, for example, if:
\begin{enumerate}
\item $\cC$ is the underlying weak fibration category of a pro-admissible model category (Theorem~\ref{t:model_Isa}).
\item $\cC$ is essentially small and pro-admissible (Theorem~\ref{t:model-2}).
\item $\cC$ is homotopically small, pro-admissible and cocomplete (Theorem~\ref{t:model-2}).
\end{enumerate}

\subsection{A formula for mapping spaces}

Let $\cC$ be an ordinary category. Given two objects $X = \{X_i\}_{i \in \cI}$ and $Y = \{Y_j\}_{j \in \cJ}$ in $\Pro(\cC)$, the set of morphisms from $X$ to $Y$ is given by the formula
\[ \Hom_{\Pro(\cC)}(X,Y) = \lim_{j \in \cJ}\colim_{i \in \cI} \Hom_{\cC}(X_i,Y_j) \]
The validity of this formula can be phrased as a combination of the following two statements:
\begin{enumerate}
\item
The compatible family of maps $Y \lrar Y_j$ induces an isomorphism
\[ \Hom_{\Pro(\cC)}(X,Y) \x{\cong}{\lrar} \lim_{j \in \cJ}\Hom_{\Pro(\cC)}(X,Y_j) \]
\item
For each simple object $Y \in \cC \subseteq \Pro(\cC)$ the compatible family of maps $X \lrar X_i$ (combined with the inclusion functor $\cC \hrar \Pro(\cC)$) induces an isomorphism
\[ \colim_{i \in \cI}\Hom_{\cC}(X_i,Y) \x{\cong}{\lrar} \Hom_{\Pro(\cC)}(X,Y) \]
\end{enumerate}

In this section we want to prove that when $\cC$ is a weak fibration category, statements (1) and (2) above hold for \textbf{derived mapping spaces} in $\Pro(\cC)$, as soon as one replaces limits and colimits with their respective \textbf{homotopy limits and colimits}. As a result, we obtain the explicit formula
\[ \Map^h_{\Pro(\cC)}(X,Y) = \holim_{j \in \cJ}\hocolim_{i \in \cI} \Map^h_{\cC}(X_i,Y_j). \]

We first observe that assertion (1) above is equivalent to the statement that the maps $Y \lrar Y_j$ exhibit $Y$ as the limit, in $\Pro(\cC)$, of the diagram $j \mapsto Y_j$. Our first goal is hence to verify that the analogous statement for homotopy limits holds as well.

\begin{prop}\label{p:homotopy-limit}
Let $\cC$ be a weak fibration category and let $Y = \{Y_j\}_{j \in \cJ} \in \Pro(\cC)$ be a pro-object. Let $\ovl{\cF}: \cJ^{\triangleleft} \lrar \Pro(\cC)$ be the limit diagram extending $\cF(j) = Y_j$ so that $\ovl{\cF}(\ast) = Y$ (where $\ast \in \cJ^{\triangleleft}$ is the cone point). Then the image of $\ovl{\cF}$ in $\Pro(\cC)_\infty$ is a limit diagram. In particular, for every $X = \{X_i\}_{i \in \cI}$ the natural map
\[
\Map^h_{\Pro(\cC)}(X,Y) \lrar \holim_{j \in \cJ}\Map^h_{\Pro(\cC)}(X,Y_j)
\]
is a weak equivalence.
\end{prop}

\begin{proof}
In light of Lemma~\ref{l:cofinal_CDS} we may assume that $Y$ is indexed by an inverse poset $\cT$ (see Definition~\ref{def CDS-2}). Consider the injective weak fibration structure on $\cC^{\cT}$ (see Lemma~\ref{l:projective-injective}). We may then replace $t \mapsto Y_t$ with an injective-fibrant \textbf{levelwise equivalent} diagram $t \mapsto Y_t'$. By Proposition~\ref{p:Lw are WE} we get that the pro-object $Y' = \{Y_t'\}_{t \in \cT}$ is weakly equivalent to $Y$ in $\Pro(\cC)$, and so it is enough to prove the claim for $Y'$.

By Corollary~\ref{c:poset-reedy} the injective model structure on $\Pro(\cC)^{\cT}$ exists, and the underlying weak fibration structure is the injective one as well. Thus the diagram $t \mapsto Y_t'$ is injectively fibrant in $\Pro(\cC)^{\cT}$. The desired result now follows from Proposition~\ref{p:injective-limit}. The last claim is a consequence of~\cite[Theorem 4.2.4.1]{Lur09} and also follows from the proof of Proposition~\ref{p:injective-limit}.
\end{proof}

Our next goal is to generalize assertion (2) above to derived mapping spaces.

\begin{prop}\label{p:main}
Let $X = \{X_i\}_{i \in \cI}$ be a pro-object and $Y \in \cC \subseteq \Pro(\cC)$ a simple object. Then the compatible family of maps $X \lrar X_i$ induces a weak equivalence
\begin{equation}\label{e:formula}
\hocolim_{i \in \cI}\Map^h_{\cC}(X_i,Y) \lrar \Map^h_{\Pro(\cC)}(X,Y)
\end{equation}
\end{prop}

Before proving Proposition~\ref{p:main}, let us note an important corollary.

\begin{cor}\label{c:fully-faithful}
The natural map
\[ \cC_\infty \lrar \Pro(\cC)_\infty \]
is fully faithful.
\end{cor}

\begin{rem}\label{r:locally-small}
Since $\cC$ is not assumed to be essentially small, the mapping spaces appearing in (\ref{e:formula}) are a priori \textbf{large} spaces (see Definition~\ref{d:spaces}). Fortunately, since $\Pro(\cC)$ is a model category we know that $\Map^h_{\Pro(\cC)}(X,Y)$ is weakly equivalent to a small simplicial set. By Corollary~\ref{c:fully-faithful} the derived mapping spaces in $\cC$ are, up to weak equivalence, small as well.
\end{rem}

The rest of this section is devoted to the proof of Proposition~\ref{p:main}. The proof itself will be given in the end of this section. We begin with a few preliminaries.

\begin{define}
Let $(\cC,\cW,\Fib)$ be a weak fibration category. We denote by $\Fib^{\fib} \subseteq \cC^{[1]}$ the full subcategory spanned by \textbf{fibrations between fibrant objects}, and by $\Triv^{\fib} \subseteq \cC^{[1]}$ the full subcategory spanned by \textbf{trivial fibrations between fibrant objects}.
\end{define}

\begin{lem}\label{l:reedy_fib_rep}
Every object $Z$ in $\Pro(\cC)$ admits a weak equivalence of the form $Z \x{\simeq}{\lrar} Z'$ with $Z' \in \Pro(\cC^{\fib})\subseteq\Pro(\cC)$.
\end{lem}

\begin{proof}
This follows from Corollary~\ref{c:factorizations} applied to the map $Z\lrar\ast$.
\end{proof}

\begin{lem}\label{l:reedy_fib_rep-2}
Under the natural equivalence
$\Pro(\cC^{[1]})\simeq\Pro(\cC)^{[1]}$
every trivial fibration in $\Pro(\cC)$, whose codomain is in $\Pro(\cC^{\fib})$, is a retract of a trivial fibration which belongs to $\Pro(\Triv^{\fib})$.
\end{lem}

\begin{proof}
Let $f:Z\lrar X$ be a trivial fibration in $\Pro(\cC)$, whose codomain is in $\Pro(\cC^{\fib})$. By Lemma~\ref{l:every map natural} we may assume that $X,Y$ are both indexed by the same cofiltered category $\cI$ and that $f$ is given by a morphism in $\cC^{\cI}$. By Condition (3) of Definition~\ref{d:induce} there exists a coinitial functor $\mu:\cJ\lrar \cI$ and a factorization
\[ \mu^*Z\xrightarrow{g} Y\xrightarrow{h}\mu^*X\]
in $\cC^{\cJ}$ of the map $\mu^*f: \mu^*Z \lrar \mu^*X$, such that $g$ is a cofibration and $h$ is both a trivial fibration and a levelwise trivial fibration. It follows that $\mu^* X$ belongs to $\Pro(\cC^{\fib})$ and the map $Y\xrightarrow{h}\mu^*X$ belongs to $\Pro(\Triv^{\fib})$. The commutative diagram
\[
\xymatrix{
\mu^*Z\ar[r]^=\ar[d]_{\mathbf{Cof}} & \mu^*Z\ar[d]^{\mathbf{W}\cap\mathbf{Fib}} \\
            Y\ar[r] & X \\
}
\]
then admits a lift $Y \lrar Z$. Using the isomorphisms $Z \x{\cong}{\lrar} \mu^*Z$, $X \x{\cong}{\lrar} \mu^*X$ and their inverses we obtain a retract diagram in $\Pro(\cC)^{[1]}$ of the form
\[\xymatrix{Z\ar[r]\ar[d]^f & Y\ar[r]\ar[d] & Z\ar[d]^f\\
           X \ar[r] & \mu^*X\ar[r] & X} \]
and so the desired result follows.
\end{proof}

For the proof of Proposition~\ref{p:main} below we need the following notion. Let
\begin{equation}\label{e:square}
\xymatrix{
\cA \ar^{\vphi}[r]\ar^{\tau}[d] & \cB \ar^{\rho}[d] \\
\cC \ar^{\psi}[r] & \cD \\
}
\end{equation}
be a diagram of categories equipped with a commutativity natural isomorphism $\rho \circ \vphi \rightarrow \psi \circ \tau$. Given a triple $(b,c,f)$ where $b \in \cB,c \in \cC$ and $f: \rho(b) \lrar \psi(c)$ is a morphism in $\cD$, we denote by $\cM(\cA,b,c,f)$ the category whose objects are triples $(a,g,h)$ where $a$ is an object of $\cA$, $g:b \lrar \vphi(a)$ is a morphism in $\cB$ and $h:\tau(a) \lrar c$ is a morphism in $\cC$ such that the composite
\[ \rho(b) \x{\rho(g)}{\lrar} \rho(\vphi(a)) \cong \psi(\tau(a)) \x{\psi(h)}{\lrar} \psi(c) \]
is equal to $f$.

\begin{define}\label{d:cc}
We say that the square (\ref{e:square}) is \textbf{categorically Cartesian} if for every $(b,c,f)$ as above the category $\cM(\cA,b,c,f)$ is weakly contractible.
\end{define}

Our main claim regarding categorically Cartesian diagrams is the following:

\begin{lem}\label{l:cat-car}
Let $\cD$ be a category and $\cD_0 \subseteq \cD$ a full subcategory. Let $\cE \subseteq \Pro(\cD)$ be a full subcategory containing $\cD_0$, such that each object of $\cE$ is a retract of an object in $\cE \cap \Pro(\cD_0)$. Then the diagram
\[
\xymatrix{
\cD_0 \ar[r]\ar[d] & \cE \ar^{\iota_{\cE}}[d] \\
\cD \ar^-{\iota_{\cD}}[r] & \Pro(\cD) \\
}
\]
is categorically Cartesian.
\end{lem}

\begin{proof}
Let $d \in \cD, e \in \cE$ be objects and $f: \iota_{\cE}(e) \lrar \iota_{\cD}(d)$ a morphism. We need to show that $\cM(\cD_0,d,e,f)$ is weakly contractible. By our assumptions there exists a $e' \in \Pro(\cD_0) \cap \cE$ and a retract diagram of the form $e \lrar e' \lrar e$. Let $f': \iota_{\cE}(e') \lrar \iota_{\cD}(d)$ be the map obtained by composing the induced map $\iota_\cE(e') \lrar \iota_\cE(e)$ with $f$. We then obtain a retract diagram of simplicial sets
\[ \Ne\cM(\cD_0,d,e,f) \lrar \Ne\cM(\cD_0,d,e',f') \lrar \Ne\cM(\cD_0,d,e,f) \]
and so it suffices to prove that $\cM(\cD_0,d,e',f')$ is weakly contractible. In particular, we might as well assume that $\cE = \Pro(\cD_0)$ and suppress $\iota_{\cE}$ from our notation.

Now let $e = \{e_i\}_{i \in \cI}$ be an object of $\Pro(\cD_0)$, $d$ be an object of $\cD$ and $f: e \lrar d$ a map in $\Pro(\cD)$. We may identify $\cM(\cD_0,d,e,f)$ with the Grothendieck construction of the functor $\cH_d: ((\cD_0)_{e/})^{\op} \lrar \Set$ which associates to each $(e \lrar x) \in ((\cD_0)_{e/})^{\op}$ the set of morphisms $g: x \lrar d$ in $\cD$ such that the composite $e \lrar x \x{g}{\lrar} d$ in $\Pro(\cD)$ is $f$. By~\cite{Th79} we may consider $\cM(\cD_0,d,e,f)$ as a model for the homotopy colimit of the functor $\cH_d$. Now according to Corollary~\ref{c:pro} the natural functor
\[ \cI^{\op} \lrar ((\cD_0)_{e/})^{\op} \]
sending $i$ to $e \lrar e_i$ is cofinal. By Theorem~\ref{t:cofinal-2} it suffices to prove that the Grothendieck construction of the restricted functor $\cH_d|_{\cI^{\op}}: \cI^{\op} \lrar \Set$ is weakly contractible. But this is exactly the content of Lemma~\ref{l:lem-1}.
\end{proof}

Our next goal is to construct an explicit model for the homotopy colimit on the left hand side of (\ref{e:formula}). Let $X = \{X_i\}_{i \in \cI} \in \Pro(\cC^{\fib})$ and let $Y \in \cC^{\fib}$ be a fibrant simple object. Let $\cG(X,Y)$ be the Grothendieck construction of the functor
$\cH_Y: \left(\cC^{\fib}_{X/}\right)^{\op} \lrar \ovl{\Cat}$ which sends the object $(X \lrar X') \in \left(\cC^{\fib}_{X/}\right)^{\op}$ to the category $\uline{\Hom}_{\cC}(X',Y)$. Unwinding the definitions, we see that an object in
$\cG(X,Y)$ corresponds to a diagram of the form
\begin{equation}\label{e:gro}
\xymatrix{
& Z \ar^{g}[r]\ar^{f}[d] & Y \ar[d] \\
X\ar[r] & X' \ar[r] & \ast \\
}
\end{equation}
where $X'$ is a fibrant object of $\cC$, $f: Z \lrar X'$ is a trivial fibration in $\cC$, and $X,Y$ are fixed. By the main result of~\cite{Th79} the nerve of the category $\cG(X,Y)$ is a model for the homotopy colimit of the composed functor $N\circ\cH_Y: \left(\cC^{\fib}_{X/}\right)^{\op} \lrar \cS$. We have a natural functor
\[ \cF_{X,Y}: \cG(X,Y) \lrar \uline{\Hom}_{\Pro(\cC)}(X,Y) \]
which sends the object corresponding to the diagram (\ref{e:gro}) to the external rectangle in the diagram
\[
\xymatrix{
X \times_{X'} Z \ar[r]\ar[d] & Z \ar^{g}[r]\ar^{f}[d] & Y \ar[d] \\
X\ar[r] & X' \ar[r] & \ast \\
}
\]
considered as an object of $\uline{\Hom}_{\Pro(\cC)}(X,Y)$.

\begin{prop}\label{p:cofinal}
Let $X,Y$ be as above. Then the functor
\[ \cF_{X,Y}: \cG(X,Y) \lrar \uline{\Hom}_{\Pro(\cC)}(X,Y) \]
is cofinal.
\end{prop}

\begin{proof}
Let $X \in \Pro(\cC^{\fib})$ and let $Y \in \cC^{\fib}$ be a simple fibrant object. Let $W \in \uline{\Hom}_{\Pro(\cC)}(X,Y)$ be an object corresponding to a diagram of the form
\begin{equation}\label{e:W}
\xymatrix{
Z \ar[r]\ar^{p}[d] & Y \ar[d] \\
X \ar[r] & \ast \\
}
\end{equation}
where $p$ is a trivial fibration in $\Pro(\cC)$. We want to show that the category $\cG(X,Y)_{W/}$ is weakly contractible. Unwinding the definitions we see that objects of $\cG(X,Y)_{W/}$ correspond to diagrams of the form
\begin{equation}
\xymatrix{
Z \ar[r]\ar^{p}[d] & Z' \ar[r]\ar^{p'}[d] & Y \ar[d] \\
X \ar[r] & X' \ar[r] & \ast \\
}
\end{equation}
where $p': Z' \lrar X'$ is a trivial fibration in $\cC$.

Now let $\cD = \cC^{[1]}$ be the arrow category of $\cC$ and let $\cD_0 = \Triv^{\fib} \subseteq \cD$ the full subcategory spanned by trivial fibrations between fibrant objects. The category $\Pro(\cD)$ can be identified with the arrow category $\Pro(\cC)^{[1]}$. Let $\cE \subseteq \Pro(\cD)$ be the full subcategory spanned by trivial fibrations whose codomain is in $\Pro(\cC^{\fib})$. According to Lemma~\ref{l:reedy_fib_rep-2} every object $\cE$ is a retract of an object in $\Pro(\cD_0)$. We hence see that the categories $\cD,\cD_0$ and $\cE$ satisfy the assumptions of Lemma~\ref{l:cat-car}. It follows that the square
\begin{equation}\label{e:square-3}
\xymatrix{
\cD_0 \ar[r]\ar[d] & \cE \ar^{\iota_{\cE}}[d] \\
\cD \ar^{\iota_{\cD}}[r] & \Pro(\cD) \\
}
\end{equation}
is categorically Cartesian. Now the object $Y$ corresponds to an object $d = (Y \lrar \ast) \in \cD$ and the trivial fibration $p:Z \lrar X$ corresponds to an object $e \in \cE$. The diagram (\ref{e:W}) then gives a map $f:\iota_{\cE}(e) \lrar \iota_{\cD}(d)$. The category $\cM(\cD_0,d,e,f)$ of Definition~\ref{d:cc} can then be identified with $\cD(X,Y)_{W/}$. Since (\ref{e:square-3}) is categorically Cartesian we get that $\cD(X,Y)_{W/}$ is weakly contractible as desired.
\end{proof}

We are now ready to prove the main result of this subsection.

\begin{proof}[Proof of Proposition~\ref{p:main}]
We begin by observing that both sides of (\ref{e:formula}) remain unchanged up to a weak equivalence by replacing $Y$ with a fibrant model $Y \x{\simeq}{\lrar} Y'$. We may hence assume without loss of generality that $Y$ itself is fibrant. According to Lemma~\ref{l:reedy_fib_rep} we may also assume that each $X_i$ is fibrant as well.

Now according to Corollary~\ref{c:pro} the natural functor $\iota:\cI^{\op} \lrar \left(\cC^{\fib}_{X/}\right)^{\op}$ which sends $i$ to $X \lrar X_i$ is cofinal. Let $\cG(\cI,X,Y)$ be the Grothendieck construction of the restricted functor $\left(\cH_Y\right)|_{\cI^{\op}}: \cI^{\op} \lrar \ovl{\Cat}$. Since $\iota$ is cofinal we known by Theorem~\ref{t:cofinal-2} that the natural map $\cG(\cI,X,Y) \lrar \cG(X,Y)$ induces a weak equivalence on nerves. By Proposition~\ref{p:cofinal} the functor $\cF_{X,Y}$ induces a weak equivalence on nerves and so the composed functor
\[ \cG(\cI,X,Y) \lrar \uline{\Hom}_{\Pro(\cC)}(X,Y) \]
induces a weak equivalence on nerves as well. Now the nerve of the category $\cG(\cI,X,Y)$ is a model for the homotopy colimit of the functor sending $i \in \cI^{\op}$ to $\Map^h_{\cC}(X_i,Y)$. On the other hand, the nerve of $\uline{\Hom}_{\Pro(\cC)}(X,Y)$ is a model for $\Map^h(X,Y)$. It hence follows that the map (\ref{e:formula}) is a weak equivalence as desired.
\end{proof}

\subsection{The comparison of $\Pro(\cC)_\infty$ and $\Pro(\cC_\infty)$}\label{ss:proof-of-main}

Let $\cC$ be weak fibration category such that the induced model structure on $\Pro(\cC)$ exists. By Remark~\ref{r:locally-small} we know that $\cC_\infty$ and $\Pro(\cC)_\infty$ are locally small $\infty$-categories. Let $\Pro(\cC_\infty)$ be the pro-category of $\cC_\infty$ in the sense of Definition~\ref{d:pro}. Let $\cF$ be the composed map
\[ \Pro(\cC)_\infty \lrar \Fun(\Pro(\cC)_\infty,\cS_\infty)^{\op} \lrar \Fun(\cC_\infty,\cS_\infty)^{\op} \]
where the first map is the opposite Yoneda embedding and the second is given by restriction. Informally, the functor $\cF$ may be described as sending an object $X \in \Pro(\cC)_\infty$ to the functor $\cF(X): \cC_\infty \lrar \cS_\infty$ given by $Y \mapsto \Map^h_{\Pro(\cC)}(X,Y)$. By Proposition~\ref{p:main} we know that $\cF(X)$ is a small cofiltered limit of objects in the essential image of $\cC_\infty \subseteq \Fun(\cC_\infty,\cS_\infty)^{\op}$. It hence follows by Lemma~\ref{l:closed-filtered} and Corollary~\ref{c:fully-faithful} that the image of $\cF$ lies in $\Pro(\cC_\infty)$.
We are now able to state and prove our main theorem:

\begin{thm}\label{t:main}
The functor
\[ \cF: \Pro(\cC)_\infty \lrar \Pro(\cC_\infty) \]
is an equivalence of $\infty$-categories.
\end{thm}

\begin{proof}
We first prove that $\cF$ is fully faithful. Let $Y = \{Y_i\}_{i \in \cI}$ be a pro-object. By Proposition~\ref{p:homotopy-limit} we know that the natural maps $Y \lrar Y_i$ exhibit $Y$ is the homotopy limit of the diagram $i \mapsto Y_i$. On the other hand, by Proposition~\ref{p:main} the maps $\cF(Y) \lrar \cF(Y_i)$ exhibit $\cF(Y)$ as the limit of the diagram $i \mapsto \cF(Y_i)$ in $\Pro(\cC_\infty)$. Hence in order to show that $\cF$ is fully faithful it suffices to show that $\cF$ induces an equivalence on mapping spaces from a pro-object to a simple object. In light of Proposition~\ref{p:main} and the fact that every simple object in $\Pro(\cC_\infty)$ is $\omega$-cocompact (see Lemma~\ref{l:cocompact}), we may reduce to showing that the restriction of $\cF$ to $\cC_\infty$ is fully faithful. But this now follows from Corollary~\ref{c:fully-faithful}.

We shall now show that $\cF$ is essentially surjective. By Theorem~\ref{t:model-complete} $\Pro(\cC)_\infty$ has all limits and colimits. Since the restricted functor $\cF|_{\cC_\infty}$ is fully faithful and its essential image are the corepresentable functors we may conclude that the essential image of $\cF$ in $\Pro(\cC_\infty)$ contains every object which is a colimit of corepresentable functors. But by Lemma~\ref{l:generate} every object in $\Pro(\cC_\infty)$ is a colimit of corepresentable functors. This concludes the proof of Theorem~\ref{t:main}.
\end{proof}

Let $f: \cC \lrar \cD$ be a weak right Quillen functor between two weak fibration categories. Then the prolongation $\Pro(f): \Pro(\cC) \lrar \Pro(\cD)$ preserves all limits. It is hence natural to ask when does $\Pro(f)$ admit a left adjoint.

\begin{lem}\label{l:left-adjoint}
Let $\cC,\cD$ be weak fibration categories and $f: \cC \lrar \cD$ a weak right Quillen functor. The functor $\Pro(f): \Pro(\cC) \lrar \Pro(\cD)$ admits a left adjoint $L_f$ if and only if for every $d \in \cD$ the functor $R_d: c \mapsto \Hom_{\cD}(d,f(c))$ is small. Furthermore, when this condition is satisfied then $R_d$ belongs to $\Pro(\cC) \subseteq \Fun(\cC,\Set)^{\op}$ and $L_f$ is given by the formula
\[ L_f(\{d_i\}_{i \in \cI}) = \lim_{i \in \cI}R_{d_i},\]
where the limit is taken in the category $\Pro(\cC)$.
\end{lem}

\begin{proof}
First assume that a left adjoint $L_f: \Pro(\cD) \lrar \Pro(\cC)$ exists. By adjunction we have
\[ \Hom_{\Pro(\cC)}(L_f(d),c) = \Hom_{\Pro(\cD)}(d,f(c)) = \Hom_{\cD}(d,f(c)) = R_d(c) \]
for every $c \in \cC, d \in \cD$ and so the functor $R_d$ is corepresented by $L_f(d)$, i.e., corresponds to the pro-object $L_f(d) \in \Pro(\cC) \subseteq \Fun(\cC,\Set)^{\op}$. It follows that $R_d$ is small.

Now assume that each $R_d$ is small. Let $\widetilde{\cC} \lrar \cC$ be the Grothendieck construction of the functor $R_d$. Since $f$ preserves finite limits it follows that $R_d$ preserves finite limits. This implies that $\widetilde{\cC}$ is cofiltered and so by Proposition~\ref{p:characterization} $R_d$ belongs to the essential image of $\Pro(\cC)$ in $\Fun(\cC,\Set)^{\op}$. We may then simply \textbf{define} $L_f: \Pro(\cD) \lrar \Pro(\cC)$ to be the functor
\[ L_f(\{d_i\}_{i \in \cI}) = \lim_{i \in \cI}R_{d_i}.\]
where the limit is taken in $\Pro(\cC)$. The map of sets $\Hom_{\cC}(c,c') \lrar R_{f(c)}(c') = \Hom_{\cD}(f(c),f(c'))$ determines a counit transformation $L_f \circ \Pro(f) \Rightarrow \Id$ and it is straightforward to verify that this counit exhibits $L_f$ as left adjoint to $\Pro(f)$.
\end{proof}

\begin{rem}
The condition of Lemma~\ref{l:left-adjoint} holds, for example, in the following cases:
\begin{enumerate}
\item The categories $\cC$ and $\cD$ are small.
\item The categories $\cC$ and $\cD$ are accessible and $f$ is an accessible functor (see~\cite[Example 2.17 (2)]{AR}).
\end{enumerate}
\end{rem}

\begin{prop}
Let $f: \cC \lrar \cD$ be a weak right Quillen functor between weak fibration categories such that the condition of Lemma~\ref{l:left-adjoint} is satisfied. Suppose that the induced model structures on $\Pro(\cC)$ and $\Pro(\cD)$ exist. Then the adjoint pair
\[L_f:\Pro(\cD)\rightleftarrows\Pro(\cC):\Pro(f)\]
given by Lemma~\ref{l:left-adjoint} is a Quillen pair.
\end{prop}

\begin{proof}
Since $f(\Fib_{\cC}) \subseteq \Fib_{\cD}$ it follows by adjunction that $L_f({}^{\perp} \Fib_{\cD}) \subseteq {}^{\perp} \Fib_{\cC}$. Since $f(\Fib_{\cC} \cap \cW_{\cC}) \subseteq \Fib_{\cD} \cap \cW_{\cD}$ it follows by adjunction that
that $L_f({}^{\perp}(\Fib_{\cD}\cap \mcal{W}_{\cD})) \subseteq {}^{\perp}(\Fib_{\cC}\cap \mcal{W}_{\cC})$. By Properties (1) and (2) of Definition~\ref{d:induce} we may now conclude that $L_f$ preserves cofibrations and trivial cofibrations.
\end{proof}

By Remark~\ref{r:adjunction} we may consider the induced adjunction of $\infty$-categories
\[(L_f)_\infty:\Pro(\cD)_\infty\rightleftarrows\Pro(\cC)_\infty:\Pro(f)_\infty.\]
We then have the following comparison result.

\begin{prop}\label{p:weak r q functor}
Under the assumptions above, the right derived functor $\Pro(f)_{\infty}: \Pro(\cC)_\infty \lrar \Pro(\cD)_\infty$ is equivalent to the prolongation of the right derived functor $f_\infty: \cC_\infty \lrar \cD_\infty$ (see Definition~\ref{d:prolong}), under the equivalence of Theorem~\ref{t:main}.
\end{prop}

\begin{proof}
By the universal property of Theorem~\ref{t:universal}, it suffices to prove that both functors preserve cofiltered limits and restrict to equivalent functors on the full subcategory $\cC_{\infty}\subset\Pro(\cC_{\infty})$. Now $\Pro(f_{\infty})$ preserve cofiltered limits by definition and $\Pro(f)_{\infty}$ preserves all limits by Remark~\ref{r:adjunction}. Moreover, the restriction of both functors to $\cC^{\fib}_{\infty}\simeq\cC_{\infty}$ is the functor induced by $f$.
\end{proof}

\subsection{Application: $\omega$-presentable $\infty$-categories}\label{ss:omega presentable}

Let $\cM$ be a combinatorial model category. By~\cite[Proposition A.3.7.6.]{Lur09} and the main result of~\cite{Dug01}, the underlying $\infty$-category $\cM_\infty$ is presentable. For many purposes it is often useful to know that $\cM_\infty$ is not only presentable, but also \textbf{$\omega$-presentable}, i.e., equivalent to the ind-category of its subcategory of $\omega$-compact objects. Recall that a model category $\cM$ is said to be \textbf{$\omega$-combinatorial} if its underlying category is $\omega$-presentable and $\cM$ admits a sets of generating cofibrations and trivial cofibrations whose domains and codomains are $\omega$-compact.

In this subsection we use (a dual version of) our main result, Theorem \ref{t:main}, to give sufficient conditions on an $\omega$-combinatorial model category $\cM$ which insure that its underlying $\infty$-category is $\omega$-presentable. First note that the definition of a weak fibration category can be directly dualized to obtain the notion of a \textbf{weak cofibration category}. If $\cC$ is a weak cofibration category then $\cC^{\op}$ is naturally a weak fibration category, and Theorem~\ref{t:main} can be readily applied to $\Ind(\cC) \cong \left(\Pro(\cC^{\op})\right)^{\op}$.

\begin{prop}\label{p:omega}
Let $(\cM,\cW,\cF,\cC)$ be an $\omega$-combinatorial model category and let $\cM_0 \subseteq \cM$ be the full subcategory spanned by $\omega$-compact objects. Let $\cW_0$ and $\cC_0$ denote the classes of weak equivalences and cofibrations between objects in $\cM_0$, respectively. Suppose that $\Mor(\cM_0) = \cW_0 \circ \cC_0$. Then $(\cM_0)_\infty$ is essentially small, admits finite colimits and
$$\Ind((\cM_0)_\infty)\simeq\cM_\infty.$$
In particular, $\cM_\infty$ is $\omega$-presentable, and every $\omega$-compact object in $\cM_\infty$ is a retract of an object in $\cM_0$.
\end{prop}

\begin{proof}
By our assumption $\Mor(\cM_0) = \cW_0 \circ \cC_0$ and hence $(\cM_0,\cW_0,\cC_0)$ forms a weak cofibration category. The formation of colimits induces an equivalences of categories $\Ind(\cM_0) \cong \cM$. Since $\cM$ is $\omega$-combinatorial it admits generating sets $\cI,\cJ$ of cofibrations and trivial cofibrations respectively such that the domains and codomains of all maps in $\cI$ and $\cJ$ are in $\cM_0$. It follows that:
\begin{enumerate}
\item The fibrations in $\cM$ are $(\cC_0\cap \cW_0)^{\perp}$.
\item The trivial fibrations in $\cM$ are $\cC_0^{\perp} $.
\end{enumerate}
From Proposition~\ref{p:induce} we now get that the model structure on $\cM\simeq\Ind(\cM_0)$ is induced from $\cM_0$ in the sense of Definition \ref{d:induce}. From Theorem \ref{t:main} we can now deduce that
$$\Ind((\cM_0)_\infty)\simeq\cM_\infty ,$$
as desired. It follows that $(\cM_0)_\infty$ is locally small and hence essentially small (since $\cM_0$ is essentially small). From Corollary~\ref{p:WFC_PB} we get that $(\cM_0)_\infty$ admits finite colimits. The characterization of $\omega$-compact objects in $\cM_\infty$ now reduces to a well-known property of ind-categories.
\end{proof}

\begin{example}
Let $\cS$ be the category of simplicial sets with the Kan-Quillen model structure and let $\cS_0 \subseteq \cS$ be the full subcategory spanned by $\omega$-compact objects. Since $\cS$ is a presheaf category (of sets) it is easy to see that the $\omega$-compact objects are exactly those simplicial sets which have finitely many non-degenerate simplices. By Proposition~\ref{p:omega} we may hence conclude that
\begin{equation}\label{e:ind}
\cS_{\infty} \simeq \Ind((\cS_0)_\infty).
\end{equation}
We note that it is well-known that the $\infty$-category $\cS_{\infty}$ is $\omega$-combinatorial. However, Proposition~\ref{p:omega} gives us a bit more information, as it says that $\cS_\infty$ is more specifically the ind-category of $(\cS_0)_\infty$. This allows one to explicitly determine the $\infty$-category modelled by $\cS_0$. Indeed,~\ref{e:ind} implies that $(\cS_0)_\infty$ is a full subcategory of $\cS_\infty$. The essential image of the functor $(\cS_0)_\infty \lrar \cS_\infty$ is well-known as well: it consists of exactly those spaces which can be written as a colimit of a constant diagram $K \lrar \cS_\infty$ with value $\ast$ indexed by a $K \in \cS_0$. By Corollary~\ref{p:WFC_PB} the $\infty$-category $(\cS_0)_\infty$ admits finite colimits (i.e., colimits indexed by $\cS_0$) and these must coincide with the respective colimits in $\cS_\infty$ in view of Proposition~\ref{p:main}. We then obtain an explicit description of the $\infty$-category modelled by the weak cofibration category $\cS_0$ as the smallest full subcategory of $\cS_\infty$ containing $\ast$ and closed under finite colimits. Finally, we note that $(\cS_0)_\infty$ does not contain all $\omega$-compact objects of $\cS_\infty$ (as one has Wall finiteness obstruction), but every $\omega$-compact object of $\cS_\infty$ is indeed a retract of an object of $(\cS_0)_\infty$.
\end{example}

\section{Application: \'Etale homotopy type and shape of topoi}\label{s:topos}

Let $(\cC,\tau)$ be a small Grothendieck site and let $\PS_\Del(\cC)$ (resp. $\Sh_{\Del}(\cC)$) be the category of small simplicial presheaves (resp. small simplicial sheaves) on $\cC$. The category $\PS_{\Del}(\cC)$ (resp. $\Sh_{\Del}(\cC)$) can be given a weak fibration structure, where the weak equivalences and fibrations are local in the sense on Jardine~\cite{Jar87}. It is shown in~\cite{BS1} that $\PS_{\Del}(\cC)$ and $\Sh_{\Del}(\cC)$ are homotopically small and pro-admissible. Thus, by Theorem~\ref{t:model-2}, the induced model structure exists for both $\Pro(\PS_{\Del}(\cC))$ and $\Pro(\Sh_{\Del}(\cC))$. We refer to these model structures as \textbf{the projective model structures} on
$\Pro(\PS_{\Del}(\cC))$ and $\Pro(\Sh_{\Del}(\cC))$ respectively.

We denote by $\Sh_\infty(\cC)$ the $\infty$-topos of sheaves on $\cC$. The underlying  $\infty$-categories of $\PS_{\Del}(\cC)$ and $\Sh_{\Del}(\cC)$ are naturally equivalent by~\cite[Theorem 5]{Jar07} and both form a model for the hypercompletion $\what{\Sh}_\infty(\cC)$ of the $\infty$-topos $\Sh_\infty(\cC)$ by~\cite[Proposition 6.5.2.14]{Lur09}. We hence obtain the following corollary of Theorem~\ref{t:main}:

\begin{cor}\label{c:main-topos}
We have a natural equivalences of $\infty$-categories
\[ \Pro(\PS_{\Del}(\cC))_\infty \simeq \Pro(\what{\Sh}_\infty(\cC)) \]
\[ \Pro(\Sh_{\Del}(\cC))_\infty \simeq \Pro(\what{\Sh}_\infty(\cC)). \]
\end{cor}

We have an adjunction
\[\Gamma^*:\cS \rightleftarrows \Sh_{\Del}(\cC) : \Gamma_* \]
where $\Gamma_*$ is the global sections functor and $\Gamma^*$ is the constant sheaf functor. As explained in~\cite{BS1}, the functor $\Gamma^*$ (which is a left functor in the adjunction above) is a weak \textbf{right} Quillen functor. Since the categories $\cS$ and $\Sh_{\Del}(\cC)$ are locally presentable and $\Gamma^*$ is accessible (being a left adjoint $\Gamma^*$ preserves all small colimits), we obtain a Quillen adjunction
\[L_{\Gamma^*}:\Pro(\Sh_{\Del}(\cC)) \rightleftarrows \Pro(\cS):\Pro(\Gamma^*),\]
where $\Pro(\Gamma^*)$ is now the right Quillen functor. In light of Remark~\ref{r:adjunction}, this Quillen adjunction induces an adjunction of $\infty$-categories
\[
(L_{\Gamma^*})_{\infty}:\Pro(\Sh_{\Del}(\cC))_\infty \rightleftarrows \Pro(\cS)_\infty:\Pro(\Gamma^*)_{\infty}.
\]

\begin{define}\label{d:realization}
The \textbf{topological realization of $\cC$} is defined in to be
\[|\cC| := (L_{\Gamma^*})_{\infty}(*)\in \Pro(\cS)_\infty,\]
where $*$ is a terminal object of $\Sh_{\Del}(\cC).$
\end{define}

This construction has an $\infty$-categorical version that we now recall. Let $\cX$ be an $\infty$-topos. According to~\cite[Proposition 6.3.4.1]{Lur09} there exists a unique (up to a contractible space of choices) geometric morphism
\[q^*:\cS_\infty\rightleftarrows\cX:q_*.\]
By definition of a geometric morphism, the functor $q^*$ preserves finite limits. As a right adjoint, the functor $q_*$ preserves all limits. Moreover, both functors are accessible. Thus the composite $q_*\circ q^*$ is an accessible functor which preserves finite limits, and hence represents an object of $\Pro(\cS_{\infty})$ (see Remark~\ref{r:lurie}). This object is called the \textbf{shape} of $\cX$ and is denoted $\Shp(\cX)$. This definition appears in~\cite[Definition 7.1.6.3]{Lur09}.

In order to compare the above notion of shape with Definition~\ref{d:realization} we will need the following lemma:

\begin{lem}\label{l:geometric}
Let $\cC$ be a Grothendieck site. Then the derived functor
\[ \Gamma_\infty^*:\cS_\infty \lrar (\Sh_{\Del}(\cC))_\infty \]
preserves finite limits and has a right adjoint.
\end{lem}

\begin{proof}
Since the functor $\Gamma^*$ is a weak right Quillen functor between weak fibration categories we get from Corollary~\ref{p:weak right preserves limits} that $\Gamma^*_\infty$ preserves finite limits. Furthermore, if one endows $\Sh_{\Del}(\cC)$ with the model structure of~\cite{Joy83,Jar87} (in which the cofibrations are the monomorphisms and the weak equivalences are the local weak equivalences) we clearly obtain a Quillen adjunction
\[\Gamma^*:\cS \rightleftarrows \Sh_{\Del}(\cC) : \Gamma_* .\]
In light of Remark~\ref{r:adjunction} we get that $\Gamma^*_\infty$ has a right adjoint, namely $(\Gamma_*)_\infty$.
\end{proof}

We can now state and prove the main theorem of this section:

\begin{thm}\label{t:main shape}
For any Grothendieck site $\cC$ we have a weak equivalence in $\Pro(\cS_{\infty})$
\[|\cC|\simeq \Shp(\what{\Sh}_\infty(\cC)).\]
\end{thm}

\begin{proof}
The functor $\Gamma^*:\cS\lrar\Sh_{\Del}(\cC)$ induces a functor $\Gamma^*_{\infty}:\cS_{\infty} \lrar \left(\Sh_{\Del}(\cC)\right)_{\infty}\simeq\what{\Sh}_{\infty}(\cC)$. By Lemma~\ref{l:geometric} the functor $\Gamma^*_\infty$ is the left hand side of a geometric morphism between $\cS_\infty$ and $\what{\Sh}_{\infty}(\cC)$ and hence must coincides with $q^*$ up to equivalence by~\cite[Proposition 6.3.4.1]{Lur09}. The functor $q^*:\cS_\infty\rightarrow\what{\Sh}_\infty(\cC)$, in turn, is accessible (being a left adjoint) and commutes with finite limits, hence its prolongation to $\Pro(\cS_{\infty})$ admits a left adjoint:
\begin{equation}\label{e:adj-3}
L:\Pro(\what{\Sh}_\infty(\cC))\rightleftarrows\Pro(\cS_\infty):\Pro(q^*).
\end{equation}
By Proposition~\ref{p:weak r q functor} the functor $\Pro(\Gamma^*)_{\infty}$ is equivalent $\Pro(\Gamma^*_{\infty})$ and hence to $\Pro(q^*)$. By uniqueness of left adjoints, it follows that the adjunction
\begin{equation}\label{e:adj-2}
(L_{\Gamma^*})_{\infty}:\Pro(\Sh_{\Del}(\cC))_\infty \rightleftarrows \Pro(\cS)_\infty:\Pro(\Gamma^*)_{\infty}
\end{equation}
is equivalent to the adjunction (\ref{e:adj-3}) and so the image of $|\cC|$ under the equivalence $\Pro(\cS)_\infty \simeq \Pro(\cS_\infty)$ (which is a particular case of Corollary~\ref{c:main-topos}) is given by the object $L(\ast)$. Now for every object $X \in \what{\Sh}_\infty(\cC)$, the pro-object $L(X)$ is given, as an object in $\Fun(\cS_\infty,\cS_\infty)^{\op}$, by the formula
\[ L(X)(K) \simeq \Map_{\Pro(\cS_\infty)}(L(X),K) \simeq\Map_{\what{\Sh}_\infty(\cC)}(X,q^*(K)) .\]
In particular, the object $L(\ast)\in \Pro(\cS_\infty)$ corresponds to the functor
\[ K \mapsto \Map_{\what{\Sh}_\infty(\cC)}(\ast,q^*(K)) \simeq \Map_{\what{\Sh}_\infty(\cC)}(q^*(\ast),q^*(K)) \simeq \Map_{\cS_\infty}(\ast,q_*q^*(K)) \simeq q_*q^*(K) \]
and so we obtain a natural equivalence $L(\ast) \simeq q_* \circ q^*$ in $\Fun(\cS_\infty,\cS_\infty)^{\op}$ and consequently a natural equivalence
$|\cC|\simeq \Shp(\what{\Sh}_\infty(\cC))$
as desired.
\end{proof}

\section{Application: Several models for profinite spaces}\label{s:profinite}

In this section we apply Theorem~\ref{t:main} in order to relate the model categorical and the $\infty$-categorical aspects of profinite homotopy theory. In \S\ref{ss:isaksen} we describe a certain left Bousfield localization, due to Isaksen, of the induced model structure on the category $\Pro(\cS)$ of pro-spaces. This localization depends on a choice of a collection $K$ of Kan complexes. We identify the underlying $\infty$-category of this localization as the pro-category of a suitable $\infty$-category $(K_{\nil})_{\infty}$. In \S\ref{ss:example1} and~\ref{ss:example2} we describe explicit examples where $(K_{\nil})_{\infty}$ is equivalent to the $\infty$-category of $\pi$-finite spaces and $p$-finite spaces respectively. Finally, in \S\ref{ss:quick} we relate Isaksen's approach to that of Quick and Morel, via two direct Quillen equivalences. These Quillen equivalences appear to be new.

\subsection{Isaksen's model}\label{ss:isaksen}

Consider the category of small simplicial sets $\cS$ with the Kan-Quillen model structure. According to Theorem~\ref{t:model_Isa} the induced model structure on $\Pro(\cS)$ exists. The pro-admissibility of $\cS$ follows from the left and right properness. This model structure was first constructed in \cite{EH76} and further studied in~\cite{Isa01}, where it was called the \textbf{strict model structure}. Isaksen shows in~\cite{Isa05} that for $K$ any small set of fibrant object of $\cS$, one can form the maximal left Bousfield localization $L_K \Pro(\cS)$ of $\Pro(\cS)$ for which all the objects in $K$ are local. In order to describe the fibrant objects of $L_K\Pro(\cS)$, Isaksen defines first the class $K_{\nil}$ of $K$-nilpotent spaces. This is the smallest class of Kan complexes that is closed under homotopy pullbacks and that contains $K$ and the terminal object $\ast$. In particular, $K_{\nil}$ is closed under weak equivalences between Kan complexes. The fibrant objects of $L_K\Pro(\cS)$ are the fibrant objects in $\Pro(\cS)$ which are isomorphic to a pro-space that is levelwise in $K_{\nil}$. The weak equivalences in $L_K\Pro(\cS)$ are the maps $X \lrar Y$ in $\Pro(\cS)$ such that for any $A$ in $K$, the map
\[\Map^h_{\Pro(\cS)}(Y,A) \lrar \Map^h_{\Pro(\cS)}(X,A)\]
is a weak equivalence.

Our goal in this section is to prove that $L_K\Pro(\cS)$ is a model for the pro-category of the $\infty$-category underlying $K_{\nil}$. We say that a map in $K_{\nil}$ is a weak equivalence (resp. fibration) if it is a weak equivalence (resp. fibration) when regarded as a map of simplicial sets. Since $\cS^{\fib}$ is a category of fibrant objects and $K_{\nil} \subseteq \cS^{\fib}$ is a full subcategory which is closed under weak equivalences and pullbacks along fibrations it follows that $K_{\nil}$ inherits a structure of a category of fibrant objects.

\begin{lem}\label{l:K_nil}
The natural map
\[ (K_{\nil})_{\infty} \lrar \cS_\infty \]
is fully faithful.
\end{lem}
\begin{proof}
Since $K_{\nil} \subseteq \cS^{\fib}$ is closed under weak equivalences the natural map
\[ \uline{\Hom}_{K_{\nil}}(X,Y) \lrar \uline{\Hom}_{\cS^{\fib}}(X,Y) \]
is an isomorphism for any $X,Y \in K_{\nil}$ (see Definition~\ref{d:hom}).
\end{proof}

The main theorem of this subsection is the following:
\begin{thm}\label{t:main-isaksen}
Let $K$ be a small set of fibrant objects in $\cS$. Then the $\infty$-category $L_K\Pro(\cS)_\infty$ is naturally equivalent to $\Pro((K_{\nil})_{\infty})$.
\end{thm}
\begin{proof}
Let
\[ \iota: L_K\Pro(\cS) \lrar \Pro(\cS) \]
be the identity, considered as a right Quillen functor, and let $ \iota_\infty: (L_K\Pro(\cS))_\infty \lrar \Pro(\cS)_\infty$ be the associated functor of $\infty$-categories. We first claim that $\iota_\infty$ is fully faithful. By Corollary~\ref{c:map} it is enough to prove that if $X,Y$ are two fibrant objects of $L_K\Pro(\cS)$ (i.e., fibrant $K$-local objects of $\Pro(\cS)$) then the induced map
\begin{equation}\label{e:induced}
\uline{\Hom}_{L_K\Pro(\cS)}(X,Y) \lrar \uline{\Hom}_{\Pro(\cS)}(X,Y)
\end{equation}
induces a weak equivalence on nerves. But since the classes of trivial fibrations are the same for $\Pro(\cS)$ and $L_K\Pro(\cS)$ we see that this map~\ref{e:induced} is in fact an \textbf{isomorphism}, and hence in particular a weak equivalence after taking nerves. It follows that $\iota_\infty$ is fully-faithful.

By Theorem~\ref{t:main} we have a natural equivalence of $\infty$-categories
\[ \Pro(\cS)_\infty \x{\simeq}{\lrar} \Pro(\cS_\infty). \]
By Lemma~\ref{l:K_nil} the inclusion $(K_{\nil})_{\infty} \hrar \cS_\infty$ is fully faithful and so it follows from~\cite{Lur09} that the induced functor
\[ \Pro((K_{\nil})_{\infty}) \lrar \Pro(\cS_\infty) \]
is fully faithful. Hence in order to finish the proof it suffices to show that the essential image of the composed functor
\[ \iota'_\infty: (L_K\Pro(\cS))_\infty \lrar \Pro(\cS)_\infty \x{\simeq}{\lrar} \Pro(\cS_\infty) \]
coincides with the essential image of $\Pro((K_{\nil})_{\infty})$.

Now the essential image of $\iota_\infty'$ is given by the images of those objects in $\Pro(\cS)$ which are equivalent in $\Pro(\cS)$ to a fibrant object of $L_K\Pro(\cS)$. According to~\cite{Isa05} the latter are exactly those fibrant objects of $\Pro(\cS)$ which belong to the essential image of $\Pro(K_{\nil})$. We hence see that the essential image of $\iota'_\infty$ is contained in the essential image of $\Pro((K_{\nil})_{\infty}) \lrar \Pro(\cS_\infty)$. On the other hand, the essential image of $\iota_\infty'$ clearly contains $(K_{\nil})_{\infty}$. Since $L_K\Pro(\cS)$ is a model category we know by Theorem~\ref{t:model-complete} that$(L_K\Pro(\cS))_\infty$ has all small limits. Since $\iota_\infty$ is induced by a right Quillen functor we get from Remark~\ref{r:adjunction} and~\cite[Proposition 5.2.3.5]{Lur09} that $\iota_\infty$ preserves limits. It hence follows that the essential image of $\iota_\infty'$ is closed under small limits. By Lemma~\ref{l:generate} every object in $\Pro((K_{\nil})_{\infty})$ is a small (and even cofiltered) limit of simple objects and hence the essential image of $\iota_\infty'$ coincides with the essential image of $\Pro((K_{\nil})_{\infty})$.
\end{proof}

\subsection{Example: the $\infty$-category of $\pi$-finite spaces}\label{ss:example1}
In this subsection we show that for a specific choice of $K$, Isaksen's model category $L_K(\Pro(\cS))$ is a model for the $\infty$-category of \textbf{profinite spaces}. Let us begin with the proper definitions:

\begin{define}\label{d:profinite}
Let $X \in \cS_\infty$ be a space. We say that $X$ is \textbf{$\pi$-finite} if it has finitely many connected components and for each $x \in X$ the homotopy groups $\pi_n(X,x)$ are finite and vanish for large enough $n$. We denote by $\cS^{\pi}_\infty \subseteq \cS_\infty$ the full subcategory spanned by $\pi$-finite spaces. A \textbf{profinite space} is a pro-object in the $\infty$-category $\cS^{\pi}_\infty$. We refer to the $\infty$-category $\Pro\left(\cS^{\pi}_\infty\right)$ as the \textbf{$\infty$-category of profinite spaces}.
\end{define}

\begin{rem}\label{r:abuse}
By abuse of notation we shall also say that a simplicial set $X$ is $\pi$-finite if its image in $\cS_\infty$ is $\pi$-finite.
\end{rem}

In order to identify a suitable candidate for $K$ we first need to establish some terminology. Let $\Del_{\leq n} \subseteq \Del$ denote the full subcategory spanned by the objects $[0],\ldots,[n] \in \Del$. We have an adjunction
\[\tau_n:\Fun\left(\Del^{\op},\cD\right) \adj \Fun\left(\Del^{\op}_{\leq n},\cD\right):\cosk_n\]
where $\tau_n$ is given by restriction functor and $\cosk_n$ by right Kan extension. We say that a simplicial object $X \in \Fun\left(\Del^{\op},\cD\right)$ is \textbf{$n$-coskeletal} if the unit map $X \lrar \cosk_n \tau_nX$ is an isomorphism. We say that $X$ is \textbf{coskeletal} if it is $n$-coskeletal for some $n$.

\begin{define}
Let $X \in \cS$ be a simplicial set. We say that $X$ is \textbf{$\tau_n$-finite} if it is levelwise finite and $n$-coskeletal. We say that $X$ is \textbf{$\tau$-finite} if it is $\tau_n$-finite for some $n \geq 0$. We denote by $\cS_{\tau} \subseteq \cS$ the full subcategory spanned by $\tau$-finite simplicial sets. We note that $\cS_{\tau}$ is essentially small.
\end{define}

\begin{lem}\label{l:cofinite}
If $X$ is a minimal Kan complex then $X$ is $\tau$-finite if and only if it is $\pi$-finite (i.e., if the associated object in $\cS_\infty$ is $\pi$-finite, see Remark~\ref{r:abuse}).
\end{lem}

\begin{proof}
Since $X$ is minimal it follows that $X_0$ is in bijection with $\pi_0(X)$ and hence the former is finite if and only if the latter is. Furthermore, for each $x \in X_0$ and each $n \geq 1$ the minimality of $X$ implies that for every map $\tau: \partial \Del^n \lrar X$ such that $\tau(\Del^{\{0\}}) = x$, the set of maps $\sig: \Del^n \lrar X$ such that $\sig|_{\partial \Del^n} = \tau$ are (unnaturally) in bijection with $\pi_n(X,x)$. This implies that $X$ is levelwise finite if and only if all the homotopy groups of $X$ are finite. This also implies that if $X$ is coskeletal then its homotopy groups vanish in large enough degree. On the other hand, if the homotopy groups of $X$ vanish for large enough degree then there exists a $k$ such that for every $n > k$ the fibers of the Kan fibration $p_n:X^{\Del^n} \lrar X^{\partial \Del^n}$ are weakly contractible. Since $X$ is minimal we may then deduce that $p_n$ is an isomorphism. Since this is true for every $n > k$ this implies that $X$ is $k$-coskeletal. We hence conclude that $X$ is $\tau$-finite if and only if it is $\pi$-finite.
\end{proof}

\begin{rem}
If $X$ is not assumed to be minimal but only Kan then $X$ being $\tau$-finite implies that $X$ is $\pi$-finite, but not the other way around. If one removes the assumption that $X$ is Kan then there is no implication in any direction.
\end{rem}

\begin{cor}
Let $X$ be a simplicial set. Then $X$ is $\pi$-finite if and only if $X$ is equivalent to a minimal Kan $\tau$-finite simplicial set.
\end{cor}
\begin{proof}
This follows from Lemma~\ref{l:cofinite} and the fact that any simplicial set is equivalent to one that is minimal Kan.
\end{proof}

Let us now recall the ``basic building blocks'' of $\pi$-finite spaces. Given a set $S$, we denote by $\K(S,0)$ the set $S$ considered as a simplicial set. For a group $G$, we denote by $\B G$ the simplicial set
\[ (\B G)_n = G^n \]
The simplicial set $\B G$ can also be identified with the nerve of the groupoid with one object and automorphism set $G$. It is often referred to as the \textbf{classifying space} of $G$. We denote by $\E G$ the simplicial set given by
\[ (\E G)_n = G^{n+1} \]
The simplicial set $\E G$ may also be identified with $\cosk_0(G)$. The simplicial set $\E G$ is weakly contractible and carries a free action of $G$ (induced by the free action of $G$ on itself), such that the quotient may be naturally identified with $\B G$, and the quotient map $\E G \lrar \B G$ is a $G$-covering.

Now recall the \textbf{Dold-Kan correspondence}, which is given by an adjunction
\[ \Gam: \Ch_{\geq 0} \adj \Ab^{\Del^{\op}} : N \]
such that the unit and counit are natural isomorphisms (\cite[Corollary III.2.3]{GJ99}). We note that in this case the functor $\Gam$ is simultaneously also the right adjoint of $N$. Furthermore, the homotopy groups of $\Gam(C)$ can be naturally identified with the \textbf{homology groups} of $C$.

For every abelian group $A$ and every $n \geq 2$ we denote by $\K(A,n)$ the simplicial abelian group $\Gam(A[n])$ where $A[n]$ is the chain complex which has $A$ at degree $n$ and $0$ everywhere else. Then $\K(A,n)$ has a unique vertex $x$ and $\pi_k(\K(A,n),x) = 0$ if $k \neq 0$ and $\pi_n(\K(A,n),x) = A$. Though $\K(A,n)$ is a simplicial abelian group we will only treat it as a simplicial set (without any explicit reference to the forgetful functor). Let $\L(A,n) \lrar \K(A,n)$ be a minimal fibration such that $\L(A,n)$ is weakly contractible. This property characterizes $\L(A,n)$ up to an isomorphism over $\K(A,n)$. There is also an explicit functorial construction of $\L(A,n)$ as $W\K(A,n-1)$, where $W$ is the functor described in~\cite[\S V.4]{GJ99} (and whose construction is originally due to Kan).

Now let $G$ be a group and $A$ a $G$-module. Then $\K(A,n)$ inherits a natural action of $G$ and $\L(A,n)$ can be endowed with a compatible action (alternatively, $\L(A,n)$ inherits a natural action via the functor $W$). We denote by $\K(A,n)_{hG} = (\E G \times \K(A,n))/G$ the (standard model of the) homotopy quotient of $\K(A,n)$ by $G$ and similarly $\L(A,n)_{hG} = (\E G \times \L(A,n))/G$.

\begin{lem}\label{l:minimal-kan}
For every $S,G,A$ and $n \geq 2$ as above the simplicial sets $\K(S,0), \B G$ and $\K(A,n)$ and $\K(A,n)_{hG}$ are minimal Kan complexes, and the maps $\E G \lrar \B G, \L(A,n) \lrar \K(A,n)$ and $\L(A,n)_{hG} \lrar \K(A,n)_{hG}$ are minimal Kan fibrations.
\end{lem}

\begin{proof}
The fact that $\K(S,0)$ is minimal Kan complex is clear, and $\B G$ is Kan because it is a nerve of a groupoid. Now, since $\B G$ is $2$-coskeletal and reduced, in order to check that it is also minimal, it suffices to check that if $\sig,\tau: \Del^1 \lrar \B G$ are two edges which are homotopic relative to $\partial \Del^1$ then they are equal. But this is clear since $\B G$ is the nerve of a discrete groupoid. In order to check that the map $p:\E G \lrar \B G$ is a minimal fibration it is enough to note that $\E G$ is $0$-coskeletal and the fibers of $p$ are discrete. Finally, by~\cite[Lemma III.2.21]{GJ99} the simplicial set $\K(A,n)$ is minimal and the map $\L(A,n) \lrar \K(A,n)$ is a minimal fibration. The analogous claims for $\K(A,n)_{hG}$ and $\L(A,n)_{hG} \lrar \K(A,n)_{hG}$ follow from~\cite[Lemma VI.4.2]{GJ99}.
\end{proof}

\begin{define}
Let $K^{\pi} \subseteq \cS^{\fib}$ be a (small) set of representatives of all isomorphism classes of objects of the form $\K(S,0)$, $\B G$, $\K(A,n)_{hG}$ and $\L(A,n)_{hG}$ for all finite sets $S$, finite groups $G$, and finite $G$-modules $A$.
\end{define}

\begin{rem}
By construction all the objects in $K^{\pi}$ are $\pi$-finite. Combining Lemma~\ref{l:minimal-kan} with Lemma~\ref{l:cofinite} we may also conclude that all the objects in $K^{\pi}$ are $\tau$-finite.
\end{rem}

We now explain in what way the spaces in $K^{\pi}$ are the building blocks for all $\pi$-finite spaces.

\begin{prop}\label{p:pi-finite}
Every object in $K^{\pi}_{\nil}$ is $\pi$-finite. Conversely, every $\pi$-finite space is a retract of an object in $K^{\pi}_{\nil}$.
\end{prop}

\begin{proof}
Since the class of Kan complexes which are $\pi$-finite contains $K^{\pi}$ and $\ast$ and is closed under homotopy pullbacks and retracts it contains $K^{\pi}_{\nil}$ by definition. On the other hand, let $X$ be a $\pi$-finite simplicial set. We wish to show that $X$ is a retract of an object in $K^{\pi}_{\nil}$. We first observe that we may assume without loss of generality that $X$ is connected. Indeed, if $X = X_0 \coprod X_1$ with $X_0,X_1 \neq \emptyset$ then $X$ is a retract of $X_0 \times X_1 \times \left[\Del^0 \coprod \Del^0\right]$, and $\Del^0 \coprod \Del^0 = S(\{0,1\}))$ belongs to $K^{\pi}$. It follows that if $X_0$ and $X_1$ are retracts of objects in $K^{\pi}_{\nil}$ then so is $X$. Hence, it suffices to prove the claim when $X$ is connected.

By possibly replacing $X$ with a minimal model we assume that $X$ is minimal Kan. Let $\{X(n)\}$ be the Moore-Postnikov tower for $X$. Since $X$ is minimal we have $X_0 = \{x_0\}$ and $X(1) = \B G$ with $G = \pi_1(X,x_0)$ (see~\cite[Proposition 3.8]{GJ99}). We may hence conclude that $X(1) \in K^{\pi}_{\nil}$. Now according to~\cite[Corollary 5.13]{GJ99} we have, for each $n \geq 2$ a pullback square of the form
\[ \xymatrix{
X(n) \ar[r]\ar[d] & \L(\pi_n(X,x),n+1)_{hG} \ar[d] \\
X(n-1) \ar[r] & \K(\pi_n(X,x),n+1)_{hG}
}\]
Hence $X(n-1) \in K^{\pi}_{\nil}$ implies that $X(n) \in K^{\pi}_{\nil}$, and by induction $X(k) \in K^{\pi}_{\nil}$ for every $k \geq 0$. Since $X$ is minimal and $\pi$-finite Lemma~\ref{l:cofinite} implies that $X$ is $\tau$-finite. Hence there exists a $k$ such that $X \cong X(k)$ and the desired result follows.
\end{proof}

By Proposition~\ref{p:pi-finite} the fully faithful inclusion $(K^{\pi}_{\nil})_\infty \lrar \cS_\infty$ of Lemma~\ref{l:K_nil} factors through a fully faithful inclusion $\iota_\pi: (K^{\pi}_{\nil})_\infty \lrar \cS^\pi_{\infty}$, and every object in $\cS^\pi_{\infty}$ is a retract of an object in the essential image of $\iota_\pi$. This fact has the following implication:

\begin{cor}\label{c:pro-pi-equiv}
The induced map
\[ \Pro(\iota_\pi): \Pro((K^{\pi}_{\nil})_{\infty}) \lrar \Pro\left(\cS^{\pi}_\infty\right) \]
is an equivalence of $\infty$-categories.
\end{cor}
\begin{proof}
By~\cite[Proposition 5.3.5.11(1)]{Lur09} the map $\Pro(\iota_\pi)$ is fully faithful. Now let $X$ be a $\pi$-finite space. By Proposition~\ref{p:pi-finite} there is a retract diagram $X \x{i}{\lrar} Y \x{r}{\lrar} X$ with $Y \in K^{\pi}_{\nil}$. Let $f = ir: Y \lrar Y$ and consider the pro-object $Y^f$ given by
\[ \hdots \x{f}{\lrar} Y \x{f}{\lrar} Y \x{f}{\lrar} \hdots \x{f}{\lrar} Y .\]
The maps $i$ and $r$ can then be used to produce an equivalence
\[ X \simeq Y^f \]
in $\Pro(\cS^{\pi}_\infty)$. This shows that the $\Pro(\iota_\pi)$ is essentially surjective and hence an equivalence.
\end{proof}

Applying Theorem~\ref{t:main-isaksen} we may now conclude that Isaksen's model category $L_{K^{\pi}}\Pro(\cS)_\infty$ is indeed a model for the $\infty$-category of profinite spaces. More precisely, we have the following

\begin{cor}\label{c:isaksen-finite}
The underlying $\infty$-category of $L_{K^{\pi}}\Pro(\cS)$ is naturally equivalent to the $\infty$-category $\Pro(\cS^{\pi}_{\infty})$ of profinite spaces.
\end{cor}

\subsection{Example: the $\infty$-category of pro-$p$ spaces}\label{ss:example2}

In this subsection we will show that for a specific choice of $K$, Isaksen's model category $L_K(\Pro(\cS))$ is a model for a suitable $\infty$-category of \textbf{pro-$p$ spaces}. We begin with the proper definitions.

\begin{define}[{\cite[Definition 2.4.1, Definition 3.1.12]{Lur11b}}]\label{d:pro-p}
Let $X \in \cS_\infty$ be a space and $p$ a prime number. We say that $X$ is \textbf{$p$-finite} if it has finitely many connected components and for each $x \in X$ the homotopy groups $\pi_n(X,x)$ are finite $p$-groups which vanish for large enough $n$. We denote by $\cS^{p}_\infty \subseteq \cS_\infty$ the full subcategory spanned by $p$-finite spaces. A \textbf{pro-$p$ space} is a pro-object in the $\infty$-category $\cS^{p}_\infty$. We refer to the $\infty$-category $\Pro\left(\cS^{p}_\infty\right)$ as the \textbf{$\infty$-category of pro-$p$ spaces}.
\end{define}

\begin{define}
Let $K^p$ be a (small) set of isomorphism representatives for all $\K(S,0), \B\ZZ/p$ and $\K(\ZZ/p,n)$ for all finite sets $S$ and all $n \geq 2$.
\end{define}

As in Lemma~\ref{l:K_nil} we obtain a fully faithful inclusion $(K^{p}_{\nil})_\infty \lrar S_\infty$. Out next goal is to identify its essential image. We first recall a few facts about nilpotent spaces.

Let $G$ be a group. Recall that the \textbf{upper central series} of $G$ is a sequence of subgroups
\[\{e\}=Z_0(G)\subset Z_1(G)\subset Z_2(G)\subset\ldots\subset G\]
defined inductively by $Z_0(G)=\{e\}$ and $Z_i(G)=\{g\in g |[g,G]\subset Z_{i-1}(G)\}$. In particular, $Z_1(G)$ is the center of $G$. Alternatively, one can define $Z_i(G)$ as the inverse image along the map $G\lrar G/Z_{i-1}(G)$ of the center of $G/Z_{i-1}(G)$.

\begin{define}\
\begin{enumerate}
\item
A group $G$ is called \textbf{nilpotent} if $Z_n(G) = G$ for some $n$.
\item
A $G$-module $M$ is called \textbf{nilpotent} if $M$ has a finite filtration by $G$-submodules
\[0=M_n\subset M_{n-1}\subset \ldots\subset M_1\subset M_0=M\]
such that the induced action of $G$ on each $M_i/M_{i+1}$ is trivial.
\item
A space $X$ is called \textbf{nilpotent} if for each $x \in X$, the group $\pi_1(X,x)$ is nilpotent and for each $n \geq 2$ the abelian group $\pi_n(X,x)$ is a nilpotent $\pi_1(X,x)$-module.
\end{enumerate}
\end{define}

We recall the following well-known group theoretical results:

\begin{prop}
Let $G$ be a finite $p$-group. Then:
\begin{enumerate}
\item
$G$ is nilpotent.
\item
Let $M$ be a finite abelian $p$-group equipped with an action of $G$. Then $M$ is nilpotent $G$-module.
\end{enumerate}
\end{prop}

\begin{proof}
The first claim is~\cite[IX \S 1, Corollary of Theorem I]{Ser62}. The second claim follows from~\cite[IX \S 1, Lemme II]{Ser62} via a straightforward inductive argument.
\end{proof}

\begin{lem}\label{l:p-abelian}
Let $A$ be a finite abelian $p$-group. Then $\K(A,n)$ belongs to $K^p_{\nil}$ for $n\geq 1$.
\end{lem}

\begin{proof}
Let $\cG$ be the class of groups $A$ such $K(A,n) \in K^p_{\nil}$ for every $n \geq 1$. By construction $\cG$ contains the group $\ZZ/p$. Now let $0\lrar A\lrar B\lrar C\lrar 0$ be a short exact sequence of abelian groups such that $A,C \in \cG$ and let $n \geq 1$ be an integer. We then have a homotopy pullback square
\[\xymatrix{
\K(B,n)\ar[d]\ar[r]&\L(A,n+1)\ar[d]\\
\K(C,n)\ar^{p}[r]&\K(A,n+1)
}
\]
where the $p:\K(C,n)\lrar \K(A,n+1)$ is the map classifying the principal $\K(A,n)$-fibration $\K(B,n)\lrar\K(C,n)$. Since $\L(A,n+1)$ is contractible and $\K(C,n)$ and $\K(A,n+1)$ are in $K^p_{\nil}$ we conclude that $\K(B,n) \in K^p_{\nil}$ as well. Since this is true for every $n \geq 1$ it follows that $B \in \cG$. It follows that the class $\cG$ is closed under extensions and hence contains all finite abelian $p$-groups.
\end{proof}

We can now prove the $p$-finite analogue of proposition~\ref{p:pi-finite}.

\begin{prop}\label{p:p-finite}
Every object of $K^p_{\nil}$ is $p$-finite. Conversely, every $p$-finite space is a retract of an object of $K^p_{\nil}$.
\end{prop}
\begin{proof}
Since the class of Kan complexes which are $p$-finite contains $K^{p}$ and $\ast$ and is closed under homotopy pullbacks and retracts it contains $K^{p}_{\nil}$ by definition. Now let $X$ be a $p$-finite space. As in the proof of~\ref{p:pi-finite} we may assume without loss of generality that $X$ is a connected minimal Kan complex. By Lemma~\ref{l:cofinite} $X$ is $\tau$-finite. According to~\cite[Proposition V.6.1]{GJ99}, we can refine the Postnikov tower of $X$ into a finite sequence of maps
\[X=X_k\lrar X_{k-1}\lrar\ldots\lrar X_1\lrar X_0=\ast\]
in which the map $X_i\lrar X_{i-1}$ fits in a homotopy pullback square
\[\xymatrix{
X_i\ar[d]\ar[r]& \L(A_i,n_i)\ar[d]\\
X_{i-1}\ar[r]& \K(A_i,n_i)
}
\]
where $n_i\geq 1$ is an integer and $A_i$ is an abelian subquotient of one of the homotopy group of $X$. Since $X$ is $p$-finite every $A_i$ is a finite abelian $p$-group. Applying Lemma~\ref{l:p-abelian} inductively we may conclude that each $X_i$ is in $K^p_{\nil}$, and hence $X \in K^p_{\nil}$ as desired.
\end{proof}

By Proposition~\ref{p:p-finite} the fully faithful inclusion $(K^p_{\nil})_\infty \lrar \cS_\infty$ factors through a fully faithful inclusion $\iota_p: K^{p}_{\nil} \lrar \cS^p_{\infty}$, and every object in $\cS^p_{\infty}$ is a retract of an object in the essential image of $\iota_p$. As for the case of profinite spaces we hence obtain an equivalence after passing to pro-categories:

\begin{cor}
The induced map
\[ \Pro(\iota_p): \Pro((K^{p}_{\nil})_{\infty}) \lrar \Pro\left(\cS^{p}_\infty\right) \]
is an equivalence of $\infty$-categories.
\end{cor}
\begin{proof}
The proof is identical to the proof of Corollary~\ref{c:pro-pi-equiv}.
\end{proof}

Applying Theorem~\ref{t:main-isaksen} we may now conclude that Isaksen's model category $L_{K^p}\Pro(\cS)_\infty$ is a model for the $\infty$-category of pro-$p$ spaces. More precisely, we have the following

\begin{cor}\label{c:isaksen-p finite}
The underlying $\infty$-category of $L_{K^{p}}\Pro(\cS)$ is naturally equivalent to the $\infty$-category $\Pro(\cS^{p}_{\infty})$ of pro-$p$ spaces.
\end{cor}

\subsection{Comparison with Quick and Morel model structures}\label{ss:quick}
Let $\cF \subseteq \Set$ denote the full subcategory spanned by finite sets and let $\hat{\cS}$ denote the category of simplicial objects in $\Pro(\cF)$. In~\cite{Qui11} Quick constructs a model structure on $\hat{\cS}$ in order to model profinite homotopy theory. This model structure is fibrantly generated with sets of generating fibrations denoted by $P$ and set of generating trivial fibrations denoted by $Q$. We note that the domain and codomain of any map in $P$ or $Q$ is isomorphic to an object of $K^{\pi}$. Furthermore, for any object $X \in K^{\pi}$, the map $X \lrar \ast$ is either contained in $P \cup Q$ or is a composition of two such maps. In particular, every object in $K^{\pi}$ is fibrant in Quick's model structure.

In this subsection we will construct a Quillen equivalence between $\hat{\cS}$ and Isaksen's model category $L_{K^{\pi}}\Pro(\cS)$. Corollary~\ref{c:isaksen-finite} then implies that $\hat{\cS}$ is indeed a model for the $\infty$-category $\Pro(\cS^{\pi}_\infty)$ of profinite spaces.

The following proposition asserts that the category $\hat{\cS}$ can be naturally identified with the pro-category of $\cS_{\cofin}$. This makes it easier to compare it with the Isaksen model structure considered in the previous subsection.

\begin{prop}
The natural full inclusion $\iota:\cS_{\cofin} \lrar \hat{\cS}$ induces an equivalence of categories
\[ \Pro(\cS_{\cofin}) \lrar \hat{\cS} \]
\end{prop}

\begin{proof}
According to (the classical version of)~\cite[5.4.5.1]{Lur09} what we need to check is that $\tau$-finite simplicial sets are $\omega$-cocompact in $\hat{\cS}$, that every object of $\hat{\cS}$ is a cofiltered limit of $\tau$-finite simplicial sets and that the inclusion $\cS_{\cofin}\lrar\hat{\cS}$ is fully faithful.

We first show that the functor $\cS_{\cofin}\lrar\hat{\cS}$ is fully faithful. This functor factors as a composition
\[ \cS_{\cofin}\lrar\Fun(\Delta^{\op},\cF)\lrar\Fun(\Delta^\op,\Pro(\cF))=\hat{\cS} .\]
Now the first functor is fully faithful by definition of $\cS_{\cofin}$ and the second functor is fully faithful because $\cF \lrar \Pro(\cF)$ is fully faithful. We hence obtain that $\cS_{\cofin} \lrar \hat{\cS}$ is fully faithful.

Next, we show that any object $X \in \hat{\cS}$ is a cofiltered limit of $\tau$-finite simplicial sets. Since the natural map
\[ X \lrar \lim_n \cosk_n(\tau_n(X)) \]
is an isomorphism it is enough to show that for every $n \geq 0$, every $n$-coskeletal object in $\hat{\cS}$ is a cofiltered limit of $\tau_n$-finite simplicial sets. Unwinding the definitions, we wish to show that any functor $\Del^{\op}_{\leq n} \lrar \Pro(\cF)$ is a cofiltered limit of functors $\Del^{\op}_{\leq n} \lrar \cF$. Since the category $\cF$ is essentially small and admits finite limits and since the category $\Del^{\op}_{\leq n}$ is finite we may use~\cite[\S 4]{Mey80} to deduce that the inclusion $\cF \subseteq \Pro(\cF)$ induces an equivalence of categories
\begin{equation}\label{e:meyer}
\Pro\left(\Fun\left(\Del^{\op}_{\leq n}, \cF\right)\right) \x{\simeq}{\lrar} \Fun\left(\Del^{\op}_{\leq n}, \Pro(\cF)\right) .
\end{equation}
It hence follows that every object in $\Fun\left(\Del^{\op}_{\leq n}, \Pro(\cF)\right)$ is a cofiltered limit of objects in $\Fun\left(\Del^{\op}_{\leq n}, \cF\right)$, as desired.

Finally, we show that every $\tau$-finite simplicial set is $\omega$-cocompact in $\hat{\cS}$. Let $X$ be a $\tau$-finite simplicial set. We need to show that the functor $\Hom_{\hat{\cS}}(-,X)$ sends cofiltered limits to filtered colimits. Let $n$ be such that $X$ is $n$-coskeletal. Then $\Hom_{\hat{\cS}}(Z,X) \cong \Hom_{\hat{\cS}}(\tau_n(Z),\tau_n(X))$. Since the functor $\tau_n$ preserves limits it is enough to show that $\tau_n(X)$ is $\omega$-cocompact in $\Fun\left(\Del^{\op}_{\leq n}, \Pro(\cF)\right)$. But this again follows from the equivalence~\ref{e:meyer}.
\end{proof}

Using the equivalence of categories $\hat{\cS} \cong \Pro(\cS_{\cofin})$ we may consider Quick's model structure as a model structure on $\Pro(\cS_{\cofin})$, which is fibrantly generated by the sets $P$ and $Q$ described in~\cite[Theorem 2.3]{Qui11}\footnote{Note that there is a small mistake in the generating fibrations in~\cite{Qui11}. An updated version of this paper can be found on the author's webpage \url{http://www.math.ntnu.no/~gereonq/}. In this version the relevant result is Theorem 2.10.} (where we consider now $P$ and $Q$ as sets of maps in $\cS_{\cofin} \subseteq \Pro(\cS_{\cofin})$). We note that $\cS_{\cofin}$ is \textbf{not} a weak fibration category and that this model structure is \textbf{not} a particular case of the model structure of Theorem~\ref{t:model-2}. This follows, in particular, from the following observation:
\begin{prop}\label{p:counter}
There exist maps in $\hat{\cS} \cong \Pro(\cS_{\cofin})$ which are weak equivalences with respect to Quick's model structure but that are not isomorphic to levelwise weak equivalences.
\end{prop}
\begin{proof}
Let $\Ab$ denote the category of abelian groups. The first homology functor $H_1: \cS_{\tau} \lrar \Ab$ induces a functor $\Pro(H_1): \Pro(\cS_{\tau}) \lrar \Pro(\Ab)$. We now note that if $f: X \lrar Y$ is a map in $\cS_{\tau}$ which is a weak equivalence in Quick's model structure (when considered as a map in $\what{\cS}$) then $f$ induces an isomorphism on homology with all finite coefficients and hence an isomorphism on $H_1$ by the universal coefficients theorem (recall that $H_1(X)$ is finitely generated for every $X \in \cS_\tau$). It then follows that every levelwise weak equivalence in $\Pro(\cS_{\tau})$ is mapped by $\Pro(H_1)$ to an isomorphism in $\Pro(\Ab)$. It will hence suffice to exhibit a weak equivalence in $\what{\cS}$ which is not mapped to an isomorphism in $\Pro(\Ab)$.

Since the nerve of any category is $2$-coskeletal it follows that $\cS_{\tau}$ contains the nerve of any finite category. In particular, $\cS_\tau$ contains the nerve of the finite groupoid $\mathrm{B}\ZZ/n$ with one object and automorphism group $\ZZ/n$, as well as the nerve of the category $\cI$ with two objects $0,1$ and two non-identity morphisms $\alp,\bet$, both going from $0$ to $1$. We note that $\Ne\cI$ is weakly equivalent to $S^1$. Now for every $n$ we have a functor $f_n: \cI \lrar \mathrm{B}\ZZ/n$ which sends $\alp$ to $1 \in \ZZ/n$ and $\bet$ to $0 \in \ZZ/n$. Furthermore, if $n | n'$ then the quotient map $\ZZ/n' \lrar \ZZ/n$ which sends $1 \in \ZZ/n'$ to $1 \in \ZZ/n$ is compatible with $f_n$ and $f_n'$. We may hence assemble the nerves $\Ne\mathrm{B}\ZZ/n$ into a pro-object $\{\Ne\mathrm{B}\ZZ/n\}_{n \in \NN} \in \Pro(\cS_{\tau})$ indexed by the inverse poset $(\NN,|)$, in which case the maps $f_n$ determine a map $F:\Ne \cI \lrar \{\Ne\mathrm{B}\ZZ/n\}_{n \in \NN}$ in $\Pro(\cS_{\tau})$. Now $\pi_1(\Ne\cI) = \pi_1(S^1) \cong \ZZ$ and for each $n$ the map $\Ne f_n: \Ne\cI \lrar \Ne\mathrm{B}\ZZ/n$ induces the natural quotient $\ZZ \lrar \ZZ/n$ on the level of homotopy groups. The map $F$ is hence a model for the profinite completion of the circle, and as such is a weak equivalence in Quick's model structure. However, the corresponding map $\ZZ \lrar \{\ZZ/n\}_{n \in \NN}$ in $\Pro(\Ab)$ is not an isomorphism.
\end{proof}

\begin{rem}\label{r:Raptis}
The opposite model category $\what{\cS}^{\op}$ is an $\omega$-combinatorial model category: the underlying category $\what{\cS}^{\op} \cong \Ind(\cS^{\op}_\tau)$ is $\om$-presentable and the generating cofibrations and trivial cofibrations have their domains and codomains in $\cS^{\op}_{\tau}$.  However, the full subcategory $W \subseteq (\what{\cS}^{\op})^{[1]}$ spanned by weak equivalences is not an $\omega$-accessible category. Indeed, the $\omega$-compact objects in $W$ are the weak equivalences between objects of $\cS^{\op}_{\tau}$ and by Proposition~\ref{p:counter} not all weak equivalences in $\what{\cS}^{\op}$ are filtered colimits of such. This settles negatively a question raised by G. Raptis.
\end{rem}

Since the inclusion $\vphi:\cS_{\cofin} \lrar \cS$ is fully faithful and preserves finite limits it follows that the induced functor
\[ \Phi: \Pro(\cS_{\cofin}) \lrar \Pro(\cS) \]
is fully faithful and preserves all limits. The functor $\Phi$ admits a left adjoint
\[ \Psi: \Pro(\cS) \lrar \Pro(\cS_{\cofin}) \]
whose value on simple objects $X \in \cS$ is given by
\[ \Psi(X) = \{X'\}_{(X \lrar X') \in (\cS_{\cofin})_{X/}} \]

\begin{rem}\label{r:counit}
Since $\Phi$ is fully faithful we see that for every $X \in \Pro(\cS_{\cofin})$ the counit map
\[ \Psi(\Phi(X)) \lrar X \]
is an \textbf{isomorphism}.
\end{rem}

\begin{prop}
The adjunction
\[ \Psi: \Pro(\cS) \adj \Pro(\cS_{\cofin}) : \Phi \]
is a Quillen adjunction between Isaksen's strict model structure on the left, and Quick's model structure on the right.
\end{prop}

\begin{proof}
We need to check that $\Phi$ preserves fibrations and trivial fibrations. Since the model structure on $\Pro(\cS_{\cofin})$ is fibrantly generated by $P$ and $Q$, which are sets of maps in $\cS_{\cofin}$, it is enough to check that all the maps in $P$ are Kan fibrations of simplicial sets and all the maps in $Q$ are trivial Kan fibrations. This fact can be verified directly by examining the definition of $P$ and $Q$.
\end{proof}

\begin{lem}\label{l:isaksen quick quillen}
The Quillen adjunction of the previous proposition descends to a Quillen adjunction
\[ \Psi_{K^\pi}: L_{K^\pi}\Pro(\cS) \adj \Pro(\cS_{\cofin}) : \Phi_{K^\pi}. \]
\end{lem}

\begin{proof}
We need to verify that $\Phi_{K^\pi}$ is still a right Quillen functor. Since the trivial fibrations in $L_{K^\pi}\Pro(\cS)$ are the same as the trivial fibrations in $\Pro(\cS)$ it is enough to check that all the maps in $\Phi(P)$ are fibrations in $L_{K^\pi}\Pro(\cS)$. We now observe that the domain and codomain of every map in $P$ is in $K^{\pi}$ and hence $K^{\pi}$-local. By~\cite[Proposition 3.3.16]{Hir03} the maps in $P$ are also fibrations in $L_{K^\pi}\Pro(\cS)$.
\end{proof}

\begin{lem}\label{l:we}
A map $f: X \lrar Y$ in $\Pro(\cS)$ is an equivalence in $L_{K^{\pi}}\Pro(\cS)$ if and only if $\Psi(f)$ is an equivalence in $\Pro(\cS_{\cofin})$.
\end{lem}
\begin{proof}
From Theorem \ref{t:model_Isa} (3) it follows that every object in $\Pro(\cS)$ (and hence in $L_{K^\pi}\Pro(\cS)$) is cofibrant. Thus, by Lemma \ref{l:isaksen quick quillen}, $\Psi_{K^\pi}$ must preserve weak equivalences. It hence suffices to show that $\Psi$ detects weak equivalences.

By definition the weak equivalences in $L_{K^\pi}\Pro(\cS)$ are exactly the maps $f: X \lrar Y$ such that the induced map
\[ \Map^h_{\Pro(\cS)}(Y,A) \lrar \Map^h_{\Pro(\cS)}(X,A) \]
is a weak equivalence for every $A \in K^{\pi}$.
Since $A$ is a fibrant simplicial set it is fibrant in $\Pro(\cS)$. On the other hand, every $A \in K^{\pi}$ is fibrant in Quick's model structure. By adjunction we get for every $X \in \Pro(\cS)$ a natural weak equivalence
\[ \Map^h_{\Pro(\cS)}(X,A) = \Map^h_{\Pro(\cS)}(X,\Phi(A)) \simeq \Map^h_{\Pro(\cS_{\cofin})}(\Psi(X),A) \]
It follows that if $f: X \lrar Y$ is a map such that $\Psi(f)$ is a weak equivalence in $\Pro(\cS_{\cofin})$ then $f$ is a weak equivalence in $L_{K^\pi}\Pro(\cS)$.
\end{proof}

\begin{thm}\label{t:isaksen-is-quick}
The Quillen adjunction
\[ \Psi: \Pro(\cS) \adj \Pro(\cS_{\cofin}) : \Phi \]
descends to a \textbf{Quillen equivalence}
\[ \Psi_{K^\pi}: L_{K^\pi}\Pro(\cS) \adj \Pro(\cS_{\cofin}) : \Phi_{K^\pi}. \]
\end{thm}

\begin{proof}
By Lemma \ref{l:isaksen quick quillen}, the adjunction $\Psi_{K^\pi} \dashv \Phi_{K^\pi}$ is a Quillen adjunction. In order to show that it is also a Quillen equivalence we need to show that the derived unit and counit are weak equivalences. Since all objects of $\Pro(\cC)$ are cofibrant the same holds for $L_{K^\pi}\Pro(\cC)$. It follows that if $X \in \Pro(\cS_{\cofin})$ is fibrant then the actual counit
\[ \Psi_{K^\pi}(\Phi_{K^\pi}(X)) \lrar X \]
is equivalent to the derived counit. But this counit is an isomorphism by Remark~\ref{r:counit}. It is left to show that the derived unit is a weak equivalence.

Let $X \in L_{K^\pi}\Pro(\cS)$ be a cofibrant object and consider the map
\[ X \lrar \Phi_{K^\pi}((\Psi_{K^\pi}(X))^{\fib}) .\]
By Lemma~\ref{l:we} it is enough to check that the map
\[ \Psi_{K^\pi}(X) \lrar \Psi_{K^\pi}(\Phi_{K^\pi}((\Psi_{K^\pi}(X))^{\fib})) \]
is a weak equivalence. By Remark~\ref{r:counit} the latter is naturally isomorphic to $(\Psi_{K^\pi}(X))^{\fib}$ and the desired result follows.

\end{proof}

\begin{cor}
There is an equivalence of $\infty$-categories
\[ \hat{\cS}_\infty \simeq \Pro(\cS_{\cofin})_\infty \simeq \Pro(\cS^{\pi}_\infty) \]
\end{cor}

In~\cite{Mor96} Morel constructed a model structure on the category $\hat{\cS} \cong \Pro(\cS_{\cofin})$ in order to study \textbf{pro-$p$ homotopy theory}. Let us denote this model structure by $\Pro(\cS_{\cofin})_p$. The cofibrations in $\Pro(\cS_{\cofin})_p$ are the same as the cofibrations in Quick's model structure $\Pro(\cS_{\cofin})$, but the weak equivalences are more numerous. More precisely, the weak equivalences in $\Pro(\cS_{\cofin})_p$ are the maps which induce isomorphism on cohomology with $\mathbb{Z}/p\mathbb{Z}$ coefficients, whereas those of $\Pro(\cS_{\cofin})$ can be characterized as the maps which induce isomorphism on cohomology with coefficients in any finite local system. In particular, $\Pro(\cS_{\cofin})_p$ is a left Bousfield localization of $\Pro(\cS_{\cofin})$. This implies that the adjunction
\[\Psi:\Pro(\cS)\adj \Pro(\cS_{\cofin})_p:\Phi\]
is still a Quillen adjunction.

\begin{lem}\label{l:we-morel}
A map $f: X \lrar Y$ in $\Pro(\cS)$ is an equivalence in $L_{K^{p}}\Pro(\cS)$ if and only if $\Psi(f)$ is an equivalence in $\Pro(\cS_{\cofin})_p$.
\end{lem}

\begin{proof}
The proof is similar to the proof of Lemma~\ref{l:we}, using the fact that every $A \in K^p$ is fibrant in $\Pro(\cS_{\cofin})_p$ (see~\cite[Lemme 2]{Mor96}).
\end{proof}

\begin{thm}\label{t:isaksen-is-morel}
The Quillen adjunction
\[ \Psi: \Pro(\cS) \adj \Pro(\cS_{\cofin})_p : \Phi \]
descends to a Quillen equivalence
\[ \Psi_{K^p}: L_{K^p}\Pro(\cS) \adj \Pro(\cS_{\cofin})_p : \Phi_{K^p}. \]
\end{thm}

\begin{proof}
Since $L_{K^p}\Pro(\cS)$ and $\Pro(\cS_{\cofin})_p$ are left Bousfield localizations of $L_{K^\pi}\Pro(\cS)$ and $\Pro(\cS_{\cofin})$ respectively, it follows from Theorem~\ref{t:isaksen-is-quick} that $\Psi_{K^p}$ preserves cofibrations and from Lemma~\ref{l:we-morel} that $\Psi_{K^p}$ preserves trivial cofibrations. The rest of the proof is identical to the proof of Theorem~\ref{t:isaksen-is-quick} using Lemma~\ref{l:we-morel} instead of Lemma~\ref{l:we}.
\end{proof}

\begin{rem}
A slightly weaker form of this theorem is proved by Isaksen in~\cite[Theorem 8.7.]{Isa05}. Isaksen constructs a length two zig-zag of adjunctions between  $ L_{K^p}\Pro(\cS)$ and $\Pro(\cS_{\cofin})_p$ and the middle term of this zig-zag is not a model category but only a relative category.
\end{rem}

\end{document}